\documentclass[reqno,oneside,10pt,letterpaper]{amsart}
\usepackage{dwpreamble}
\usepackage{dwcommands}

\begin{document}

\title[Complex powers of the wave operator and the spectral action on Lorentzian scattering spaces]{\Large Complex powers  of the wave operator \\ {\Large and  the spectral action on Lorentzian scattering spaces}}

\author{}
\address{Sorbonne Université and Université Paris Cité, CNRS, IMJ-PRG, F-75005 Paris, France}
\email{nguyen-viet.dang@imj-prg.fr}
\author[]{\normalsize Nguyen Viet \textsc{Dang} \& Micha{\l} \textsc{Wrochna}}
\address{CY Cergy Paris Universit\'e, Laboratoire AGM, F-95302 Cergy-Pontoise, France}
\email{michal.wrochna@cyu.fr}
%\keywords{wave equation, spectral zeta functions, scattering theory, Quantum Field Theory on curved spacetimes, index theory}

\begin{abstract} We consider perturbations of Minkowski space as well as more general spacetimes on which the wave operator $\square_g$ is known to be essentially self-adjoint \cite{vasyessential}. We define complex powers  $(\square_g-i\varepsilon)^{-\cv}$ by functional calculus, and show that the trace density exists as a meromorphic function of $\cv$. We relate its poles to geometric quantities, in particular to the scalar curvature.  The results allow us to formulate a spectral action principle which serves as a simple Lorentzian model for the bosonic part of the Chamseddine--Connes action. Our proof combines microlocal resolvent estimates, including radial propagation estimates, with uniform estimates for the Hadamard parametrix. The arguments operate  in Lorentzian signature directly and do not rely on a transition from the Euclidean setting. 
\end{abstract}

\maketitle
%\tableofcontents

\section{Introduction}

\subsection{Introduction and main result} 

 The relationships between the geometry of  compact Riemannian manifolds and the spectral theory of elliptic operators have been a rich ground for discovery for decades, owing to powerful methods based on heat kernel and resolvent expansions, complex powers, residue traces, zeta functions and related notions \cite{Minakshisundaram1949,Ray1971,Quillen1985,BGV,Melrose1993,gilkey,shubin,paycha}. They have also profoundly influenced the world of relativistic physics,  relying on the presumption that a generalization to Lorentzian manifolds is possible \cite{Hawking1977,Connes1996,Chamseddine1997,Connes2008a,suijlekom}. This generalization was however found to be problematic on many levels. In particular, while it is possible to make sense of, e.g., formal heat kernel coefficients $\{a_i\}$ for the wave operator $\square_g$ on a Lorentzian manifold $(M,g)$ by writing  transport equations analogous to the Riemannian case, and  interpret them  in terms of a   Lorentzian  Hadamard parametrix \cite{Moretti1999,Bytsenko2003}, their relation to global objects defined by spectral theory is vastly unclear. 

On the other hand, it was recently found that if the spacetime $(M,g)$ has special symmetries or if instead, it is well-behaved at large distances,  then it is possible to interpret $\square_g$ as a self-adjoint operator in the sense of the canonical $L^2(M)$ space defined using the volume form of $g$. In fact, the essential self-adjointness of  $\square_g$ on \emph{static} spacetimes was proved by Derezi\'nski--Siemssen \cite{derezinski}, and on \emph{non-trapping Lorentzian scattering spaces} by Vasy \cite{vasyessential}; this was then generalized by Nakamura--Taira  to other differential operators of real principal type on long-range perturbations of Minkowski space \cite{nakamurataira}, cf.~\cite{cdv,kaminski,taira,Derezinski2019,Taira2020} for related recent works on self-adjointness of non-elliptic operators. As a consequence of self-adjointness, it is possible to define, e.g., $(\square_g-i\varepsilon)^{-\cv}$ for $\varepsilon>0$ and $\cv\in\cc$ abstractly by functional calculus. However, the relation to the local geometry is then an open question.

In the present work we demonstrate that globally defined complex powers of $\square_g$ are in fact related to spacetime geometry  in a way that parallels to a large extent the results known in the Riemannian case. We consider the setting of non-trapping Lorentzian scattering spaces introduced by Vasy in \cite{vasyessential},  which assumes the metric and the null geodesic flow to behave asymptotically in a certain way, see \sec{ss:bi}. The main feature is that this class contains perturbations of Minkowski space and other asymptotically Minkowski spacetimes (including the class considered in \cite{BVW,GHV}), and no particular  symmetry of $(M,g)$ nor real analyticity  is assumed.  We also make the assumption of \emph{global hyperbolicity} of $(M,g)$, which  arises naturally in, e.g., the solvability of the Cauchy problem.  Our main result is the following theorem.

\begin{theorem}[{cf.~Theorem \ref{mainmainprop} and Proposition \ref{mainmainprop}}]\label{mainthm} Assume $(M,g)$ is a globally hyperbolic, non-trapping Lorentzian scattering space (or that $(M,g)$ is an ultrastatic spacetime) and assume its  dimension $n$ is even. Then for all $\varepsilon>0$, the Schwartz kernel of $(\square_g-i \varepsilon)^{-\cv}$ has for $\Re\cv>\frac{n}{2}$ a well-defined on-diagonal restriction $(\square_g-i \varepsilon)^{-\cv}(x,x)$, which extends as a meromorphic function of $\cv\in\cc$ with poles at $\{\n2,$ $\n2-1$, $\n2-2$, $\dots$, ${1}\}$. Furthermore, 
\beq\label{mainresult}
\lim_{\varepsilon\to 0^+}\res_{\cv=\frac{n}{2}-1} \left(\square_g- i \varepsilon\right)^{-\cv}(x,x)=  \frac{R_g(x)}{{i}6(4\pi)^{\n2} \Gamma\big(\frac{n}{2}-1\big) }, 
\eeq
where $R_g(x)$ is the scalar curvature at $x\in M$. 
\end{theorem}

Theorem \ref{mainthm} can be seen as the Lorentzian version of a result attributed to Kastler \cite{kastler} and Kalau--Walze \cite{kalauwalze} in the Riemannian case (and announced previously by Connes as a consequence of classical theorems in elliptic theory due to Minakshisundaram--Pleijel \cite{Minakshisundaram1949}, Seeley \cite{seeley}, Wodzicki \cite{wodzicki} and other authors; see \cite[Thm.~1.148]{Connes2008a} for the heat kernel based argument and \cite[\S1.7]{gilkey} for an approach in the spirit of Atiyah--Bott--Patodi \cite{Atiyah1975}). This type of relationships has been used to justify definitions of curvature in non-commutative geometry, see Connes--Marcolli \cite[Def.~1.147]{Connes2008} and Connes--Moscovici \cite{Connes2014}.

The $\varepsilon$ regularizer in \eqref{mainresult} deals with the fact that in contrast to the compact Riemannian setting, $\square_g$ is not bounded from below. It is also responsible for the relationship with \emph{Feynman inverses}, see \sec{ss:proof}.

The importance of \eqref{mainresult} in physics stems from the fact that the r.h.s.~is proportional to the  \emph{Einstein--Hilbert Lagrangian}, and the variational principle $\delta_g R_g(x)=0$ is  equivalent to the Einstein equations for $g$. The l.h.s., on the other hand, refers to the spectral theory of the self-adjoint operator $\square_g$.  

\begin{remark} Our main case of interest are perturbations of Minkowski space (in arbitrary spatial and time directions) as well as more general Lorentzian scattering spaces, but the results are also valid for \emph{ultrastatic} spacetimes $(M,g)$ in the sense that $M=\rr\times Y $ and  $g=dt^2-h$ for some ($t$-independent) complete Riemannian manifold $(Y,h)$. In that case, essential self-adjointness follows from \cite{derezinski} and the proof of Theorem \ref{mainthm}  simplifies considerably, see Remark \ref{rem:static2}.
\end{remark}

The residues at the other poles can also be computed, and we show the following result.

\begin{theorem}[{cf.~Theorem \ref{thm:spectralaction}}]\label{mainthm2} For any Schwartz function $f$ with Fourier transform  supported in $\open{0,+\infty}$ and any $N\in\nn_{\geqslant 0}$, we have for $\varepsilon>0$ the large $\scalambda>0$ expansion
\beq\label{eq:main}
 f\big((\square_g+i\varepsilon)/\scalambda^2\big)(x,x) = \sum_{j=0}^N \scalambda^{n-2j} C_j(f) \, a_j(x) + {\pazocal{O}}(\varepsilon,\scalambda^{n-2N-1}),
\eeq
where each $C_j(f)$ depends only on $j\in\nn_{\geqslant0}$, the space-time dimension $n$ and $f$, and $a_j(x)$ are directly related to the Hadamard coefficients, in particular  $a_0(x)=(4\pi)^{-\n2}$, $C_0(f)=i^{-1}e^{\frac{in\pi}{4}} \int_0^\infty\widehat{f}(t)t^{\frac{n}{2}-1}dt$ and $a_1(x)={-}(4\pi)^{-\n2}\frac{1}{6} R_g(x)$, $C_1(f)=i^{-1}{e^{\frac{i(n-2)\pi}{4}} }\int_0^\infty\widehat{f}(t)t^{\frac{n}{2}-2}dt$.
\end{theorem}

 We refer to Theorem  \ref{thm:spectralaction} for a precise calculation of the first three terms of the asymptotic expansion of $f\big((\square_g+m^2+i\varepsilon)/\scalambda^2\big)(x,x)$ in terms of the regularizer $\varepsilon$ and also the mass $m$.
This formulation parallels as closely as possible the \emph{spectral action principle} established in the Riemannian case by Chamseddine--Connes \cite{Connes1996,Chamseddine1997}, which has since then become a milestone in high energy physics developments driven by the noncommutative geometry program, see e.g.~\cite{Chamseddine2007,Connes2008a,Varilly2009,suijlekom,Eckstein2018}. A particularity of Theorem \ref{mainthm2} is that in contrast to  results in the Riemannian setting, we do not allow for functions $f$ supported away from zero or in a half-line: intuitively, the reason is that the bottom of the spectrum plays  a r\^ole which cannot be harmlessly disregarded in the Lorentzian case.

\subsection{Structure of proof} \label{ss:proof}  The primary difficulty is unquestionably the non-ellipticity of $\square_g$, which makes known methods from the Riemannian setting inapplicable to our situation. We stress that except at the very final stage  (where we work \emph{locally} with quadratic forms on $\rr^n$ to compute numerical factors in the residues), our proof does not involve any kind of transition from Euclidean to  Lorentzian signature. Instead, we use  techniques from partial differential equations and microlocal analysis that emphasize the structure of the null geodesic flow on cotangent space and its asymptotic behavior, see \sec{ss:br} for bibliographical remarks.  

  Theorems \ref{mainthm} and \ref{mainthm2} rely on precise regularity estimates for the resolvent $(\square_g-z)^{-1}$. The key feature is that for $\Im z > 0$,   $(\square_g-z)^{-1}$ is  a \emph{Feynman inverse}, meaning that the singularities of the Schwartz kernel (characterized by its wavefront set) are the same as for Duistermaat--H\"ormander's Feynman parametrix \cite{DH} and Feynman propagators in Quantum Field Theory and related contexts \cite{Radzikowski1996,GHV}. In consequence, close to the diagonal in $M\times M$, the Schwartz kernel of $(\square_g-z)^{-1}$ can be approximated by the Feynman version of the \emph{Hadamard parametrix}, which is sufficiently explicit for the extraction of local geometrical quantities.  Complex powers $(\square_g-i\varepsilon)^{-\cv}$ are then expressed in terms of the resolvent as integrals over an infinite contour in the complex upper half-plane. To be useful, however, this requires  the estimates for the resolvent,  parametrix and  errors to be \emph{uniform in $z$}, with sufficient {decay} along the integration contour. This complicates the analysis of the Hadamard parametrix, as apart from difficulties due to  light-cone singularities, there is competition between regularity and decay in $\module{\Im z}$.  It is also worth stressing that it is not possible to eliminate any error term by solving a Cauchy problem for $\square_g-z$ because the associated retarded and advanced fundamental solutions badly behave as $\module{\Im z}\to+\infty$.
  
   With these issues in mind,  the proofs (in the Lorentzian scattering space case) are organized as follows:\smallskip

\ben
\item\label{istep1} Setting $P=\square_g$ or $P=\square_g+m^2$, we use radial estimates in weighted scattering Sobolev spaces (due in the present context to Vasy \cite{vasyessential} and generalizing results by Melrose \cite{melrosered}) to derive mapping properties of the resolvent $(P-z)^{-1}$, uniformly in $z$. By integrating on a contour $\gamma_\varepsilon$ in the  upper half-plane (see Figure \ref{fig:contour} in \sec{ss:restocp}) we deduce local mapping properties of $(P-i\varepsilon)^{-\cv}$ in Sobolev spaces. \smallskip
\item\label{istep2} In \sec{ss:firstparametrix}, for $\Im z\geqslant 0$ we construct a $z$-dependent parametrix of $P-z$ which is the sum of two independent parts, each with  singularities propagating in only one of the two components of the characteristic set $\Sigma$. We show that the parametrix has \emph{Feynman wavefront set}  uniformly along the contour $\gamma_\varepsilon$. This step uses a time-dependent factorization of $P-z$  in Shubin's parameter-dependent pseudo-differential calculus \cite{shubin} and the hyperbolicity of $P$.  \smallskip
\item\label{istep3} In \sec{ss:wavefront} we relate $(P-z)^{-1}$ with the parametrix from step \eqref{istep2}. The argument first emphasizes common behaviour at the radial sets, and then uses radial estimates and propagation of singularities to obtain a global result. The main conclusion is that $(P-z)^{-1}$ has Feynman wavefront set uniformly in $z$ along $\gamma_\varepsilon$.\smallskip
\item\label{istep4} In \secs{sec:elementary}{s:hadamardformal} we construct a $z$-dependent version $H_N(z,.)$ of the Hadamard parametrix of $P-z$, and show in \sec{s:hadamardformal} that it also has Feynman wavefront set uniformly in $z$. We prove regularity estimates for $H_N(z,.)$ and the remainders, with control of the decay for large $\module{z}$ and the behaviour near the real axis. An important prerequisite  are H\"older--Zygmund and microlocal estimates shown in  \sec{sec:elementary} for an elementary family of distributions $\Fse{z}$ on $\rr^n$ which serves as the building block of the parametrix in normal coordinates.
\smallskip
 \item\label{istep5}  For $\Im z> 0$, we relate the resolvent $(P-z)^{-1}$ to  the uniform Hadamard parametrix $H_N(z,.)$ using the estimates from step \eqref{istep4} and the Feynman form of the wavefront set proved in step \eqref{istep3} by a composition argument.  The local analysis of the Schwartz kernel of $(P-i\varepsilon)^{-\cv}$ and other functions of $P$ is  reduced in this way to contour integrals involving $H_N(z,.)$. \smallskip
 \item\label{istep6}  The meromorphic continuation of $\alpha\mapsto (P-i\varepsilon)^{-\cv}$ and its poles are computed on the level of contour integrals of $H_N(z,.)$. To compute the  residues we use a homological argument which can be interpreted as a local Wick rotation of quadratic forms.  Theorem \ref{mainthm2} is concluded from the full version of Theorem \ref{mainthm} by a Mellin transform argument. \smallskip
\een

Various auxiliary proofs are collected in the appendices.

We stress that although the occurence of local geometric quantities in the Hadamard parametrix is a well-known phenomenon, the relationship with globally defined functions of $P$ proved in steps \eqref{istep1}--\eqref{istep6} is new.

\begin{remark}\label{rem:static2} The case of $(M,g)$ ultrastatic is simpler because one can then give a quasi-explicit formula for $(P-z)^{-1}$ in terms of the  Laplace--Beltrami operator $\Delta_h$ on the Cauchy surface. The formula implies that $(P-z)^{-1}$ is already of the form of the parametrix in step \eqref{istep2}, so step \eqref{istep3} is no longer needed, and resolvent estimates can be derived directly, see  \sec{app:static}. From that point on, steps \eqref{istep4}--\eqref{istep6} apply verbatim.
\end{remark}

\subsection{Bibliographical remarks}\label{ss:br} The construction of complex powers of elliptic operators is due to Seeley \cite{seeley} in the case of classical pseudo-differential operators on compact manifolds, and was extended to  various other elliptic settings, among others in works by Rempel--Schulze \cite{Rempel1983},  Guillemin \cite{Guillemin1985}, Grubb \cite{Grubb1986}, Schrohe \cite{Schrohe1986,Schrohe1988}, Loya \cite{Loya2001,Loya2003}, Coriasco--Schrohe--Seiler \cite{Coriasco2003} and Ammann--Lau\-ter--Nis\-tor--Vasy \cite{ALNV}, cf.~recent work by Hintz \cite{hintz20}. Various results in the spirit of the Kastler--Kalau--Walze identity were obtained, e.g., by Ponge \cite{Ponge2006,Ponge2007,Ponge2008} (in particular, \cite{Ponge2008} discusses lower dimensional geometric invariants) and Battisti--Coriasco \cite{Battisti2011}.  The residues of the spectral zeta function have a natural interpretation in terms of the Guillemin--Wodzicki residue  \cite{wodzicki,Guillemin1985}  (cf.~Connes--Moscovici \cite{Connes1995}, Lesch \cite{Lesch}, Lesch--Pflaum \cite{Lesch2000}, Paycha \cite{paycha2,paycha}, Maeda--Manchon--Paycha \cite{Maeda2005}), see \cite{Dang2021} for a generalization to the Lorentzian case  considered here.

 Complex powers of non-elliptic first order pseudo-differential operators were obtained as paired Lagrangian distributions by Antoniano--Uhlmann \cite{Antoniano1985}, see also Greenleaf--Uhlmann \cite[\S 3]{Greenleaf1990}.   Using the calculus of paired Lagrangian distributions, complex powers of the wave operator corresponding to a retarded or advanced problem were  constructed by Joshi \cite{joshi} in the case of time-independent coefficients.  Enciso--Gonz\'alez--Vergara \cite{Enciso2017} later showed    that a particular fractional power coincides with the Dirichlet-to-Neumann map on static anti-de Sitter spacetimes. We remark that our complex powers are different from Joshi's, as the former are associated to a self-adjoint operator and are related to a Feynman problem rather than to a retarded or advanced one. However, it can be conjectured that they are paired Lagrangian distributions as well.

The approach to Lorentzian complex powers  in the present paper builds on the non-elliptic Fredholm theory introduced by Vasy in \cite{Vasy2013}, originally in the context of the retarded and advanced problem on (Kerr--)de Sitter spaces, and further developed in a series of works  tailored to the study of wave and Einstein equations (see e.g.~\cite{BVW,Hintz2015,Hintz2016,Hintz2017}),  culminating in the resolution of the Kerr--de Sitter stability conjecture by Hintz--Vasy \cite{Hintz2018a} and  the proof of linear stability of Kerr black holes by H\"afner--Hintz--Vasy \cite{Hafner2020}. The global approach to the Feynman problem for the wave equation on a class of Lorentzian scattering spaces was pioneered by Gell-Redman--Haber--Vasy \cite{GHV} (including a non-linear version), cf.~Baskin--Vasy--Wunsch \cite{BVW} for previous work on the retarded and advanced problem in that setting. The construction also applies to the Klein--Gordon operator on de Sitter spaces, and in both settings, its positivity and microlocal properties were  studied  by Vasy \cite{Vasy2017b} and Vasy--Wrochna \cite{vasywrochna}. The Feynman invertibility of the  Klein--Gordon operator $\square_g+m^2$ with $m>0$ on asymptotically Minkowski spacetimes was shown  by  G\'erard--Wrochna \cite{GWfeynman,Gerard2019b} (cf.~\cite{Gerard2020} for a brief account) using an approximate diagonalization of the evolution, related to the parametrix in \sec{ss:firstparametrix} (though the focus here is on the behaviour in $z$). In the already mentioned work of Vasy \cite{vasyessential} on  essential self-adjointness of $\square_g$, the resolvent is constructed in terms of a Feynman problem  which coincides with the Gérard--Wrochna definition by a result of Taira \cite{Taira2020a}.     Vasy \cite{vasyessential}  also shows a limiting absorption principle for $\square_g+m^2$, followed by an improvement by Taira \cite{Taira2020a}, cf.~the earlier work of Derezi\'nski--Siemssen \cite{derezinski} for the limit{ing} absorption principle in the static case (possibly with electromagnetic potentials). 

Related developments connecting the global theory of hyperbolic operators with space-time geometry have included a Lorentzian Atiyah--Patodi--Singer index theorem due to  B\"ar--Strohmaier \cite{Bar2019}, see also Braverman \cite{Braverman2020} for a spatially non-compact generalization,  and very recently  a local version was shown by B\"ar--Strohmaier \cite{Baer2020}.  Furthermore,  Strohmaier--Zelditch  proved a  Gutzwiller--Duistermaat--Guillemin trace formula and a Weyl law  for  time-like Killing vector fields on stationary space-times \cite{Strohmaier2020b,Strohmaier2020,Strohmaier2020a}, which in particular provides a spectral-theoretical way of recovering the scalar curvature and thus a spectral action in the stationary case. It is worth emphasizing that Feynman inverses appear naturally in all these developments (and in \cite{Baer2020} the relationship with the Hadamard parametrix is used, see below). 

We remark that non-elliptic Fredholm problems and radial estimates have arisen in many  contexts outside of relativistic settings, see e.g.~\cite{Dyatlov2016,Dyatlov2019,hassell}. In particular, we emphasize similarities with the work of Dyatlov--Zworski \cite{Dyatlov2016} on Anosov flows, which proves the meromorphic continuation of the Ruelle zeta function using microlocal resolvent estimates;  we expect  that semi-classical methods could provide useful alternatives to the arguments in  \sec{ss:wavefront}.

The Hadamard parametrix for  Laplace--Beltrami operators on pseudo-Riemannian manifolds is a classical tool in analysis, see e.g.~\cite{HormanderIII,soggeHangzhou,Zelditch2017} for textbook accounts focused on the Riemannian or Lorentzian time-independent case. It plays a fundamental rôle in Quantum Field Theory on curved spacetimes, where it is used to substract singularities from $N$-point functions to get well-defined non-linear  quantities \cite{DeWitt1975,Fulling1989,Kay1991,Radzikowski1996,Moretti1999,Brunetti2000,Hollands2001}. In particular, Radzikowski \cite{Radzikowski1996} proved the relationship between the bi-solution Hadamard parametrix and the Feynman parametrix of Duistermaat--H\"ormander. In the present work we work directly with the ($z$-dependent)  Feynman version of the Hadamard parametrix. Its analogue for fixed $z$ was constructed by Zelditch \cite{StevenZelditch2012} in the ultra-static case, and by Lewandowski \cite{Lewandowski2020} and Bär--Strohmaier \cite{Baer2020} in the general case using a family of distributions with distinguished wavefront set (the former construction also  gives a unified treatment of even and odd dimensions). The Hadamard parametrix is also useful in spectral theory, and was applied e.g.~by Sogge \cite{Sogge1988}, Dos Santos Ferreira--Kenig--Salo \cite{Ferreira2014}  and Bourgain--Shao--Sogge--Yao \cite{Bourgain2015} in the context of $L^p$ resolvent estimates on compact Riemannian manifolds  (including estimates uniform in the spectral parameter $z$), and by Zelditch \cite{StevenZelditch2012} in the problem of analytic continuation of eigenfunctions.

Finally, we mention only very non-exhaustively works in noncommutative geometry aimed at establishing a Lorentzian theory \cite{Moretti2003,Paschke2004,Strohmaier2006,Franco2013,DAndrea2016,VandenDungen2016,VanDenDungen2018,Devastato2018,Bizi2018,Besnard2018}. In contrast to the problem considered here, their focus is mostly on the formalism of spectral triples or on distance formul\ae. In the last few years this has included progress on spectral actions by D'Andrea--Kurkov--Lizzi \cite{DAndrea2016}, Devastato--Farnsworth--Lizzi--Martinetti \cite{Devastato2018} and Martinetti--Singh \cite{martinetti}, which involves however a transition from Euclidean signature and relies on special symmetries or analyticity, very differently from the present paper's result.

\subsection{Remarks on assumptions, outlook} The essential self-adjointness in \cite{vasyessential} and the results of the present paper extend in a straightforward way to the Hermitian bundle setting provided that the principal symbol of $P$ is a scalar wave operator and that formal self-adjointness of $P$ holds true for a \emph{positive} scalar product.   
We further comment on this  in Remark \ref{rem:vector1}.

The assumptions on spacetime geometry in   Theorems \ref{mainthm} and \ref{mainthm2}  are not expected  to be sharp. In fact, only steps \eqref{istep1} and \eqref{istep3} of the proofs use the hypothesis that $(M,g)$ is a Lorentzian scattering space, and one could try to adapt the arguments depending on estimates available in a given class of spacetimes. For instance, in view of the estimates in \cite{Vasy2013}, a natural candidate could be the class of asymptotically de Sitter spacetimes.  The essential self-adjointness of $\square_g$ is also conjectured to be true for asymptotically static spacetimes, see Derezi\'nski--Siemssen \cite[\S5.8, \S8.6]{Derezinski2020}, and it is therefore natural to ask whether \eqref{istep1} and \eqref{istep3} remain valid in that general setting.

A mathematically  delicate point is the limit $\varepsilon\to 0^+$ of the Schwartz kernel of $(P-i\varepsilon)^{-\cv}$ and of other functions of $P-i\varepsilon$. Namely, observe that in Theorems \ref{mainthm} and \ref{mainthm2} we \emph{first} compute a residue or an expansion and \emph{then} take the $\varepsilon\to 0^+$ limit, but one could ask whether the order of these operations can be reversed. We give an affirmative answer in the setting  of Theorem \ref{mainthm2} if $\square_g$ is replaced by $P=\square_g+m^2$ with $m>0$ (this then produces extra terms in the expansion which vanish in the limit $m\to 0^+$, see  Theorem \ref{thm:spectralaction}), this requires however  stronger assumptions including \emph{non-trapping at energy} $m^2$, see \sec{ss:bi} and \sec{ss:realaxis1}.   For the sake of illustration,  we also prove in \sec{a:LAPfeynman}   a limiting absorption principle for $(P-i\varepsilon)^{-\cv}$ in the case of space-compact ultrastatic spacetimes, with  similar conclusions on the possibility of taking the $\varepsilon\to 0^+$ limit before computing  residues.
  
  The study of the  $\varepsilon\to 0^+$ limit with $m=0$ rather than $m>0$ requires a different approach, based for instance on  recent work by Vasy \cite{Vasy2019}, cf.~Bouclet--Burq \cite{Bouclet2018}.

Finally, we do not consider here functions of Dirac operators nor generalizations needed to derive a spectral action principle for the whole Standard Model in Lorentzian signature, we expect this however to be a fruitful topic of research in the near future.

\subsection*{Acknowledgments} The authors are particularly grateful to Fabien Besnard, Christian Brouder, Jan Derezi\'nski, Peter Hintz, Rapha\"el Ponge, Kouichi Taira, Gunther Uhlmann and Andr\'as Vasy for helpful discussions. Support from the grant ANR-16-CE40-0012-01 is gratefully acknowledged. N.V.D.~acknowledges the support of the Institut Universitaire de France. We thank the MSRI in Berkeley and the Mittag--Leffler Institute in Djursholm for their kind hospitality during thematic programs and workshops in 2019--20.  

\section{Complex powers on Lorentzian scattering spaces}\label{sec:complexpowers}

\subsection{Klein--Gordon operator}\label{ss:kg} Let $(M,g)$ be a Lorentzian manifold. We use the  convention $(+,-,\dots,-)$ for the signature of  $g$. We denote by $L^2(M)$ the canonical $L^2$ space associated to the volume density $d\vol_g$ of $g$, i.e.~the $L^2(M)$ norm  is
$$
\| u \|= \bigg(\int_M \left| u(x) \right|^2d\vol_g\bigg)^{\12}.
$$

Let $P=\square_g+m^2$ be the \emph{wave} or \emph{Klein--Gordon operator}, i.e.~$\square_g=\frac{1}{\sqrt{|g|}}\p_\mu(\sqrt{|g|} g^{\mu\nu}\p_\nu)$ is the Laplace--Beltrami operator in Lorentzian signature, $|g|=\left|\det g\right|$ and $m^2 \geqslant 0$.

\subsection{Lorentzian scattering spaces} We will need to make assumptions on the asymptotic structure of $(M,g)$ at spacetime infinity. To that end it is convenient to assume that $M$ is the interior of a \emph{compact manifold with boundary} $\M$. 

We use the notation $\cf(\M)$ for the space of smooth function on $\M$, meant in the usual sense of smooth extensibility across the boundary (denoted in what follows by $\p \M$).

 Let $\rho$ be a boundary-defining function of $\p \M$, i.e.~a function $\rho\in\cf(\M)$ such that 
$\rho>0$ on $M$, $\p \M=\{\rho=0\}$, and $d\rho\neq 0$ on $\p \M$. Recall that by the collar neighborhood theorem, there exists $W\supseteq\pM$, $\epsilon>0$ and a diffeomorphism $\phi:\clopen{0,\epsilon}\times \pM \to W$ such that $\rho\circ \phi$ agrees with the projection to the first component of $\clopen{0,\epsilon}\times \pM$. We  use notation proper of $\clopen{0,\epsilon}\times \pM$ and drop $\phi$ in the notation when working close to the boundary, i.e.~in the collar neighborhood $W$. In this sense we can find local coordinates of the form $(\rho,y_1,\dots,y_{n-1})$, where $(y_1,\dots,y_{n-1})$ are local coordinates on $\p \M$.

We use the framework of  Melrose's $\sc$-geometry \cite{melrosered},  mostly following   the presentation in \cite{vasygrenoble,vasyessential}. Let $\be T\M$ be the  \emph{scattering tangent bundle}  (or in short, ${\rm sc}$\emph{-tangent bundle}) of $\M$. Recall that $\be T\M$ can be defined as the unique vector bundle over $\M$ such that {all} of its smooth sections $V\in \cf(\M;\be T\M)$ {are} locally of the form
\beq\label{sctvf}
V=V_0(\rho,y) \rho^2\p_\rho + \sum_{i=1}^{n-1} V_i(\rho,y)\rho \p_{y_j}, \ V_0,V_i \in \cf(\overline{U}), \ i=1,\dots,n-1
\eeq
on {a} chart neighborhood $\overline{U}$ with local coordinates of the form $(\rho,y_1,\dots,y_{n-1})$. Away from the boundary, $\be T\M$ is defined  in the exact same way as the tangent bundle $TM$, and there is indeed a canonical isomorphism $\be T_M\M \to TM$. 

The ${\rm sc}$\emph{-cotangent bundle} $\be T^*\M$ is by definition the dual bundle of $\be T\M$. Thus, in local coordinates $(\rho,y_1,\dots,y_{n-1})$, the smooth sections of $\be T^*\M$ are $\cf(\M)$-generated by $(\rho^{-2}d\rho, \rho^{-1} d y_1$, $\dots, \rho^{-1} d y_{n-1})$. Again,  over the interior there is a canonical isomorphism
\beq\label{eq:isom}
\be T^*_M\M \to T^*M.
\eeq
Next, an \emph{$\sc$-metric} is by definition a  non-degenerate smooth section of the fiberwise symmetrized tensor product bundle $\be T^*\M \otimes_{\rm s} \be T^*\M$. 

\begin{definition} $(\M,\xoverline{g})$ is a \emph{Lorentzian scattering space} (or in short, \emph{Lorentzian ${\rm sc}$-space})  if $\xoverline{g}\in \cf(\M; \be T^*\M \otimes_{\rm s} \be T^*\M )$ is of Lorentzian signature. 
\end{definition}

\begin{example}\label{ex:mi} The standard example is  $M=\rr^n$, with $\M=\overline{\rr^n}$  the \emph{radial compactification of $\rr^n$}. Recall that $\overline{\rr^n}$ is defined as the quotient of $\rr^n\sqcup\big(\clopen{0,1}_\rho \times \ss^{n-1}_y \big)$ by the  relation which identifies any non-zero $x\in\rr^n$ with the point $(\rho,y)$, where $\rho=r^{-1}$ and $(r,y)$ are the polar coordinates of $x$. The smooth structure near $\{\rho=0\}$ is the obvious one in $(\rho,y)$ coordinates.  Observe that the vector field $\p_r=-\rho^2\p_\rho$ is of the form \eqref{sctvf}. More generally, switching now to standard coordinates $(x_0,\dots,x_{n-1})$ on $\rr^n$, the frame $(\p_{x_0},\dots,\p_{x_{n-1}})$ smoothly extends to $\be T^*\overline{\rr^n}$, and any $V\in \cf( \overline{\rr^n};\be T\,\overline{\rr^n})$ is in the $\cf(\overline{\rr^n})$-span of $(\p_{x_0},\dots,\p_{x_{n-1}})$, i.e.~the coefficients smoothly extend across $\{\rho=0\}$ in $(\rho,y)$ coordinates  on top of being smooth in $\rr^n$. Similarly, any $\xoverline{g}\in \cf(\overline{\rr^n}; \be T^*\overline{\rr^n}\otimes_{\rm s} \be T^*\overline{\rr^n} )$ is in the $\cf(\overline{\rr^n})$-span of $dx^\mu\otimes_{\rm s} dx^\nu$ for $\mu,\nu=0,\dots,n-1$. In particular, the Minkowski metric $\eta=dx_0^2-(dx_1^2+\cdots+dx_{n-1}^2)$ on $\rr^n$ extends to an $\rm sc$-metric on $\overline{\rr^n}$ and in this sense Minkowski space is a Lorentzian $\rm sc$-space.
\end{example}

We will assume that $g$ (the Lorentzian metric on the boundaryless manifold $M$) extends to an $\rm sc$-metric on $\M$, and so $(\M,\xoverline{g})$ is a Lorentzian $\rm sc$-space. The volume density of $(M,g)$, $d\vol_g$, extends then to an ${\rm sc}$-density on $\M$, meaning that in local coordinates $(\rho,y)$ it is of the form $\mu(\rho,y)\left|\rho^{-2} d\rho\, \rho^{-n+1}d y  \right|$ with $\mu\in \cf(\M)$.

\subsection{Bicharacteristics and Hamilton flow}\label{ss:bi}
When discussing microlocalisation it is useful to compactify the fibers of $\be T^*\M$. The base manifold $\M$ having  a boundary already, the {fiberwise radial compactification} of $\be T^*\M$ yields a manifold with corners {(see \cite[\S6.4]{melrosered} for details)}, which we will denote by $\co$.

As a manifold with corners, $\co$ has two boundary hypersurfaces: the first  one is \emph{base infinity} or \emph{spacetime infinity}, which we denote by 
$
\basinf
$ 
(instead of using the more pedantic, but heavier notation $\overline{\be T}{}^*_{\p\overline{M}}\M$) and the other one  is \emph{fiber infinity}, which we denote   by
$
\fibinf.
$
We stress that despite what the notation could suggest, these two boundary hypersurfaces \emph{do} intersect at the corner $\corner\neq\emptyset$, and we have of course
$$
\p\co=\basinf\cup\fibinf.
$$
Let $\bra \xi\ket^{-1}$ be the formal notation for a boundary-defining function of fiber infinity.   For $z\in\cc$, the \emph{principal symbol} of $\square_g-z$ in the sense of the $\sc$-calculus is the function $p_z$ on $\p\co$ given by:
\beq\label{eqpri}
 p_z(x,\xi)=\begin{cases}
{-} | \xi |^{-2}  (\xi \cdot g^{-1} \xi) \ \mbox{ on } \ \fibinf,\\
 \bra \xi \ket^{-2}  ({-}\xi \cdot g^{-1} \xi  - z) \ \mbox{ on } \ \basinf.
 \end{cases}
\eeq
This is well-defined at $\fibinf$ thanks to the  $| \xi |^{-2}$  factor that compensates for the degree $2$ homogeneity of $\xi \cdot g^{-1} \xi$ in $\xi$.  This is also well-defined at $\basinf$ as a consequence of  the assumption that $g$ extends to an $\rm sc$-metric. Furthermore, the definition is consistent at the corner.

The \emph{characteristic set} of $\square_g-z$, denoted by $\Sigma_z$, is defined as the closure of $p^{-1}_z(\{0\})$ in $\p\co$. Note that  $\Sigma_z\subset\fibinf$ unless $z$ is real. Furthermore,  $\Sigma_z\cap\fibinf=\Sigma_0\cap\fibinf$ is always non-empty but does not depend on $z$. This is why various hypotheses can be simply written in terms of $\Sigma_\sigma$ with $\sigma\in\rr$.

The \emph{Hamilton vector field} of $p_0$ on $\be T^*\M$, denoted by $H_{p_0}$, is the  extension of the usual Hamilton vector field defined in the interior, i.e.~the standard definition on $T^*M$ induces a vector field on $\be T^*_M\M$ via the isomorphism \eqref{eq:isom}, and this then extends to a vector field $H_{p_0}$ on $\be T^*\M$. {Similarly, $\be T^*\M$  has a symplectic and contact structure inherited from $T^*M$ by extension, see e.g.~\cite[\S2]{Melrose1996}.}  In local coordinates  on  $\be T^*\M$ of the form $(\rho, y, \varrho, \eta)$, where  $(\varrho,\eta)$ are the dual coordinates of $(\rho,y)$, $H_{p_0}$ is given by
$$
H_{p_0} = \rho \bigg( (\p_\varrho p)  (\rho \p_{\rho} + \eta \cdot\p_\eta )  -  ( \rho\p_\rho +  \eta \cdot\p_\eta) p \p_\varrho + \sum_{i=1}^{n-1}\big( (\p_{\eta_i} p) \p_{y_i} - (\p_{y_i} p) \p_{\eta_i}  \big) \bigg).   
$$
The \emph{rescaled Hamilton vector field} $\xoverline{H}_{p_0}\defeq \bra \xi \ket^{-1} \rho^{-1} H_{p_0}$ extends to a smooth vector field on $\co$ which is tangent to $\p\co$. We call its flow on $\p \co$ the \emph{Hamilton flow}, and for $\sigma\in\rr$, the \emph{bicharacteristics} are the integral curves of the rescaled Hamilton vector field within $\Sigma_\sigma$. 

 \begin{definition}\label{maindef} For $\sigma\in\rr$ we  say that $(M,g)$ is \emph{non-trapping at energy $\sigma$} if the following conditions are satisfied:
 \ben
 \item There are two submanifolds $L_-\subset \basinf$ and $L_+\subset \basinf$, each transversal to $\corner$, which are sources, resp.~sinks for the Hamiltonian flow in $\Sigma_\sigma$. More precisely, this means that within $\Sigma_\sigma$,
 \ben 
 \item[a)] $d{p}_0\neq 0$ on $L_\pm$ and $\xoverline{H}_{p_0}$ is tangent to $L_\pm$,
 \item[b)]\label{sscond1} there exists a quadratic defining function $\rho_\pm$ of $L_\pm$ and a smooth function $\beta_\pm>0$ satisfying
$$
 \mp\xoverline{H}_{p_0} \rho_\pm= \beta_\pm \rho_\pm + s_\pm + r_\pm  
$$
 for some smooth $s_\pm,r_\pm$ such that $s_\pm>0$  on $L_\pm$ and $r_\pm$ vanishes cubically at $L_\pm$,
 \item[c)]\label{sscond2} there exists $\beta_{0,\pm}\in\cf(\co)$ such that $\beta_{0,\pm}>0$ on $L_\pm$ and
 $\mp\xoverline{H}_{p_0} \rho = \beta_{0,\pm} \rho$.
 \een
 \item[(2)]\label{ntpop}  Within $\Sigma_\sigma$, each bicharacteristic goes either from  $L_+$ to $L_-$, or from $L_-$ to $L_+$, or stay{s} within $L_+$ or $L_-$.
 \een
We  say that $(M,g)$ is \emph{non-trapping} if (1) and (2) hold true with $\Sigma_0\cap \fibinf$ instead of $\Sigma_\sigma$.
 \end{definition}

 We will simply refer to $L_-$ as \emph{sources}  and to $L_+$ as \emph{sinks}. In (1) of Definition \ref{maindef}, by saying that $\rho_\pm\in \cf(\co)$ is a \emph{quadratic defining function of $L_\pm$} we mean that $\rho_{\pm}=\sum_i \rho_{\pm,i}^2$ for finitely many  $\rho_{\pm,i}$ such that  $L_\pm=\cap_i\{ \rho_{\pm,i} =0\}$ within $\Sigma_\sigma\cap\basinf$, and the differentials $d \rho_{\pm,i}$ are  linearly independent  on $L_\pm\cap \Sigma_\sigma$.  A more detailed discussion of  conditions b)--c) can be found in \cite[\S 5.4.7]{vasygrenoble}.
 
 \begin{example} A (non-exhaustive) class of examples is provided by the non-trapping Lorentzian scattering metrics  introduced in \cite{BVW} and further studied in the context of the Feynman problem in \cite{GHV}.   {Namely, one assumes the existence of} a function $v\in \cf(M)$ such that for all $V\in \cf(\M;\be T\M)$ the sign of $g(V,V)$ and $v$ are the same at $\p M=\{\rho=0\}$. {Furthermore, near $\{v=\rho=0\}$,} the $\sc$-metric $g$ is {assumed to be} of the form
  $$
  g = v \frac{d\rho^2}{\rho^4}-\bigg(  \frac{d\rho}{\rho^2} \otimes_{\rm s} \frac{\omega}{\rho}\bigg)-\frac{\tilde g}{\rho^2},
 $$ 
 where $\omega$ is a smooth $1$-form such that $\omega=dv+\cO(v)+\cO(\rho)$, and the restriction of $\tilde g\in  \cf(\M; T^*\M \otimes_{\rm s}  T^*\M)$ to the joint annihilator of $d\rho, dv$ is positive. As discussed in \cite[\S{3.6}]{BVW}, this implies the existence of sources/sinks at 
 $
 \{ \rho=v=0,\, \varrho=\gamma=0, \, \mp\gamma>0\}
 $ in coordinates $(\rho,v,w,\varrho,\beta,\gamma)\in \be T^*\M$. One then needs to ensure that the non-trapping property \eqref{ntpop} of Definition \ref{maindef} holds true in $\Sigma_0\cap \fibinf$, see \cite[\S{3.2}]{BVW}. Minkowski space is a special case, see \cite[\S{3.1}]{BVW}, and we also note that in practice it is possible to consider perturbations that do not have the structure of sinks and sources, but for which the propagation estimates used in the sequel remain valid nevertheless. We also refer to \cite[\S2]{vasyessential} for remarks on the assumption of non-trapping at $\sigma\neq 0$.
  \end{example}
 
In \cite{vasyessential}, Vasy proves the following theorem, cf.~the work of Nakamura--Taira \cite{nakamurataira} for the case of real principal type operators of arbitrary orders on $\rr^n$ under a similar non-trapping condition.
 
 \begin{theorem}[{\cite[Thm.~1]{vasyessential}}] Assume $(M,g)$ is non-trapping. Then $P$ acting on $\ccf$ is essentially self-adjoint in $L^2(M)$. 
 \end{theorem}
 
 As a consequence, if we denote in the same way the closure of $P$ acting on   $\ccf$,  functions of $P$ can be defined using the functional calculus for self-adjoint operators. We are particularly interested in Schwartz kernels of functions of $P$, and therefore we need to know more precise mapping properties of the resolvent,  also basing on the results from \cite{vasyessential}.

\subsection{Sobolev spaces}\label{ss:jj} If $s\in\zz_{\geqslant 0}$, then the \emph{$\rm sc$-Sobolev space of order $s$} is by definition the space:
$$
\Hsc{s,0}=\big\{ u\in L^2(M) \ \vert \ \forall k\leqslant s \mbox{ and } V_1,\dots,V_k \in \cf(\M;\be T\M), \ V_1\dots V_k u \in L^2(M)\big\}.
$$
The definition of  $\Hsc{s,0}$ and of its norm $\| \cdot\|_{s,0}$  for  arbitrary $s\in\rr$ is most efficiently formulated with the help of $\sc$-pseudo-differential operators, see \sec{ss:scattering}. The \emph{weighted Sobolev spaces} are defined for non-zero $\ell\in \rr$ by 
$$
\Hsc{s,\ell}=\rho^{\ell} \Hsc{s,0},
 $$
with norm $\| u \|_{s,\ell}=\| \rho^{-\ell}  u \|_{s,0}$, where $\rho$ is as before a boundary-defining function of $\p\M$. Thus, higher $s$ means more regularity, and higher $\ell$ means more decay at spacetime infinity, i.e.~at $\p \M$. In the special case of Minkowski space modelled on $\overline{\rr^n}$, the space $\Hsc{s,0}$ coincides with the usual Sobolev space $H^s(\rr^n)$, and if we choose as boundary defining function $\rho=(1+|x|^2)^{-\12}\eqdef\bra x\ket^{-1}$ then $\Hsc{s,\ell}$ coincides with the weighted Sobolev space $\bra x\ket^{-\ell} H^{s}(\rr^n)$. 

The definition of $\Hsc{s,\ell}$ can be usefully generalized to weight orders $\ell$ that vary in phase space, i.e.~to $\ell\in\cf(\co)$ rather than just $\ell\in \rr$, see \sec{ss:scattering}.  We will also use the Fr\'echet spaces
$$
 \Hsc{\infty,\ell}\defeq  \textstyle\bigcap_{s\geqslant 0} \Hsc{s,\ell}, \quad  \Hsc{s,\infty}\defeq  \textstyle\bigcap_{\ell\geqslant 0} \Hsc{s,\ell}.
$$

We stress that  unless $s=\ell=0$, the definition of $\Hsc{s,\ell}$  refers to the manifold with boundary $\M$. As a rule, we do not necessarily emphasize the dependence on $\M$ or $\xoverline{g}$ in the notation if there is an ``$\rm sc$'' subscript, which indicates  the dependence on the scattering structure already.  Apart from  the spaces with an ``$\rm sc$'' subscript, we use standard notation. For instance, $\cf_{\rm c}(M)$, $\cf(M)$, $H^s_{\rm c}(M)$ and $H^s_{\rm loc}(M)$ are the standard  spaces on  the \emph{boundaryless} manifold $M$ (in contrast to the space of smooth functions $\cf(\M)$ on the manifold with boundary $\M$), and similarly for the space of  distributions $\cD'(M)$ and of compactly supported distributions $\cE'(M)$ on $M$. 

Note that for all $s\in\rr$  and $\ell\in\cf(\co)$ we have the continuous inclusions
$$
 H^s_{\rm c}(M)\subset \Hsc{s,\ell}\subset H^s_{\rm loc}(M)\subset \cD'(M).
$$

\subsection{Estimates for imaginary spectral parameter} We now  consider the operator $P-z$ for $z\in\cc$, focusing first on the case $\Im z>0$.

 For $\Im z\neq 0$, $P-z\in \Psi^{2,0}_{\sc}(M)$ is microlocally {elliptic} in the sense of the $\sc$-calculus except at fiber infinity $\fibinf$; propagation estimates take place inside of $\Sigma_0\cap\fibinf$.
   
For $\ell\in \cf(\co)$ we set $\ell_\pm=\ell|_{L_\pm}$.  We will say that $\ell$ is \emph{monotone in $\Sigma_z$} if it is monotone along the  Hamiltonian flow restricted to $\Sigma_z$. In various estimates, $\sS,\sL\in\rr$ will always be sufficiently negative numbers, which can be taken arbitrarily negative. 

For some arbitrary $c>0$ let $Z=\{ \Im z \geqslant c \module{\Re z}\}$, and let $\delta>0$.

\begin{proposition}[{\cite[Prop.~2]{vasyessential}}]  \label{prop1} Let $s\in\rr$, and let $\ell\in \cf(\co)$ be monotone in $\Sigma_0$ and such that $\ell_->-\12$ and $\ell_+<-\12$. Then for all $s'\in\rr$,  all $\ell'\in \cf(\co)$ with $\ell'_-\in\open{-\12,\ell_-}$  and all $u\in\Hsc{{s'},{\ell'}}$,
\beq\label{est1}
\|  u\|_{s,\ell} + (\Im z)^\12  \|  u\|_{s-\12,\ell+\12}  \leqslant C (\|  (P-z) u \|_{s-1,\ell+1} + \err{u}), 
\eeq
uniformly for $z\in Z\cap\{ \module{z}\geqslant\delta\}$.
\end{proposition}
\beproof The proof is based on a slight modification of the estimates in \cite[\S 5.4]{vasygrenoble} and can be found in \cite{vasyessential}. We only sketch it very briefly for the reader's convenience.  

The basic ingredients are the \emph{higher decay radial estimate} at sources and the \emph{lower decay radial estimate} into the sinks, recalled in more detail in \sec{ss:pe}, see also \sec{ss:scattering} for prerequisites on scattering pseudo-differential calculus. The first estimate (see Proposition \ref{radial1}) reads
$$
\| A u\|_{s,\ell} + (\Im z)^\12  \| A u\|_{s-\12,\ell+\12}   \leqslant C (\| B (P-z) u \|_{s-1,\ell+1} + \err{u}), 
$$
for all $u\in\Hsc{{s'},{\ell'}}$,  $\ell>\ell'>-\12$, $s,s',\in\rr$, $L_- \subset \elll_\sc(A)$, $\wf'_\sc(A)$ contained in a small neighborhood of $L_-$ in $\elll_\sc(B)$, and within $\elll_\sc(B)$, bicharacteristics from $\wf'_\sc(A)$ tend to $L_-$ in the forward  direction along the flow. The second estimate (see Proposition \ref{radial2}) is
$$
\| A u\|_{s,\ell}  +(\Im z)^{\12}  \| A u\|_{s-\12,\ell+\12}\leqslant C (\| B_1 u \|_{s,\ell}  + \| B (P-z) u \|_{s-1,\ell+1} + \| B u \|_{s',\ell'} + \err{u}), 
$$
for  all $u\in\Hsc{{s'},{\ell'}}$,   $\ell<-\12$, $\ell',s,s'\in\rr$, $L_+ \subset \elll_\sc(A)$, $\wf'_\sc(A)$ contained in a small neighborhood of $L_+$ in $\elll_\sc(B)$, and within $\elll_\sc(B)$, bicharacteristics from $\wf'_\sc(A)\setminus L_+$ tend to $L_+$ in the forward direction along the flow, and intersect $\elll_\sc(B_1)$ in the backward direction.

Thanks to the non-trapping assumption, by taking $\ell$ as in the assumption of the proposition, the two estimates applied in a neighborhood of $\Sigma_0\cap\fibinf$ can be combined with propagation of singularities estimates (Proposition \ref{pos}) and with the elliptic estimate to yield \eqref{est1} (see e.g.~\cite[\S 3.2]{hassell} for a pedagogical explanation of how to combine this type of estimates).
\qeds

By iterating the estimate  \eqref{est1} we can conclude that for all $\Im z>0$ and all $N\in\nn_{\geqslant 0}$,  
\beq\label{est3}
(P-z)^{-N}:L^2_{\rm c}(M)\to H^N_{\loc}(M).
\eeq
By replacing $P$ by $-P$ (the rôle of $L_+$ and $L_-$ is then exchanged) we also obtain  \eqref{est3} for $\Im z<0$. Note that in contrast to the elliptic case, we cannot expect that the image is in $H^{2N}_{\loc}(M)$.

To show regularity properties of non-integer powers we will need the following more precise statement.

\begin{proposition}  \label{prop2} Let $\varepsilon>0$, $N\in\nn_{>0}$, $s\in\rr$, and let $\ell\in \cf(\co)$ be monotone in $\Sigma_0$ and such that $\ell_->-\12$ and $\ell_+<-N$. Then
\beq\label{est2}
\| (P-i \varepsilon)^{-N}  (P-z)^{-1}f   \|_{s,\ell}\leqslant C (\Im z)^{-\12}   (\|  f \|_{s-N-\12,\ell+N+\12} + \norm{f}), 
\eeq
uniformly for all $z\in Z\cap\{ \module{z}\geqslant\delta\}$ and   for all $f\in L^2(M)\cap \Hsc{s-N-\12,\ell+N+\12}$.
\end{proposition}
\beproof Since $\ell_->-\12$ and $\ell_+<-N+\12$, we can  apply Proposition \ref{prop1}  with  $u= (P-i \varepsilon)^{-N}  (P-z)^{-1}f\in L^2(M)$ and then iterate the estimate, $N$ times in total. By dropping the part proportional to $\Im z$ from the l.h.s.~each time, we obtain
$$
\| (P-i \varepsilon)^{-N}  (P-z)^{-1}f   \|_{s,\ell}\leqslant C (\| (P-z)^{-1}  f \|_{s-N,\ell+N} +\err{u}).
$$
Since $\ell_- + N -\12>-\12$ and $\ell_+ +N -\12<-\12$,  we can apply Proposition \ref{prop1}  to  $v=(P-z)^{-1}f\in L^2(M)$. By keeping only the part proportional to $(\Im z)^\12$ on the l.h.s.~we obtain
$$
(\Im z)^{\12}\| (P-z)^{-1}  f \|_{s-N,\ell+N}\leqslant  C(\|  f \|_{s-N-\12,\ell+N+\12} + \err{v}).
$$
The terms $\err{u}$ and $\err{v}$ above can be estimated by $\norm{(P-z)^{-1}f}$ and thus by $(\Im z)^{-1}$ using the self-adjointness of $P$. Combining the estimates yields \eqref{est2}. \qed

\ber\label{remdisc} Proposition \ref{prop2} is also valid for $N= 0$ if $\ell_->0$, as can easily be seen by dropping the first part of the proof.
\eer

 \subsection{From resolvent to complex powers}\label{ss:restocp} As $P$ is a self-adjoint operator,  the complex powers $(P-i\varepsilon)^{-\cv}$ are well-defined by functional calculus for $\varepsilon>0$ and $\cv\in\cc$, and also for $\varepsilon=0$ if $\Re\cv <0$. We  deduce below various regularity properties of  $(P-i\varepsilon)^{-\cv}$ from resolvent estimates. 
 
 We will express $(P-i\varepsilon)^{-\cv}$ as an integral of $(P-z)^{-1}$ over a contour $\gamma_\varepsilon$ defined as follows. Let $\tilde\gamma_{\varepsilon}$ be a   contour going from $\Re z\ll 0$ to  $\Re z\gg 0$ in the upper half-plane of the form
    $$
    \tilde\gamma_{\varepsilon} = e^{i(\pi-\theta)}\opencl{-\infty,\textstyle\frac{\varepsilon}{2}}\cup \{\textstyle\frac{\varepsilon}{2} e^{i\omega}\, | \, \pi-\theta<\omega<\theta\}\cup  e^{i\theta}\clopen{\textstyle\frac{\varepsilon}{2},+\infty}
    $$  
    for some fixed $\theta\in\open{0,\pid}$. We then define $\gamma_\varepsilon\defeq \tilde\gamma_\varepsilon+i\varepsilon$ (see Figure \ref{fig:contour}). We also define its degenerate version    $\gamma_{0}$, which also goes from $\Re z\ll 0$ to  $\Re z\gg 0$ in the upper half-plane and is the form
         $$
         \gamma_{0} = e^{i(\pi-\theta)}\opencl{-\infty,0}\cup  e^{i\theta}\clopen{0,+\infty}.
         $$  

 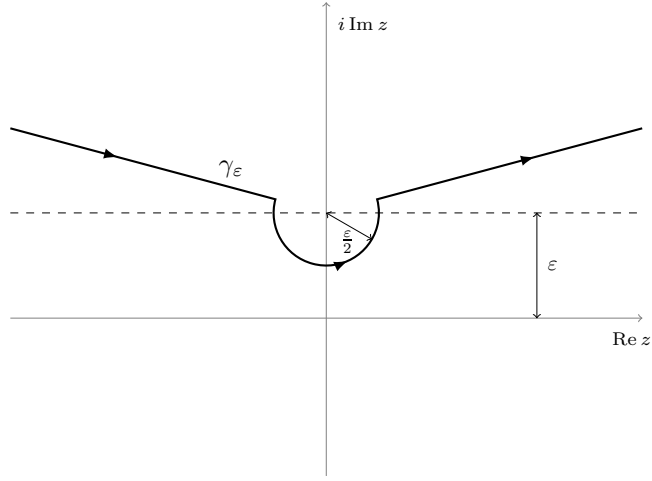
\begin{figure}
 \begin{tikzpicture}[scale=1.4]
 % Configurable parameters
 \def\bigradius{3}
 \def\incangle{15}
 
 %This is 0.5*\varepsilon
 \def\littleradius{0.5}
 
 % Axes
 \draw [help lines,->] (-1.0*\bigradius, 0) -- (1.0*\bigradius,0);
 \draw [help lines,->] (0, -0.5*\bigradius) -- (0, 1.0*\bigradius);
 
 % Contour
 \begin{scope}[shift={(0,2*\littleradius)}]
  \node at (-0.9,0.42) {$\gamma_{\varepsilon}$};
\path[draw,line width=0.8pt,decoration={ markings,
      mark=at position 0.15 with {\arrow{latex}}, 
      mark=at position 0.53 with {\arrow{latex}},
      mark=at position 0.85 with {\arrow{latex}}},postaction=decorate] (-\bigradius,{\bigradius*tan(\incangle)})   -- (-\incangle:-\littleradius) arc (180-\incangle:360+\incangle:\littleradius)   -- (\bigradius,{\bigradius*tan(\incangle)});
      \path[draw,line width=0.2pt,postaction=decorate,<->] (0,0)   -- (330:\littleradius);
\end{scope}

% Horizontal line y=\varepsilon
\path[draw,line width=0.2pt,postaction=decorate,dashed] (-\bigradius,2*\littleradius)   -- (\bigradius,2*\littleradius);

%Vertical arrow labels
\path[draw,line width=0.2pt,postaction=decorate,<->] (2,0)   -- (2,2*\littleradius);
\node at (2.15,\littleradius){\scalebox{0.8}{$\varepsilon$}};
\node at (0.2,1.45*\littleradius){\scalebox{0.8}{$\frac{\varepsilon}{2}$}};

 %Other labels 
 \node at (2.9,-0.2){$\scriptstyle \Re z$};
 \node at (0.35,2.8) {$\scriptstyle i \Im z$};
 %\node at (-1.8,2.8) {$\gamma_{R}$};
% \node at (1.9,0.29) {$l_1$};
% \node at (1.555,-0.32) {$l_2$};
 
 \end{tikzpicture}
 \caption{\label{fig:contour}The contour $\gamma_\varepsilon$ used to express $(P-i\varepsilon)^{-\cv}$ as an integral of the resolvent $(P-z)^{-1}$ for $P$ self-adjoint. If $\varepsilon=0$ the contour degenerates to two half-lines intersecting the real line at $0$.}
 \end{figure}

 \begin{proposition}\label{laplap} Assume  $(M,g)$ is non-trapping.  
 \begin{enumerate}
 \item\label{er1} For all $\varepsilon>0$ and $\cv\in\cc$, $\Hsc{\infty,0}\subset \Dom(P-i \varepsilon)^{-\cv}$.
  \item\label{er1p} For all $\Re\cv<0$, $\Hsc{\infty,0}\subset \Dom(P-i0)^{-\cv}$.
 \item\label{er2} For all $u\in \Hsc{\infty,0}$ and all $\varepsilon>0$,  the functions
$
 \cc\ni\cv\mapsto (P-i\varepsilon)^{-\cv} u \in L^2(M)
 $ 
and $
 \{\Re \cv < 0 \} \ni\cv\mapsto (P-i 0)^{-\cv} u \in L^2(M)
 $ 
 are  holomorphic. 
  \item\label{er35} For all $\varepsilon\geqslant 0$, $\Re \cv<0$, $s\geqslant 0$  and $\epsilon>0$, $$(P-i \varepsilon)^{-\cv} : \Hsc{s,\infty}\to \Hsc{s+2\alfint-\32,-\epsilon}$$ continuously. 
 \item\label{er3} For all $\varepsilon>0$, $\Re \cv>\12$, $s\geqslant 0$  and $\epsilon>0$, $$(P-i \varepsilon)^{-\cv} : \Hsc{s,\infty}\to \Hsc{s+\alfinth+\12,-\alfinth+\12-\epsilon}$$ continuously.
 \end{enumerate}
 \end{proposition} 
 \beproof   \eqref{er1}: We write $(P-i\varepsilon)^{-\cv}=(P-i\varepsilon)^{-\cv-N}(P-i\varepsilon)^{N}$ for some  $N\in \nn_0$ with $N\geqslant -\Re \cv$. Then  $(P-i\varepsilon)^{-\cv-N}$ is bounded on $L^2(M)$ and $(P-i\varepsilon)^{N}:\Hsc{\infty,0}\to \Hsc{\infty,0}$ continuously, hence $(P-i\varepsilon)^{-\cv}: \Hsc{\infty,0}\to L^2(M)$ continuously  and the claim follows. 
 
  \eqref{er1p}: The assertion follows directly from  \eqref{er1} and the fact  that (see \ref{ss:contours}) for all $\Re\cv<0$ and all $\varepsilon>0$, $\Dom (P-i 0)^{-\cv} = \Dom (P-i\varepsilon)^{-\cv}$.  
 
 \eqref{er2}: This  follows easily from \eqref{er1} and functional calculus.
  
   \eqref{er35}: Let  $N=-\alfint+1$. By   \eqref{ctin2},
   \beq\label{jkjk}
   (P-i\varepsilon)^{-\cv}=   \frac{1}{2\pi i}\int_{\gamma_{0}+i\varepsilon} \frac{(z-i\varepsilon)^{-\cv}}{(z-i)^N} (P-i)^N(P-z)^{-1} dz.
   \eeq
   By $P-i\in\Psi^{2,0}_{\sc}(M)$ and by  Remark \ref{remdisc},  for $L\in\rr$ sufficiently large the norm of
   \beq\label{soboma}
   (P-i)^{N}(P-z)^{-1}: {\Hsc{s,L}\to \Hsc{s+\12-2N,-\epsilon}}\, \mbox{ is }\, \cO(\left|\Im z\right|^{-\12}).
   \eeq
  By dominated convergence, the  integral \eqref{jkjk} is bounded on the same spaces.

 \eqref{er3}: We write $(P-i\varepsilon)^{-\cv}=(P-i\varepsilon)^{-N}(P-i\varepsilon)^{-\mu}$, where  $N=\alfinth$ and $\Re\mu>\12$, and then we express $(P-i\varepsilon)^{-\mu}$ in terms of the resolvent of $i P$ as a contour integral using  \eqref{ctin0}. This gives 
 \beq\label{jkjk2}
 (P-i\varepsilon)^{-\cv}=   \frac{1}{2\pi i}\int_{\gamma_{\varepsilon}} (z-i\varepsilon)^{-\mu}   (P-i\varepsilon)^{-N}(P-z)^{-1} dz.
 \eeq
 By Proposition \ref{prop2} (with $s+N+\12$ instead of $s$ and $\ell_+=-N-\epsilon$), for $L\in\rr$ sufficiently large the norm of
 \beq\label{soboma2}
 (P-i\varepsilon)^{-N}(P-z)^{-1}: {\Hsc{s,L}\to \Hsc{s+N+\12,-N-\epsilon}}\, \mbox{ is }\, \cO(\left|\Im z\right|^{-\12}).
 \eeq
By dominated convergence, the  integral \eqref{jkjk2} is bounded on the same spaces.
 \qeds

 \subsection{Estimates uniform down to the real axis} \label{ss:realaxis1} In \cite{vasyessential}, estimates uniform down to the real axis are obtained under the extra hypothesis that $P=\square_g+m^2$ with  $m^2\neq 0$ and $(M,g)$ is \emph{non-trapping at  energy $\sigma=m^2$} (see Definition \ref{maindef}). This is not necessary for our main results. However, we briefly discuss the improved estimates here as they  lead to  stronger  results (in terms of the dependence on $\varepsilon$ for functions of $P-i\varepsilon$) in later sections.
  
  The non-trapping at energy $\sigma=m^2$ ensures that  the Fredholm estimates  for $P-i\varepsilon$ are uniform down to  $\varepsilon=0$. Let us state this as a proposition (proved in analogy to Proposition \ref{prop1}) for further reference.
  
  \begin{proposition}[{\cite[Prop.~2]{vasyessential}}]  \label{propunif} Let $(M,g)$ be a non-trapping Lorentzian scattering space and  assume it is non-trapping at energy $\sigma=m^2\neq 0$.  Let $s\in\rr$, and let $\ell\in \cf(\co)$ be monotone in $\Sigma_\sigma$ and such that $\ell_->-\12$ and $\ell_+<-\12$. Then there exists $\delta>0$ such that for all $s'\in\rr$, all $\ell'\in \cf(\co)$ with $\ell'_-\in\open{-\12,\ell_-}$  and all $u\in\Hsc{{s'},{\ell'}}$,
  \beq\label{pest}
  \|  u\|_{s,\ell} \leqslant C (\|  (P-z) u \|_{s-1,\ell+1} + \err{u})
  \eeq
  uniformly in $z\in\{\Im z > 0\}\cup \{ | z|\leqslant \delta \}$.
  \end{proposition}
 
 An injectivity property is needed to get an invertibility statement in weighted Sobolev spaces down to $\Im z\to 0^+$.
 
  \begin{definition}\label{hyp3} We say that \emph{injectivity} holds at $\sigma=m^2\neq 0$ if for some $s\in\rr$ and some $\ell\in \cf(\co)$  monotone with $\ell_->-\12$ and $\ell_+<-\12$
 \beq\label{trivialker}
 ( u\in \Hsc{s,\ell}, \ Pu =0 )  \ \Rightarrow \  (  u=0 ).
  \eeq
  \end{definition}
 
 Sufficient conditions for injectivity with $-1<\ell_+<-\12$ are discussed in \cite{vasyessential}.  As a consequence one concludes a \emph{limiting absorption principle}, i.e.~the existence of the limiting operator $(P-i0)^{-1}$ on weighted Sobolev spaces. We state here the following variant.

 \begin{proposition}  \label{prop3} Let $(M,g)$ be a non-trapping Lorentzian scattering space and  assume  non-trapping  at energy $\sigma=m^2$ and injectivity. Let $s\in\rr$ and $\ell\in \cf(\co)$ be as in Hypothesis \ref{hyp3}. There exists $\delta>0$ such that
 \beq\label{est3p}
 \|  (P-z)^{-1}  f\|_{s,\ell} \leqslant C \|  f \|_{s-1,\ell+1}, 
 \eeq
 uniformly for $\Im z \geqslant 0$, $|z|<\delta$.
 \end{proposition}
 \proof By \cite[Thm.~5]{vasyessential}, if $\Im z \geqslant 0$ and $|z|<\delta$ with $\delta$ sufficiently small then $(P-z)^{-1}$ tends to $(P-i0)^{-1}$  in the weak operator topology as $z\to 0$. By compactness of the embeddings $\Hsc{s_1,\ell_1}\hookrightarrow\Hsc{s_2,\ell_2}$ for any $s_1>s_2$ and $\ell_1>\ell_2$  this gives boundedness in norm, i.e.~\eqref{est3p}. \qed

 \section{Wavefront set of the resolvent}\label{sec:wf}
 
 \subsection{Summary} Our next goal is to estimate the wavefront set of the resolvent $(P-z)^{-1}$ and give sufficient conditions for a given parametrix of $P-z$ to be equal $(P-z)^{-1}$ modulo smooth terms (in the sense of  having smooth Schwartz kernel in $M\times M$). This needs to be true \emph{uniformly in $z$} in an appropriate sense because we will be then interested in integrating in $z$ when considering complex powers.
 
 We remark that techniques to deduce the wavefront set of resolvents from propagation estimates were developed by Dyatlov--Zworski \cite{Dyatlov2016} in the semi-classical case, originally in the context of Anosov flows. Here, we use an argument more similar to the work of Vasy--Wrochna \cite{vasywrochna} and we also construct a parametrix  related to that of  G\'erard--Wrochna \cite{GWfeynman}. The disadvantage as compared to the semi-classical approach is that it is less evident how to deal with possible singularities of Schwartz kernels which are microlocally at $\zero\times T^*M$ or $T^*M\times \zero$, where $\zero$ is the zero section. This issue is however circumvented by considering an operator version of the wavefront set similarly as in \cite{holographic,gannotwrochna}.
 
\subsection{Uniform operator wavefront set} Let $\complex \subset \cc$ and let $h$ be a strictly positive function on $\complex$. Suppose $G_z: \cE'(M)\to \cD'(M)$ for $z\in \complex$. 

  \begin{definition}For $s\in\rr$, we write
$$
G_z = \Oreg{}
$$ 
if for all $l\in\rr$, $h(z)^{-1} G_z$ is a uniformly bounded family of continuous operators  $H^l_\c(M)\to H^{l+s}_\loc(M)$. We write $G_z=\Onorm{}$ if $G_z=\Oreg{}$ for some $s\in \rr$. 
\end{definition}

Note that $G_z=\Oreg{}$ implies  $G_z^*=\Oreg{}$, and  $G_z=\Onorm{}$ implies $G_z^*=\Onorm{}$.

We will be mostly interested in wavefront set estimates in the interior $M$ of $\M$. Over the interior, $\fibinf$ is isomorphic to $\cosb$, the boundary of the fiber compactification of $T^*M$. We denote by $\Psi^s(M)$ the class of \emph{properly supported} pseudo-differential operators of order $s\in\rr$ on $M$ (in the sense of the usual calculus on the boundaryless manifold $M$).  One says that $A\in\Psi^s(M)$ is \emph{elliptic at $q\in\cosb$} if its principal symbol is non-zero at $q$.

\bed If $G_z=\Onorm{}$ then its \emph{uniform operator wavefront set of order $s\in\rr$} is the set $\wfl{}(G_z)\subset\cosb\times \cosb$ defined as follows:  $(q_1,q_2)\notin\wfl{}(G_z)$ iff there exist  $B_i\in \Psi^{0}(M)$, elliptic at $q_i$ ($i=1,2$) and such that 
\beq\label{unireg}
 B_1 G_z  B_2^* = \Oreg{}.
\eeq
\eed
 
 We show several elementary properties of the uniform operator wavefront set, the proof of which is to a large extent analogous to \cite[\S 5.1]{holographic}.
 
Let us recall that for $A\in\Psi^s(M)$, there is a closely related notion of operator wavefront set $\wf'(A)$ which characterizes the directions $q\in\cosb$ in which microlocally, $A$ does not behave as a regularizing operator.

 \begin{lemma}\label{lem:altdefwfb} For any $q_1,q_2\in\cosb\times\cosb$, $(q_1,q_2)\notin\wfl{}(G_z)$ if and only if for $i=1,2$ there exist neighborhoods $\Gamma_i$ of $q_i$ such that \eqref{unireg} holds true for all $B_i\in\Psi^{0}(M)$ elliptic at $q_i$ and satisfying $\wf'(B_i)\subset \Gamma_i$.
 \end{lemma}
 \beproof Suppose $(q_1,q_2)\notin\wfl{}(G_z)$, so that there exists $A_i \in\Psi^0(M)$, $i=1,2$,  elliptic at $q_i$, such that $A_1 G_z A_2^* = \Oreg{}$. There exists a compact neighborhood $\Gamma_i$ of $q_i$ on which $A_i$ is elliptic. Therefore, there exists $A_i^\inv\in\Psi^0(M)$ such that
 \[
 \wf'(A^\inv_i A_i - \one)\cap \Gamma_i = \emptyset.
 \]
 Let $B_i\in\Psi^0(M)$ be elliptic at $q_i$ and such that $\wf'(B_i)\subset \Gamma_i$. These conditions  imply that
 \beq\label{eq:bababa}
 B_1(A^\inv_1 A_1 - \one)\in\Psi^{-\infty}(M), \ \ (A_2^* (A^\inv_2)^* - \one)B_2^*\in\Psi^{-\infty}(M).
 \eeq
 We can write
 \[
 \bea
 B_1 G_z B_2^*&=B_1 A_1^\inv A_1 G_z  A_2^* (A^\inv_2)^* B_2^* + B_1 (\one - A_1^\inv A_1) G_z A_2^* (A^\inv_2)^* B_2^*\\
  & \phantom{=}\, + B_1 A_1^\inv A_1  G_z (\one-A_2^* (A^\inv_2)^* )B_2^* \\
  & \phantom{=}\, + B_1 (\one - A_1^\inv A_1) G_z (\one-A_2^* (A^\inv_2)^*)B_2^*.
 \eea 
 \]
 By $A_1 G_z A_2^* = \Oreg{}$ and \eqref{eq:bababa}, all the summands are $\Oreg{}$, hence $$B_1 G_z B_2^*=\Oreg{}.$$  The opposite direction is trivial.\qed

 \begin{lemma}\label{wfs} If $G_z,G'_z=\Onorm{}$, then
 \[
\wfl{}(G_z+\tilde G_z)\subset \wfl{}(G_z)\cup\wfl{}(\tilde G_z).
 \]
 \end{lemma}
 \beproof If $(q_1,q_2)\notin \wfl{}(G_z)$ and $(q_1,q_2)\notin\wfl{}(\tilde G_z)$ then by Lemma \ref{lem:altdefwfb} we can choose $B_1,B_2$ elliptic at resp.~$q_1,q_2$ such that  $$B_1 G_z B_2^* \mbox{ and }  B_1 \tilde G_z B_2^* \mbox{ are both } \Oreg{}.$$ Hence  $B_1 (G_z+\tilde G_z) B_2^*$ are $\Oreg{}$ and thus $(q_1,q_2)\notin\wfl{}(G_z+\tilde G_z)$.\qed
 
 \begin{proposition} Suppose $\wfl{}(G_z)=\emptyset$. Then $G_z=\Oreg{}$.
 \end{proposition}
 \beproof  It suffices to show that for any $x_1,x_2\in M$ there exists $\chi_1,\chi_2\in\cf_\c(M)$ with $\chi_i\equiv 1$ near $x_i$ such that $\chi_1  G_z \chi_2=\Oreg{}$.
 
 By definition of $\wfl{}(G_z)$, for any $q,q'\in \cosb$ there exist $B_{1,q},B_{2,q'}\in\Psi^0(M)$ elliptic at resp.~$q$, $q'$, such that $B_{1,q} G_z B_{2,q'}^*=\Oreg{}$. Let $\Gamma_{1,q}$ be the set on which $B_{1,q}$ is elliptic. 
 
  Then $\{ \Gamma_{1,q} \,  | \, q\in \cosba{x_1} \}$ is an open cover of $\cosba{x_1}$. By compactness we can find a finite subcover $\{ \Gamma_{1,q_j}\}_{j=1}^{N}$. Then $B_1=\sum_{j} B_{1,q_j}^*B_{1,q_j}\in\Psi^0(M)$ is elliptic on $\cosba{x_1}$. In a similar way we construct $B_2= \sum_{l} B_{2,q'_l}^* B_{2,q'_l}\in \Psi^0(M)$ elliptic on  $\cosba{x_2}$. This gives
 \[
 B_1 G_z B_2^* = \textstyle\sum_{j,l}  B_{1,q_j}^*B_{1,q_j} G_z B_{2,q'_l}^* B_{2,q'_l}= \Oreg{}
 \]
 using that the sum is finite. 
 
 We can find a microlocal parametrix of $B_1$ and $B_2$, i.e.~$B_i^\inv\in\Psi^0(M)$ such that $R_1=\one-B_1^{\inv} B_1 $ and $R_2=\one-B_2 B_2^{\inv}$ satisfy $\wf'(R_i)\cap \cosba{x_i} =\emptyset$. This implies that there is a neighborhood $O_i$ of $x_i$ in $M$ such that $\wf'(R_i)\cap \cosba{O_i}=\emptyset$. Let $\chi_i\in\cf_\c(M)$ be such that $\supp \chi_i\subset O_i$ and $\chi_i\equiv 1$ near $x_i$. We have
 \[
 \bea
  \chi_1 G_z\chi_2 &=\chi_1 B_1^\inv (B_1 G_z  B_2^*) B^{\inv*}_2 \chi_2+ \chi_1 R_1 G_z B_2^* B^{\inv*}_2 \chi_2\\
  & \phantom{=}\, + \chi_1 B_1^\inv B_1  G_z R_2^* \chi_2 +  \chi_1 R_1 G_z R_2^* \chi_2,
 \eea 
 \]
 where all the summands are $\Oreg{}$, hence $\chi_1 G_z\chi_2\in \Oreg{}$. \qed
 
 \begin{lemma}\label{lemadj} If $G_z=\Onorm{}$ then $(q_1,q_2)\in \wfl{}(G_z)$ if and only $(q_2,q_1)\in \wfl{}(G_z^*)$.
 \end{lemma}
 \beproof If  $B_1 G_z  B_2^* = \Oreg{}$ then its formal adjoint $B_2^* G_z^*  B_1$ is  $\Oreg{}$ as well, where $B_2^*$ is elliptic at $q_2$ and $B_1^*$ is elliptic at $q_1$. \qeds
 
 \begin{lemma}\label{lem:composition} Let $G_{1,z}=\Onormsh{}(h_1(z))$ and $G_{2,z}=\Onormsh{(h_2(z))}$ and suppose that the operators $G_{2,z}$ are proper{ly supported} for all $z\in Z$. Then the composition $G_{1,z}  G_{2,z}= \Onormsh{}(h_1 h_2(z))$ is well-defined and satisfies
 \beq\label{lklklklklk}
 \wfl{6}(G_{1,z}G_{2,z}) \subset \wfl{4}(G_{1,z}) \circ \wfl{5}(G_{2,z}),
 \eeq
 where the composition of $\Gamma_1,\Gamma_2\subset  \cosb\times\cosb$ is defined by
 $$
 \Gamma_1\circ \Gamma_2 = \{ (q_1,q_2)\in \cosb\times\cosb \st \exists q\in\cosb \mbox{\,s.t.\,}  (q_1,q)\in \Gamma_1, \,   (q,q_2)\in \Gamma_2  \}.
 $$
 \end{lemma} 
 \begin{proof}   For all $A_1,A_2\in\Psi^0(M)$,
 \beq\label{eq:temokn}
 A_1 G_{1,z}G_{2,z} A_2^*= \textstyle\sum_{k} (A_1   G_{1,z} B_{k}^*) (B_k G_{2,z} A_{2}^*), 
 \eeq
 where   $B_k\in\Psi^0(M)$ is an arbitrary family such that $\sum_k  B_k^* B_k =\one$ as a locally finite sum. By taking $\wf'(B_k)$  sufficiently small and using \eqref{eq:temokn} we obtain \eqref{lklklklklk}.
 \end{proof} 
 
 Let us now explain the relation with more standard notions which will be used in later sections.
 
 \bed\label{def:o} Let $X$ be a smooth (boundaryless) manifold and let $\varlambda\subset T^*X\setminus\zero$ be conic.  Let $\{u_z\}_{z\in\complex}$ be a family of distributions on $X$. We write $u_z=\Olambda{}$ iff for all $A\in \Psi^0(X)$ satisfying $\wf'(A) \cap  \varlambda=\emptyset$ we have $A u_z = \Osmooth{}$. 
 \eed
 
 \bed\label{anewdef}
 Let $\kappa :T^*M\setminus\zero \to  \cosb$ be the quotient map for the $\rr_{>0}$ action by fiberwise dilations.  For each conic set $\varlambda\subset T^*(M\times M)\setminus\zero$ we define
 $$
 \varlambda'=\{  (\kappa(x_1;\xi_1), \kappa(x_1;-\xi_2))    \st (x_1,x_2;\xi_1,\xi_2)\in \varlambda, \, \xi_1\neq 0, \ \xi_2\neq 0  \},
 $$ 
which is a subset of $\cosb\times \cosb$.\eed
 
 \bel\label{ditoop} Suppose $\varlambda\subset T^*(M\times M)\setminus\zero$ is conic and $G_z=\Onorm{}$. If the associated family of Schwartz kernels satisfies $G_z(\cdot)=\Olambda{}$ then $\wfl{}(G_z)\subset \varlambda' $
 for all $s\in\rr$.
 \eel 
 \beproof {For ease of notation we  identify $T^*M\setminus\zero$ with   $\cosb$  using the quotient map $\kappa$.  Let $q=(x_1,x_2; \xi_1,\xi_2)\in T^*(M\times M)\setminus \Lambda$ with $\xi_1\neq 0$ and $\xi_2\neq 0$, and  let $\Gamma_i$, $i=1,2$,  be a small conic neighborhood of $(x_i;\xi_i)$, to be fixed later on. 
   Let $B_i\in\Psi^0(M)$ be elliptic at $(x_i;\xi_i)$, with $\wf'(B_i)\subset \Gamma_i$.  Let  $A\in\Psi^0(M\times M)$ be elliptic  on $\Gamma_1\times \Gamma_2$ and with symbol vanishing in a conical neighborhood of $\zero\times T^*M$ and $T^*M\times \zero$. This implies that $A(B_1\otimes B_2)\in \Psi^0(M)$  and that $A(B_1\otimes B_2)$ is elliptic at $q$. Since  $G_z(\cdot)=\Olambda{}$ and $q\notin \varlambda$,  we can take $\Gamma_1,\Gamma_2$ such that $\wf'(A(B_1\otimes B_2))$ is in a small enough neighborhood of $q$ so that $A(B_1\otimes B_2) G_z(x,y)=\Osmooth{}$. By ellipticity of $A$ on $\Gamma_1\times \Gamma_2$,  this  implies $(B_1\otimes B_2) G_z(x,y)=\Osmooth{}$, where $B_1,B_2$ acts on the first, resp.~second variable of the Schwartz kernel of $G_z$. Hence $B_1 G_z \bar{B}_2^*=\Oreg{}$ for all $s\in\rr$, where $\bar {B}_2^*\in \Psi^0(M)$ is defined via complex conjugation of the Schwartz kernel of  ${B}_2^*$ and is in consequence  elliptic at $(x_2;-\xi_2)$.  This implies $((x_1;\xi_1), (x_2;-\xi_2)) \notin \wfl{}(G_z)$ for all $s\in\rr$ as claimed. } \qeds
 
 \subsection{Uniform parametrices}\label{ss:unif} After these general considerations we return to the setting of the operator $P-z$, though for the moment the only relevant assumption is that $P$ is of real principal type.  
 
 Recall that $\fibinf$ is identified with $\cosb$ over the interior $M$ of $\M$, and in this sense the characteristic set of $\square_g-z$ over $M$ is
 $$
 \Sigma\defeq \Sigma_z  \cap \cosb
 $$ 
and does not depend on $z$. Let us denote by $t\mapsto \Phi_t$ the bicharacteristic flow in $\Sigma$. For $q_1,q_2\in \cosb$, we write
 $$
q_1\sim q_2 \ \  (\mbox{resp.}   \ q_1 \prec  q_2,  \mbox{ or } q_1 \succ  q_2)
 $$
 if $q_1,q_2\in \Sigma$ and $\Phi_t (q_1)=q_2$ for some $t\in \rr$ (resp.~$t> 0$, or $t<0$). If $q_1\sim q_2$ we denote by $\gamma_{q_1\sim q_2}$ {the closed bicharacteristic segment in $\Sigma$  from $q_1$ to $q_2$}.

We consider families of operators $\{G_z\}_{z\in \complex}$  parametrized by some $\complex\subset \cc$. If the reference weight $h(z)$ is identically $1$ we simply write $\wfl{0}(G_z)$ for the uniform Sobolev wavefront set instead of $\wf^{\,(s)}_1(G_z)$. Other particularly useful weights are $h(z)=\module{\Im z}^{-\12}$, $h(z)=\bra  \Im z\ket^{-\12}$ and $h(z)=\bra z\ket^{-\12}$. %Note that along $\gamma_\varepsilon$, $\module{\Im z}\sim\module{\Re z}\sim \bra z \ket$ so all three can be used interchangeably. 

We state below a variant of H\"ormander's propagation of singularities theorem { for the uniform operator wavefront set. For convenience, we formulate it here  as a corollary of propagation estimates in weighted Sobolev spaces  recalled in   \sec{ss:pe} in our setting.   Note that the statement is  valid in greater generality (aspects at infinity of $M$ being irrelevant), as one can give a  direct proof by adapting   Hörmander's positive commutator estimates \cite[\S6.5]{Hormander1971} along the lines of  Vasy's work \cite{vasyessential} to account for the case  $\Im z \neq 0$.  }

%It is important that positive commutator estimates that underpin the proof of the propagation singularities theorem are  \emph{uniform in $z$}. In fact, the $z$ term does not affect any commutator, and does not enter the principal (nor even sub-principal) symbol in the sense of the usual calculus over $M$.
 
 \begin{proposition}\label{opprop} Let {$Z\subset \{ \Im z \geqslant 0\}$}. Assume $G_z=\Onorm{0}$, and suppose   $(P-z)G_z=\Onorm{0}$ and
 \beq\label{eqt1}
 (q_1,q_2)\in \wfl{0}(G_z)\setminus \wfl[s-1]{0}\big((P-z)G_z\big).
 \eeq
Then $q_1\in \Sigma$,  and   $(q_1',q_2)\in\wfl{0}(G_z)$ for all $q_1'$ such that {$q_1'\prec q_1$} provided that $(q,q_2)\notin\wfl[s-1]{0}\big((P-z)G_z\big)$ for all $q\in \gamma_{q_1\sim  q_1'}$. Similarly, if $G_z(P-z)=\Onorm{0}$ then
 \beq\label{eqt2}
 (q_1,q_2)\in \wfl{0}(G_z)\setminus \wfl[s-1]{0}\big(G_z(P-z)\big),
 \eeq
then $q_2\in \Sigma$, and $(q_1,q_2')\in \wfl{0}(G_z)$   for all $q_2'$ such that {$q_2'\succ q_2$} provided that $(q_1,q)\notin\wfl[s-1]{0}\big(G_z(P-z)\big)$ for all $q\in \gamma_{q_2\sim  q_2'}$.
 \end{proposition}
 \beproof For the first statement, suppose $(q_1',q_2)\notin \wfl{0}(G_z)$. Then by definition there exist $B_1',B_2\in\Psi^0(M)$ elliptic at respectively $q_1',q_2$ such that for any bounded subset $H^l_\c(M)$, the set $B_1' G_z B_2^* \cU$ is uniformly bounded in $H^{l+s}_\loc(M)$. {Without loss of generality we can assume $B_1'$ to be supported away from $\p\M$, so that $B_1'\in\Psi^0_\sc(M)$ and $\wf'_\sc(B_1')\cap \basinf = \emptyset$ (see \sec{ss:scattering}).  We apply Proposition \ref{pos}  to each $u\in G_z B_2^* \cU$. More precisely,    we take $\ell$ arbitrary, $A_2=B_1'$, and $B,A_1\in \Psi^0(M)$ supported away from $\p\M$ such that:  $A_1$ is elliptic at $q_1$, $\wf'(A_1)$ is a small neighborhood of $q_1$,  and $B$  is elliptic on $\gamma_{q_1\sim q_1'}$ (note that in view of $A_1,A_2,B$ being supported away from  $\p\M$, this simply corresponds to  propagation of singularities at fiber infinity $\cosb$).  In consequence (using the sign condition $\Im z \geqslant 0$ to  get rid of the  $(\Im z )^\12    \norm{A_1 u }_{s-\12,\ell+\12}$  term in the estimate), we obtain that $A_1 G_z B_2^* \cU$ is bounded in $H^{l+s}_\loc(M)$. Hence, $(q_1,q_2)\notin\wfl{0}(G_z)$.} This proves the first statement. 
 
The second statement follows by applying {the analogue of the first statement for $\Im z \leqslant 0$} to  the adjoint families $G_z^*$ and $(P-z)^*$ and then using  Lemma \ref{lemadj}.
 \qeds
 
 {Note that if $Z\subset \{\Im z =0\}$, then by considering $-P$ instead of $P$ one obtains propagation in the other direction, and in consequence, in that case $q_1'\prec  q_1$ and $q_2'\succ q_2$ can be replaced by $q_1'\sim q_1$ and $q_2'\sim q_2$ in the statement of Proposition \ref{opprop}. }
 
 \begin{remark}\label{extraremark} In \eqref{eqt1}  and  \eqref{eqt2} the set $\wfl{0}(G_z)$ can be replaced by
$$
\wfl{0}(G_z)\cup \wfl[s-\12]{2}(G_z),
  $$
  and therefore by $\wfl[s-\12]{3}(G_z)$. 
  This is a consequence of the fact that for $\Im z\neq 0$, the propagation estimates become stronger{: indeed, instead of getting rid of the term $(\Im z )^\12    \norm{A_1 u }_{s-\12,\ell+\12}$  in the proof of Proposition \ref{opprop}, we can use it get an for estimate for $\norm{A_1 u }_{s-\12}$}.
 \end{remark}

  We use the notation $q_i=(x_i;\xi_i)$ for points in $\cosba{x_i}$. The \emph{enlarged diagonal} in $\cosb\times\cosb$ is the set 
  $$
{\ediag}\defeq \{ (q_1,q_2)\in \cosb\times \cosb \, | \, x_1=x_2\}.
  $$
{By abuse of notation, in later sections the image of $\ediag$ under the identification of  $\cosb$ with $T^*M\setminus \zero$ will also be denoted by $\ediag$  (note that by definition it does not contain the zero section of $T^*(M\times M)$). }
  
  \begin{definition}\label{defpara} We say that $G_z$ is a  \emph{uniform parametrix of order $s\in\rr$}  {in $M_0\subset M$} (more precisely, a right parametrix) for the family $\{(P-z)\}_{z\in Z}$ if  $G_z=\Onorm{1}$ and
 \beq\label{defptx}
   (P-z)G_z =\one+R_z {\mbox{ on  } M_0},
 \eeq
 for some  $R_z=\Oreg{0}$.  We say that $G_z$ is a \emph{uniform local parametrix of order $s\in\rr$} if \eqref{defptx} holds true for some $R_z=\Onorm{0}$ satisfying merely 
  \beq
\wfl[s-\12]{0}(R_z)\cap U=\emptyset,
  \eeq
  where $U$ is some neighborhood of $\ediag$. 
  \end{definition}

\begin{proposition}\label{uniqueness} Suppose that $G_z$ is a uniform local parametrix of order $s$ in $M$ which satisfies
\beq\label{feynwf}
\wfl{0}(G_z)\subset \{  (q_1,q_2) \in \Sigma\times \Sigma \, | \, q_1 {\succ} q_2 \}\cup \ediag.
\eeq
If $\widetilde G_z$ is a local uniform parametrix of order $s$ which also satisfies \eqref{feynwf}, then
\beq\label{coln}
\wfl{3}(\widetilde G_z-G_z) \cap U=\emptyset
\eeq
for some neighborhood $U$ of $\ediag$.
\end{proposition} \beproof There exists a neighborhood $U$ of $\ediag$ such that 
$$
\bea
\wfl[s-\12]{0}\big(  (P-z) (\widetilde G_z - G_z)\big)\cap U=\emptyset.
\eea
$$
Suppose $(q_1,q_2)\in\wfl{3}(\widetilde G_z - G_z)\cap U$. Then by Proposition \ref{opprop} and Remark \ref{extraremark},
$q_1\in\Sigma$  and $(q_1',q_2)\in\wfl{0}(\widetilde G_z - G_z)$  for some $q_1'\sim q_1$ with $x_1'\neq x_2$, and such that $q_1'{\prec} q_2$ if $q_1\sim q_2$. On the other hand, \eqref{feynwf} and the analogous assumption for $\widetilde G_z$ implies that $q_1'{\succ}  q_2$ or $x_1'=x_2$, which gives a contradiction.  This proves \eqref{coln}. \qeds

 \subsection{Global hyperbolicity}  From now on we make the additional assumption that $(M,g)$ is a \emph{global hyperbolic spacetime}. 

Recall that $(M,g)$ is a spacetime if it is equipped  with a time orientation.  It is a \emph{globally hyperbolic spacetime} (or in short, globally hyperbolic space) if in addition it admits a  Cauchy surface, i.e., a  closed subset of $M$ which is intersected exactly once by each maximally extended time-like curve. By a result of Geroch \cite{geroch} and Bernal--S\'anchez \cite{bernal1,bernal2}, there exists an ($n-1$)-dimensional smooth manifold $\Cauchy$ and an isometric diffeomorphism $\varphi: M\to \rr\times\Cauchy$ such that
\beq\label{ortdiff}
\varphi^* g = c^{2}(t,y) dt^2- h_t(y)dy^2,
\eeq
where $c\in \cf(M)$, $c>0$, $\rr\ni t \mapsto h_t(y)dy^2$ is a smooth family of Riemannian metrics, and  for all $t_0\in\rr$, $\{t_0\}\times \Cauchy$ is a smooth space-like Cauchy surface in $\rr\times\Cauchy$.

\subsection{Uniform parametrix construction} \label{ss:firstparametrix}
We will prove the existence of a \emph{Feynman parametrix} in the sense of Duistermaat--H\"ormander, which has a special form and is uniform in $z$. 

Thanks to \eqref{ortdiff} we can work on the $n$-dimensional manifold $\rr\times\Cauchy\cong M$ with coordinates denoted by $x=(t,y)$. We will need a $t$-dependent variant of the parameter-dependent pseudo-differential calculus developed by Shubin \cite{shubin}. 

Let $Z\subset\cc$. Let $U\subset \rr^{n-1}$ be an open set and $s\in\rr$.  Recall that the symbol space  $S^s(T^*U)$ consists of functions $a(y,\eta)\in\cf(T^* U)$ such that
$$
(1+\module{\eta})^{-s+\module{\beta}} \p_y^\alpha \p_\eta^\beta a(y,\eta) \mbox{ is bounded on } U \times \rr^{n-1}
$$
for all $\alpha,\beta\in\nn_{\geqslant0}^{n-1}$. We denote by  $\smiz{s}$  the space of functions $a(t,z,y,\eta)$ such that $a(t,z_0,y,\eta)\in \cf(\rr\times T^* U)$ for each fixed $z_0\in \complex$, and
\beq\label{defshub1}
\big((1+\module{\eta}+\module{z}^{\frac{1}{2}}\!\big)^{-s+\module{\beta}}\p_t^\gamma \p_y^\alpha \p_\eta^\beta a(t,z,y,\eta) \mbox{ is bounded on } I\times Z\times U \times \rr^{n-1}
\eeq
for all $\gamma\in\nn_{\geqslant0}$, all $\alpha,\beta\in\nn_{\geqslant0}^{n-1}$ and all intervals $I\Subset \rr$.   Note that taking the square root of $\module{z}$  is natural  from the point of view of the spectral theory of \emph{second order} elliptic differential operators.  The space $\smix{s}$ is defined by replacing \eqref{defshub1} by the property that
\beq\label{defshub2}
(1+\module{\eta})^{-s+\module{\beta}}\p_t^\gamma \p_y^\alpha \p_\eta^\beta a(t,z,y,\eta) \mbox{ is bounded on } I\times Z \times U \times \rr^{n-1}
\eeq
for all $\gamma\in\nn_{\geqslant0}$, all $\alpha,\beta\in\nn_{\geqslant0}^{n-1}$ and all intervals $I\Subset \rr$. Thus, elements of $\smix{s}$ depend on $z\in \complex$, but only in a very mild way, which is why we do not indicate it in the notation explicitly.

Recall that all pseudodifferential operators in $\Psi^s(\Cauchy)$ can be obtained first by reduction to the case of an open set $U\subset \rr^{n-1}$ (using a partition of unity subordinated to a locally finite cover by charts),  then by quantization of elements of $S^s(T^*U)$, and finally by adding the ideal of smoothing operators.  By applying  an analogous  procedure to $t$ and $z$-dependent elements of $S^s(T^*U)$ we obtain classes of $t$ and $z$-dependent pseudo-differential operators on $\Cauchy$. We  denote by $\miz{s}$ the class obtained from elements of $\smiz{s}$ {(plus the ideal of $t,z$-depended smoothing of operators, bounded with all derivatives in $t$  and rapidly decaying in $\module{z}$)}, and by $\mix{s}$ the class obtained from elements of $\smix{s}$ {(plus the ideal of $t,z$-depended smoothing of operators, bounded with all derivatives in $t$  and bounded  in $\module{z}$, denoted by $\miz{-\infty}$ by abuse of notation)}.  We say that $A$ is \emph{properly supported} if there exists a closed set  $K\subset Y\times Y$ with proper projections on each factor of $Y\times Y$ and such that the Schwartz kernel of $A$ is supported in $K$ for all $t\in\rr$, $z\in\complex$.

By the exact $t$-dependent analogue of the proofs in \cite[\S{9}]{shubin} we can show properties of the $\miz{s}$ and $\mix{s}$ classes under composition and taking adjoints. In particular, for all $s_1,s_2\in\rr$, for properly supported operators we have
$$
A\in\miz{s_1}, \ B\in\miz{s_2} \\ \Rightarrow \ AB\in\miz{s_1+s_2}.
$$
Of particular use for us are operators in $\miz{s}$ with symbols that are one-step poly-homogeneous in $(\eta,z^{\12})$ (for $\module{\eta}+\module{z}^{\12}\geqslant 1$), see \cite[\S{9.1}]{shubin}. We say that such an operator $A$ is \emph{elliptic with parameter} if it is properly supported and 
$$
a_s(t,z,y,\eta)\neq  0  \ \mbox{ if } \ \module{\eta}+\module{z}^{\12}  \neq 0,
$$
where $a_s$ is the leading order term in the poly-homogeneous expansion. Standard poly-homogeneous expansion arguments can be used to show that if $A\in\miz{s}$ is elliptic with parameter, then it has a parametrix in $\miz{-s}$, which is also elliptic with parameter, and the error is in $\miz{-\infty}$.

\begin{example}[{cf.~\cite[Ex.~9.1]{shubin}}] If $L(t)$ is a second order differential operator on $\Cauchy$ with coefficients in $\cf(\rr;\cf(Y))$, then the leading order term in the poly-homogeneous expansion of the symbol of $L(t)-z$ is simply $\sigma_{\pr}(L(t))-z$, where $\sigma_\pr(L(t))$ is the principal symbol in the usual $\Psi^s(\Cauchy)$ sense. Therefore, $L(t)-z$ is elliptic with parameter if 
$\sigma_\pr(L(t))$ does not intersect $Z$  at $\{ \module{\eta} =1\}$.
\end{example}

It  is also occasionally useful to work with pseudo-differential operators of order not consistent with the order of decay in $z$. Namely, we write $R(t,z)\in\mimiz{s_1}{s_2}$ if 
$$
R(t,z)=\sum_{i=1}^k R_{1,i}(t,z) R_{2,i}(t,z) 
$$
for some $k\in\nn_{>0}$ and $R_{1,i}\in \mix{s_1}$, $R_{2,i}\in \miz{s_2}$, $i=1,\dots,k$. Using the trivial inclusion $\miz{s}\subset\mix{s}$ for $s<0$ whenever needed, we can show that for all $s_1,r_1,r_2\in \rr$ and all $s_2<0$, 
\beq\label{compopdo}
\bea
A\in\mimiz{s_1}{s_2}, \ B\in\mimiz{r_1}{r_2} \\ \Rightarrow \ AB\in\mimiz{s_1+s_2+r_1}{r_2}, \\
A\in\mimiz{s_1}{s_2}, \ B\in\miz{s_2} \\ \Rightarrow \ [A,B]\in\mimiz{s_1+s_2-1}{s_2}.
\eea
\eeq

\begin{notation} We denote by $\Sigma^+$ and $\Sigma^-$  the two connected components of $\Sigma$, distinguished by the property that within $\Sigma^\pm$, bicharacteristics flow in the past/future direction. \end{notation}

\begin{proposition}\label{painful} Assume global hyperbolicity. Let $Z\subset\cc$ be an angle in the upper half-plane $\{\Im z\geqslant 0\}$ with  vertex at the origin. {Let $M_0\subset M$ be an open subset such that for each $t\in \rr$,  $\varphi^* M_0 \cap (\{ t \} \times Y)$ is included in a compact set.} Then for all $s\geqslant 0$, $\{(P-z)\}_{z\in Z}$ has a uniform parametrix $G_z$ of  order $s\in\rr$ {in $M_0$} of the form $G_z=G_z^+ + G_z^-$, where $G_z^\pm=\Onorm{12}$ satisfies the property that for each  $f\in \cE'(M)$  there exists a Cauchy surface $t_0\in\rr$ such that
\beq\label{supopo}
\supp G_{z}^{\pm}f \subset (\varphi^{-1})^*\{ \pm  t\geqslant  \pm t_0\},
\eeq
and furthermore,
\beq\label{halfwf}
\wfl{12}(G_z^\pm)\subset \{ (q_1,q_2)\in \Sigma^\mp\times \Sigma^\mp \, | \, q_1{\succ} q_2  \}\cup \ediag.
\eeq
\end{proposition}\beproof  A straightforward computation shows that the differential operator 
\beq 
Q(t,z)\defeq -c^2(t) (\varphi^* (P-z))
\eeq
is of the form
$$
Q(t,z)= D_t^2 + Q_0(t) D_t -Q_2(t,z),
$$
where $Q_0(t)=\p_t(c^{-1}(t) \left|h(t)\right|^{\12})\in\mix{0}$ is a multiplication operator and 
$$
Q_2(t,z)=  c(t) \left|h(t)\right|^{-\12}  \sum_{i,j=1}^{n-1}D_i c(t)h(t)^{ij}  \left|h(t)\right|^{\12}\!D_j- z \, c^2(t)  \in\miz{2}
$$
is elliptic with parameter. Our proof is divided into several steps.
  
\step{1} We claim that for each $i\in\nn_{\geqslant 0}$ there exists $A_i(t,z),B_i(t,z)\in\miz{1}$, each of them elliptic with parameter, and $R_i(t,z)\in\mimiz{1-i}{0}$, such that
\beq\label{facto}
Q(t,z)=(D_t - A_i(t,z))(D_t+B_i(t,z)) +   R_{i}(t,z).
\eeq
We show this inductively by adapting the arguments in \cite[\S23.2]{HormanderIII} and \cite[\S6]{GOW}  to our setting. 

 Namely, suppose that \eqref{facto} holds true for some $i\in\nn_{\geqslant 0}$. We set then
 \begin{gather}\label{parco}
 C_i \defeq -  R_i (A_i +B_i)^{\inv}, \quad
 L_{i}  \defeq R_i ( \one - (A_i +B_i)^{\inv}(A_i +B_i)),\\ \label{parco2}
 R_{i+1}  \defeq  [C_i,D_t]+[A_i, C_i ] + C_i^2 + L_i,      
 \end{gather}
 where $(A_i +B_i)^{\inv}\in\miz{-1}$  is an elliptic parametrix of $A_i +B_i$, and the dependence on $t,z$ is disregarded in the notation. Using \eqref{compopdo}, we obtain 
 $$
 C_i,  L_{i} , R_{i+1}  \in  \mimiz{1-(i+1)}{0}.
 $$
These operators are defined in (\ref{parco})--(\ref{parco2}) in such way that they satisfy the identities
 $$
 C_i B_i + R_i =  -C_i A_i + L_i, \quad C_i (D_t-A_i) + L_i = (D_t-A_i-C_i) R_{i+1},
$$
 which entail
$$
(D_t - A_i)(D_t+B_i) +   R_{i}=(D_t - A_i-C_i)(D_t+B_i+C_i) +   R_{i+1}.
$$
Thus, by setting $A_{i+1}\defeq A_{i}+C_i\in \miz{1}$ and $B_{i+1}\defeq B_{i}+C_i \in \miz{1}$ we conclude that \eqref{facto} holds true for $i+1$ in place of $i$.

Now, to show \eqref{facto} it remains to check the induction hypothesis $i=0$. To that end we set 
$$
A_0(t,z)\defeq (Q_2(t,z))^\hinv-\12 Q_0(t), \quad B_0(t,z)\defeq (Q_2(t,z))^\hinv+\12 Q_0(t),
$$
where $Q_2^\hinv$ is an approximate square root obtained from the poly-homogeneous expansion of $Q_2$ in the parameter-dependent sense. By construction, $Q_2^\hinv,A_0,B_0\in \miz{1}$. Furthermore,
$$
Q(z,t)=(D_t - A_0(t,z))(D_t+B_0(t,z)) +   R_{0}(t,z),
$$
where
\beq\label{rotoe}
R_0=\12 [Q_0,D_t]  +\frac{1}{4}Q_0^2 + \big[Q_2^\hinv,D_t\big]  \mbox{ mod } \miz{-\infty}.
\eeq
We want to show that $R_0\in\mimiz{1}{0}$. The first two terms on the l.h.s.~of \eqref{rotoe} clearly belong to that space as they  are $z$-independent. The third term equals 
\beq\label{poek}
\bea
i\big[\p_t,Q_2^\hinv\big]&=-\12\p_t(Q_2)Q_2^{\mhinv} \mbox{ mod } \miz{-\infty}, \\
&=-\12\p_t(Q_2)U^{-1}UQ_2^{\mhinv} \mbox{ mod } \miz{-\infty}, 
\eea
\eeq
where $Q_2^{\mhinv}\in \miz{-1}$ is an elliptic parametrix of $Q_2^{\hinv}$ and  $U\in \mix{1}$ is chosen elliptic and invertible. Then $U^{-1}\in\mix{-1}$ and  we have 
$$
(\p_t Q_2)U^{-1}\in \mix{1}, \quad  UQ_2^{\mhinv} \in \miz{0}, 
$$
where the second estimate is  crude, but sufficient for our purpose. We conclude that the operator in \eqref{poek} is in $\mimiz{1}{0}$ as requested.

\step{2}   We have proved in \emph{Step 1} that there exist  $A,B\in\miz{1}$ elliptic with parameter, and $R\in\mimiz{-s-1}{0}$ such that 
\beq\label{facto1}
Q(t,z)=(D_t - A(t,z))(D_t+B(t,z)) +   R(t,z).
\eeq
We can repeat the construction in \emph{Step 1} with the rôle of $A_0$ and $B_0$ reversed to  obtain $\tilde A,\tilde B\in\miz{1}$ elliptic with parameter, and $\tilde R\in\mimiz{-s-1}{0}$ such that
\beq\label{facto2}
Q(t,z)=(D_t+\tilde B(t,z))(D_t - \tilde A(t,{z})) +   \tilde R(t,{z}).
\eeq
{For the purpose of proving our main assertion, without loss of generality  we can assume that $Y=\rr^d$, and that  $\tilde A(t,z),B(t,z)\in C^\infty(\rr;\Psi^1_{{Z}}(\rr^d))$ take value in the uniform pseudo-differential class  on $\rr^d$. In fact, at each time $t\in \rr$ we can modify the definition of $Q(t,z)$ outside of a compact set and then apply the construction from \emph{Step 1} to the modified operator. This has the effect of adding   an extra error term on the r.h.s.~of \eqref{facto1} and \eqref{facto2}, but this error term vanishes outside of $M_0$ and  for this reason it will  be of no relevance in the rest of the proof (as we only want a parametrix in $M_0$).  We choose to  disregard it for the sake of simplicity of notation.  }

{In the simplified situation with $Y=\rr^d$, we want to show that $D_t - \tilde A(t,z)$ has an advanced inverse acting on $H^{-N}_\c(\rr\times \Cauchy)$, denoted in the sequel by $U_{\tilde A}^-(z)$, and  $D_t+B(t,z)$ has a retarded inverse $U_{-B}^+(z)$. This means  that $U_{\tilde A}^-(z)$ and $U_{-B}^+(z)$ are left inverses on $H^{-N}_\c(\rr\times \Cauchy)$ of the respective operators,  and for all $f\in H^{-N}_\c(\rr\times \Cauchy)$ there exists $t_0\in\rr$ such that
\beq\label{prom0}
\supp U_{\tilde A}^-({z}) f \subset \{t\leqslant t_0 \}, \quad \supp U_{-B}^+({z}) f \subset \{t\geqslant t_0 \}.
\eeq
Modulo a reparametrisation of the  time interval, the existence of $U_{-B}^+(z)$ follows from well-posedness of the inhomogenous Cauchy problem for $t\in [0,T]$ (with $T>0$ arbitrary) shown in \cite[Thm.~23.1.4]{HormanderIII}, provided the assumptions are verified.  To that end we need to check the boundedness from below condition:
\beq\label{eq:condi}
-\Im B(t,z)  \geqslant  -C \one \mbox{ on }  H^1(\rr^d)
\eeq
for all $t\in [0,T]$, where the constant  $C >0$  depends only on $T$ (we use here the notation $\Im B=\12 (B+B^*)$ and $\Re B=\12 (B-B^*)$). Indeed, by a direct computation we find 
\beq\label{eq:condi2}
\Im B=  C^*_{1/4}  (\Im Q )C_{{1/4}}+ C_0,
\eeq
where  $C_{{1/4}}(t,z) \in C^\infty(\rr;\Psi^{1/4}_{{Z}}(\rr^d))$ is an approximate square root of $\Re Q^{\hinv}(t,z)$ and $C_0(t,z)\in C^\infty(\rr;\Psi^{0}_{{Z}}(\rr^d))$, with values in the uniform pseudo-differential class on $\rr^d$. In view of  
$$
-\Im Q (t,z) = - (\Im z)  c^2(t)\geqslant 0,
$$
\eqref{eq:condi2} implies  \eqref{eq:condi} with $C=  \sup_{z\in Z}\sup_{t\in [0,T] }\| C_0(t,z) \|<+\infty$. The existence of  $U_{\tilde A}^-({z})$ is shown analogously, with  obvious sign changes\footnote{Note that for the advanced problem we need to solve the inhomogeneous Cauchy problem backwards, so the analogue of \eqref{eq:condi} has reversed sign, hence the necessity of considering $\tilde A$ in place of $-B$.}. Furthermore, thanks to the fact that the constant $C$ does not depend on $z$, the proof of \cite[Thm.~23.1.4]{HormanderIII} implies that}
\beq\label{prom}
U_{\tilde A}^-({z})= \Oreg[]{0}, \quad U_{-B}^+({z})=\Oreg[]{0},
\eeq
where  the notation refers to mapping properties $H^l_\c(\rr\times\Cauchy)\to H^l_\loc(\rr\times \Cauchy)$ for all $l\in\rr$, uniformly in $z$.
%In fact, \emph{forward} in time energy estimates for $D_t - \tilde A(t,{z})$ become stronger as $\Im z \geqslant 0$ becomes larger, and similarly for \emph{backward} energy estimates for $D_t+B(t,z)$. 

Next, using \eqref{facto1}--\eqref{facto2} we compute
$$
\bea
Q \big(  U_{\tilde A}^- - U_{-B}^+ \big)&=   \big((D_t+\tilde B)(D_t - \tilde A)+  \tilde R\big) U_{\tilde A}^- - \big((D_t - A)(D_t+B) +   R\big)U_{-B}^+\\
&=    \tilde B+A +  \tilde R U_{\tilde A}^- +R U_{-B}^+.
\eea
$$
If now $(\tilde B+A)^{\inv}\in\miz{-1}$ is an elliptic parametrix of $\tilde B+A$, we conclude that 
\beq\label{prom2}
U_{\tilde A}^- (\tilde B+A)^{\inv} \mbox{ and } \, U_{-B}^+(\tilde B+A)^{\inv} \mbox{ are }\, \Oreg[]{1}.
\eeq
and
$$
Q\big(  U_{\tilde A}^- - U_{-B}^+ \big)(\tilde B+A)^{\inv}=\one + E,
$$
where $E = \Oreg{0}$. 

\step{3} The wavefront sets of \eqref{prom} can be estimated by a variant of Egorov's theorem. More precisely, let us first show
\beq\label{tempokj}
\wfl{0}\big(U_{-B}^+(z))\big)\subset \{ (q_1,q_2)\in \Sigma^- \times \Sigma^-  \, | \, q_1\sim  q_2  \}\cup \ediag
\eeq
Let $q_1,q_2\in \cosb$ with base points $(t_1,y_1)\neq (t_2,y_2)$. If $q_1\notin \Sigma^-$ or $q_2\notin \Sigma^-$ then $(q_1,q_2)\notin \wfl{0}\big(U_{-B}^+(z)\big)$ by microlocal ellipticity. Consider now $q_1=(t_1,y_1;\tau_1,\eta_1)\in \Sigma^-$. By the arguments in the proof of \cite[Thm.~23.1.4]{HormanderIII} there exists $S(t)\in\mix{0}$ such that
\beq\label{commut}
[D_t+B(t,z),S(t)]\in \mix{-\infty}, 
\eeq
$S(t_1)\in \Psi^0(Y)$ is elliptic at $(y_1;\eta_1)$ and $\wf'(S(t))$ is a neighborhood of $\Phi^{t-t_1}\big((y_1;\eta_1)\big)$.
Consider the tensor product operator $S\otimes \one$ acting on $M=\mathbb{R}\times Y$~\footnote{Note that the tensor product of two pseudodifferential operators is not necessarily in the usual calculus.}. Furthermore, if $q_1\nsim q_2$ then that neighborhood  can be chosen in such way that $\left(S\otimes \one\right) S_2^*\in \Psi^{-\infty}(M)$ is smoothing for some $S_2\in\Psi^0(M)$ elliptic at $q_2$ since $\eta_2\neq 0 $ with the symbol of $S_2$ vanishing in some conical neighborhood of $\eta=0$  by~\cite[Thm.~18.1.35 p.~94]{HormanderIII}. In view of $q_1\in \Sigma^-$ we have $\eta_1\neq 0$ and therefore we can find $T\in\Psi^0(M)$ such that $S_1\defeq T \circ \left(S\otimes \one\right)\in \Psi^0(M)$ and $S_1$ is elliptic at $q_1$ again by~\cite[Thm.~18.1.35 p.~94]{HormanderIII}, the symbol of $T$ is also chosen to vanish in some conical neighborhood of $\eta=0$~\footnote{We have followed the method of H\"ormander~\cite[p.~390]{HormanderIII} to convert space to spacetime wavefront bounds.}.
Using \eqref{prom}, \eqref{commut} and the fact that $\left(S\otimes \one\right) S_2^*$ is smoothing, we obtain
$$
\bea
S_1 U_{-B}^+(z) S_2^* = T\left(S\otimes \one\right) U_{-B}^+(z) S_2^*& = T  U_{-B}^+(z) \left(S\otimes \one\right) S_2^*+ {T   [\left(S\otimes \one\right),U_{-B}^+(z)]} S_2^* \\ & = \Oreg{0}.
\eea
$$
Since $S_i$ is elliptic at $q_i$ this shows that $(q_1,q_2)\notin \wfl{0}\big(U_{-B}^+(z)\big)$, and in this way  we get    \eqref{tempokj}.

Using the support properties \eqref{prom0} we can improve on $\eqref{tempokj}$ and eliminate points $(q_1,q_2)$ in the wavefront {set} such that, writing $q_i=(x_i;\xi_i)$, $x_1$ is in the past of $x_2$. We can also observe that for $(q_1,q_2)\in \Sigma^+\times \Sigma^+$ (resp.~$\Sigma^-\times \Sigma^-$) with $q_1\sim q_2$,    $x_1$ is in the past of $x_2$ (resp.~in the future) if and only if $q_1{\succ}  q_2$. Therefore, we obtain that
$$
\wfl{0}\big(U_{-B}^+(z)\big)\subset \{ (q_1,q_2)\in \Sigma^-\times \Sigma^- \, | \, q_1{\succ} q_2  \}\cup \ediag.
$$
In an analogous way we prove
$$
\wfl{0}\big(U_{\tilde A}^-({z})\big)\subset \{ (q_1,q_2)\in \Sigma^+\times \Sigma^+ \, | \, q_1{\succ} q_2  \}\cup \ediag.
$$
Since $(\tilde B+A)^{\inv}\in\miz{-1}$ it follows that 
\beq\label{halfwf2}
\bea
\wfl{12}\big(U_{-B}^+ (\tilde B+A)^{\inv}\big)\subset \{ (q_1,q_2)\in \Sigma^-\times \Sigma^- \, | \, q_1{\succ} q_2  \}\cup \ediag, \\
\wfl{12}\big(U_{\tilde A}^- (\tilde B+A)^{\inv}\big)\subset \{ (q_1,q_2)\in \Sigma^+\times \Sigma^+ \, | \, q_1{\succ} q_2  \}\cup \ediag.
\eea
\eeq

We have therefore constructed a parametrix with properties \eqref{prom0}, \eqref{halfwf2} and \eqref{prom2} analogous to the ones in the statement of the proposition, but for the auxiliary operator $Q(t,z)$ instead of $P-z$.

\step{4} It now remains to reformulate the parametrix construction in terms of $P-z$. Recall that $P$ and $Q$ are related by $Q=-c^2  (\varphi^* (P-z))$, so by setting
$$
\bea
G_z^-&\defeq  - (\varphi^{-1})^* \big(U_{\tilde A}^-({z})(\tilde B+A)^{\inv}(t,z)\big) c(t)^{-2},\\
G_z^+&\defeq - (\varphi^{-1})^* \big(-U_{-B}^+({z})(\tilde B+A)^{\inv}(t,z)\big)c(t)^{-2},
\eea
$$
and $G_z\defeq G_z^+ + G_z^{-}$ we obtain a parametrix for $P-z$ with the desired properties.  
  \qed

 \subsection{Wavefront set of the resolvent}\label{ss:wavefront}   We now proceed to estimate the uniform wavefront set of $(P-z)^{-1}$. Recall that $\gamma_\varepsilon$ is the contour in the complex upper half-plane defined in \sec{ss:restocp}, cf.~Figure \ref{fig:contour}. 

\begin{lemma}\label{lem:pika} Assume that $(M,g)$ is non-trapping and $\varepsilon>0$.  Then  $\{ (P-z)^{-1}\}_{z\in\gamma_\varepsilon}$ satisfies 
\beq\label{resolventbehaves}
(P-z)^{-1}=\Onorm{1}.
\eeq
Assuming in addition non-trapping at energy $\sigma=m^2\neq 0$ and injectivity,  \eqref{resolventbehaves} holds  also true for $\{ (P-z)^{-1}\}_{z\in\gamma_0}$.
\end{lemma}
\beproof By \eqref{soboma} with $N=0$, $(P-z)^{-1}$ and $(P-\bar{z})^{-1}$ are  $\cO(\left|\Im z\right|^{-\12})$ for $z\in\cc\setminus \rr$ as bounded maps $H^l_{\c}(M)\to H^{l+\12}_{\loc}(M)$ for all $l\geqslant 0$.  The analogous claim for $l$ negative follows by duality. Finally, the $\gamma_0$ case is shown similarly using \eqref{est3p} to get control in  $z$ down to the real axis.  \qeds

We now state the key lemma. 

\begin{lemma}\label{keylemma} Assume that $(M,g)$ is globally hyperbolic, non-trapping, and let $\varepsilon>0$. {For each bicharacteristic $\gamma$ there exists $U_\gamma\subset M$  such that} if  $G_z$,  $z\in Z$, is  a uniform parametrix {in $U_\gamma$} as in Proposition \ref{painful}, then for all  $s\in\rr$
$$
\wfl{12}\big((P-z)^{-1}-G_{z}\big) \cap (\gamma \times \cosb) \subset \ediag.
$$
\end{lemma}
\beproof {Let $U_\gamma$ be a small enough neighborhood of the base projection of $\gamma$ so that for each $t\in \rr$,  $\varphi^* U_\gamma \cap (\{ t \} \times Y)$ is included in a compact set. Then Proposition   \ref{painful} yields a uniform parametrix  $G_z=G_z^++G_z^-$  in $U_\gamma$, where $G_z^\pm$ solves a pseudo-differential retarded/advanced problem.}  Let $L_{-}^\pm=L_-\cap \Sigma^\pm$ be the future/past component of the sources $L_-$. By Fredholm estimates, i.e.~Proposition \ref{prop1}, for any $s\in\rr$ and any bounded subset $\cU\subset H^{s-1}_{\rm c}(M)$, the set $(P-z)^{-1} \cU$ is uniformly bounded in $\Hsc{s,\ell}$ for arbitrary $s\in\rr$ and for some $\ell$ with $\ell_->-\12$, thus in particular with $\ell>-\12$ in a neighborhood of $L_{-}^\mp$. By support properties of $G_{z}^{\pm}$, i.e.~by \eqref{supopo}, $G_{z}^{\pm}\cU$ {vanishes in the far past/future.  Therefore,} $G_{z}^{\pm}\cU$ is uniformly bounded in $\Hsc{s,\ell}$ (after possibly modifying the definition of $\ell$ outside of a neighborhood of $L_{-}^\pm$).

  We can therefore apply the higher decay radial estimate (Proposition \ref{radial1}) to the family $(P-z)^{-1}-G_{z}^{\pm}$, which is a uniform bi-solution of $(P-z)$ microlocally near ${\gamma\cap}\Sigma^\mp$.  This allows us to  conclude that 
$
B^{\pm} \big( (P-z)^{-1}-G_{z}^{\pm} \big)  \cU
$
is $\cO_{\cf}(\bra z\ket^{-\12})$ for some $B^\pm\in \Psi^{0,0}_\sc(M)$ elliptic on $L_{-}^\mp$.  Thus, { in $\gamma \times \cosb$}, $L_{-}^\mp\times \cosb$ is disjoint from $\wfl{1}\big( (P-z)^{-1}-G_{z}^{\pm} \big)$.  By the non-trapping assumption and propagation of singularities, the whole flowout of $L_{-}^\mp\times \cosb$  (in the first variable) within $\Sigma^\mp$ is disjoint from $\wfl{1}\big( (P-z)^{-1}-G_{z}^{\pm} \big)$  { in $\gamma \times \cosb$}. This means that 
$$
\wfl{1}\big( (P-z)^{-1}-G_{z}^{\pm} \big)\subset \Sigma^\pm\times \cosb
$$ 
 { in $\gamma \times \cosb$}. We now combine this with \eqref{halfwf} to conclude:
$$
\bea
\wfl{12}\big( (P-z)^{-1}-G_{z}\big)&\subset \wfl{12}\big( (P-z)^{-1}-G_z^\pm\big)+\wfl{12}\big(G_{z}^{\mp} \big)\\ &\subset (\Sigma^\pm \times\cosb)\cup \ediag
\eea
$$
{ in $\gamma \times \cosb$}. Since $\Sigma^+\cap \Sigma^-=\emptyset$, this implies the assertion of the lemma.
\qed

\begin{theorem}\label{thm:wf} Assume that $(M,g)$ is non-trapping,  globally hyperbolic, and let $\varepsilon>0$. Then for any $s\in\rr$,  the family  $\{(P-z)^{-1}\}_{z\in \gamma_\varepsilon}$ satisfies
\beq\label{feynwf2}
\wfl{12}  \big( ( P-z)^{-1} \big)\subset \{  (q_1,q_2)\in \Sigma\times \Sigma \, | \, q_1 {\succ} q_2 \}\cup \ediag.
\eeq
Moreover, suppose that $H_z$ is a local uniform parametrix of order $s$ for $P-z$ in the sense of Definition \ref{defpara}, and $H_z$ also satisfies \eqref{feynwf2}. Then for all $x\in M$ there exists $\chi\in \ccf$ with $\chi(x)=1$ such that  
\beq\label{feynwf3} 
\chi (P-z)^{-1} \chi  = \chi  H_z \chi + \Oreg{12}.
\eeq
\end{theorem}
\beproof The estimate \eqref{feynwf2} follows now directly from Lemma \ref{keylemma} {applied to all bicharacteristics $\gamma$} and  the fact that 
\beq\label{halfwfr2}
\wfl{12}(G_z)\subset \{ (q_1,q_2)\in \Sigma\times \Sigma\, | \, q_1{\succ} q_2  \}\cup \ediag
\eeq
by Proposition \ref{painful}. The second assertion follows directly from  \eqref{feynwf2} and Proposition \ref{uniqueness}.
\qeds
 
We will  show  that  a local uniform parametrix of arbitrarily high order can be obtained by a $z$-dependent variant of the Hadamard parametrix construction. 

The result {\eqref{feynwf3}} is satisfactory for many purposes, we remark however that  it does not give stronger decay of the error term on the r.h.s.~even if $(P-z)H_z-\one$ has better decrease in $z$. {For this reason} in \sec{sec:gluing} we will use a more precise composition argument. 

\ber  Assuming in addition non-trapping and injectivity at energy $\sigma=m^2\neq 0$,  Lemma \ref{keylemma}  and Theorem \ref{thm:wf}  hold also true for $\gamma_0$ instead of $\gamma_\varepsilon$.
\eer

\ber\label{rem:vector1} All the results in \secs{sec:complexpowers}{sec:wf}  generalize in a straightforward way to the case when $P$ is a principally scalar wave operator on a finite-dimensional Hermitian bundle $E$, provided that $P$ is formally self-adjoint for the  canonical scalar product induced by the Hermitian form on fibers and by the volume form. We stress that this requires to have a scalar product which is in particular \emph{positive}.
In more general situations such as the wave equation on tensors, the propagation estimates need to be modified, see e.g.~\cite{Hintz2017}. 
\eer
\section{The elementary family \texorpdfstring{$\Fs{z}$}{F(z)}}\label{sec:elementary}

\subsection{Definition of the family \texorpdfstring{$\Fs{z}$}{F(z)}}\label{ss:elementary} In this section we  define a family $\Fse{z}$ of distributions on $\rr^n$ which is the first ingredient in the Hadamard parametrix construction. We analyze its regularity properties and its dependence on the complex parameter $z$.  More precisely, we control the wavefront set uniformly in $z$ along the contour $\gamma_\varepsilon$ defined in \sec{ss:restocp}. We also study the {H\"older} regularity asymptotically in the parameter $z$ on the upper half-plane $\{\Im z>0\}$ down to $z\in \mathbb{R}\setminus \{0\}$.

Let $\cv\in\cc$. When writing  complex powers we always use the usual branch of the $\log$ defined on $\mathbb{C}\setminus \opencl{-\infty,0}$.
For $\Im z> 0$, we define the distribution in the $x\in \mathbb{R}^{n}$ variable 
\begin{equation}\label{e:defF}
\Fs{z}=\frac{\Gamma(\cv+1)}{(2\pi)^{n}} \int e^{i\left\langle x,\xi \right\rangle}\left(\vert\xi\vert_\eta^2-z \right)^{-\cv-1}d^{n}\xi
\end{equation}
in the sense of an inverse Fourier transform, where $\eta=dx_0^2-(dx_1^2+\cdots+dx_{n-1}^2)$ is the flat Minkowski metric, and $
\vert\xi\vert_\eta^2 = -\xi \cdot \eta^{-1} \xi=  -\xi_0^2+\sum_{i=1}^{n-1}\xi_i^2
$ is defined for convenience with a \emph{minus} sign. The distribution \eqref{e:defF} is Lorentz invariant.

Next, we extend the definition \eqref{e:defF} to $z\in\rr\setminus\{0\}$. To that end we   define  the family of distributions $\left(\vert\xi\vert_\eta^2-z-i0 \right)^{-\cv-1}$ corresponding to taking the limit of  $\left(\vert\xi\vert_\eta^2-z\right)^{-\cv-1}$ as $\Im z\rightarrow 0^+$. More precisely, denoting $Q(\xi)=\vert\xi\vert_\eta^2$, for $z\in \mathbb{R}$  we define as in \cite[III, \S2.4]{Gelfand-ShilovI},
$$ 
\left(Q(\xi)-z-i0\right)^{-\cv}=\lim_{\varepsilon\rightarrow 0^+} (Q(\xi)-z-i\varepsilon)^{-\cv},
$$
considered first as a distribution on $\rr^n\setminus \{0\}$.

\begin{prop}\label{prop:flatfeynm}
The family of distributions $\{ \left(Q(\xi)-i0\right)^{-\cv}\}_{\cv\in \mathbb{C}}$ is well-defined
on $\mathbb{R}^{n}\setminus \{0\}$ by pull-back. It extends homogeneously to
 $\mathbb{R}^{n}$ as a meromorphic family in $\cv\in \mathbb{C}$
with simple poles contained in $\mathbb{N}+\frac{n}{2}$. The residues at the poles are distributions 
supported at $0\in \mathbb{R}^{n}$. 

On the other hand, if $z\in \mathbb{R}\setminus \{0\}$, then
$\{ \left(Q(\xi)-z-i0\right)^{-\cv}\}_{\cv\in \mathbb{C}}$ is a {holomorphic} family of distributions on $\rr^n$.
\end{prop}

\begin{proof} The meromorphic family of 
distributions $(t-i0)^{-\cv}$ in $\pazocal{S}^\prime(\mathbb{R})$ has singular support only at $t=0$. 
Observe that along the cone $Q=0$, we have
$dQ(\xi)\neq 0$ when $\xi\neq 0$. Therefore,
the pull-back $Q^*(t-i0)^{-\cv}$ is well-defined on $\mathbb{R}^{n}\setminus\{0\}$
with wavefront set contained in $\{ (\xi;\widehat{\xi}) \, | \, Q(\xi)=0, \, \widehat{\xi}=\tau dQ(\xi), \ \tau<0  \}$ by the pull-back theorem \cite[Thm.~8.2.4 p.~263]{H}, see also~\cite[(8.2.6) p.~265]{H}.
The distribution $Q^*(t-i0)^{-\cv}$ is homogeneous of degree $-2\cv$ hence
by  \cite[Thm.~3.2.3 p.~75]{H}, it has a unique extension as a holomorphic family of distributions
in $\cv\in \mathbb{C}\setminus \{0,1,\dots,n,\dots\}$ defined on $\mathbb{R}^{n}$. {It has poles which are} 
contained  in $\{0,1,\dots,n,\dots\}$ by \cite[Thm.~3.2.4]{H}. 
Furthermore, \cite[III, \S2.4]{Gelfand-ShilovI} tells us that the poles are simple, they  are actually contained in $\mathbb{N}+\frac{n}{2}$~\cite[p.~275]{Gelfand-ShilovI} and the residues are derivatives of $\delta_{0}$~\cite[$(1), (1)^\prime$ p.~276]{Gelfand-ShilovI}. 

In the case of $\left(Q(\xi)-z-i0\right)^{-\cv}_{\cv\in \mathbb{C}}$, we  start from the holomorphic family of distributions 
$(t-z-i0)^{-\cv}$ which has singular support at $t=z\neq 0$. The difference is that  the pull-back by the map 
$$  \mathbb{R}^{n}\ni\xi\mapsto Q(\xi)-z\in \mathbb{R} $$
can be applied everywhere since for all $\xi$ such that $Q(\xi)-z=0$ we have $dQ(\xi)\neq 0$.
\end{proof}

\begin{coro}
By inverse Fourier transform, 
$$
\Fs{z}=\frac{\Gamma(\cv+1)}{(2\pi)^{n}} \int e^{i\left\langle x,\xi \right\rangle}\left(\vert\xi\vert_\eta^2-i0-z \right)^{-\cv-1}d^{n}\xi
$$
is a well-defined family of distributions on $\rr^n$, holomorphic in $\cv\in\cc\setminus\{-1,\dots,-k,\dots\}$ for $z\in \{\Im z\geqslant 0\}\setminus \{0\}$.
\end{coro}

Thus, to regulate the infrared poles of the family $\Fs{0}$ one can  
introduce a mass $m>0$ and consider $\Fs{-m^2}$.

\subsection{H\"older estimate on \texorpdfstring{$\Fse{z}$}{F(z)}}

For large $\Re\cv$,
$\Fse{z}$ has Fourier transform $(Q-z)^{-\cv-1}$, which has good decay at infinity except along the light-cone, so the pressing question is can we control $\Fse{z}$ in Sobolev or H\"older spaces of high regularity? 
The answer is yes, but
the price to pay
is that we need to lose in terms of the decay in $z$.  
\emphasize{We trade decay in $z$ for regularity in the $x\in \mathbb{R}^{n}$ variable}.

\subsubsection{Estimates on $(Q-z)^{-\cv}$ as distributions.}
\label{ss:estimates1}
We first discuss the case of  $(Q(\xi)-z)^{-\cv}$ for integer $\cv\in \mathbb{N}$. 
 We start from the family
$\log(t-i\varepsilon)$ for $\varepsilon>0$, which is a well-defined distribution on $\mathbb{R}$.
Uniformly in $\varepsilon$, we have the estimate:
$\module{\left\langle \log(t-i\varepsilon),\varphi \right\rangle}\leqslant C_K\Vert\varphi\Vert_{L^2(\mathbb{R})}$
for all test functions $\varphi$ supported in a fixed compact set $K$. 
It follows that for all test functions $\varphi$ supported in a fixed compact set $K$ and for all integer $\cv\in \mathbb{N}$:
$$ \bea 
\module{ \left\langle  (t-i\varepsilon)^{-\cv},\varphi \right\rangle} =C_{\cv}\module{ \left\langle \partial_t^\cv \log(t-i\varepsilon),\varphi \right\rangle} 
=
C_{\cv}\module{ \left\langle  \log(t-i\varepsilon),\partial_t^\cv\varphi \right\rangle }
\leqslant CC_{\cv}\Vert\varphi\Vert_{H^{{\cv}}(\mathbb{R})} 
\eea $$
where the estimate still holds uniformly in $\varepsilon>0$. 
For large $\Im z>0$ and $\varphi$ supported in a fixed compact set $K$:
$$ \bea 
\module{ \left\langle  (t-z)^{-\cv},\varphi \right\rangle }&=\module{ \left\langle  (t-\Re z-i\Im z)^{-\cv},\varphi \right\rangle}\\
&=
\module{ \left\langle  (t-i\Im z)^{-\cv},\varphi(.-\Re z) \right\rangle } 
\leqslant \module{\Im z}^{-\cv} C_{K,\cv}\Vert\varphi\Vert_{L^2(\mathbb{R})}. 
\eea $$
The case of small $\Im z$ is handled by the previous estimates. 
So in general, for $\varphi\in C^\infty_\c(\mathbb{R}^{n})$ supported in a fixed compact set $K$, we have the estimate
\begin{eqnarray}\label{ineq:est1}
\module{ \left\langle  (t-z)^{-\cv},\varphi \right\rangle }
\leqslant \left(1+ \module{\Im z}\right)^{-\cv} C_K\Vert \varphi \Vert_{H^\cv(\mathbb{R})}
\end{eqnarray}
where $C_K$ does not depend on $z$ on the upper half-plane.
As before, let $Q$ be the quadratic form of signature $(n-1,1)$ for the Minkowski metric and let $\cv\in \mathbb{N}$. The pull-back
$Q^*(t-z)^{-\cv}=(Q-z)^{-\cv} $ is well-defined as a distribution of order $\cv$ in $\pazocal{D}^\prime(\mathbb{R}^{n}\setminus \{0\})$, uniformly in $\Im z> 0$ since
$dQ(\xi)\neq 0$ for all $\xi\neq 0$.
It follows
that for any compactly supported function $\chi$ supported in a compact {set} $K$ which does not intersect $0$, we have
$ \module{\left\langle  (Q(\xi)-z)^{-\cv},\chi\right\rangle}
\leqslant \left(1+ \module{\Im z}\right)^{-\cv} C_K \Vert \chi\Vert_{{H^{\cv}}(\mathbb{R}^n)}  $
where the pull-back is well-defined.
For non-integer $\cv$, it suffices to start from $(t-z)^{-\cv}$ which is well-defined in $L^1_{\loc}(\mathbb{R})$ for $\Re\cv<1$ hence defines a
holomorphic family of distributions of order $0$ in $\pazocal{D}^\prime(\mathbb{R})$ in the half-plane $\Re\cv<1$.
This description is \emphasize{uniform} in $z\in \{\Im z>0\}$. 
Then, to extend to all $\cv\in \mathbb{C}\setminus \mathbb{Z}$, for $k < \Re\cv<k+1 $, we use successive integration by parts:
$$
  (t-z)^{-\cv}= \frac{1}{(-\cv+k)\dots(-\cv+1) } \p_t^k (t-z)^{-\cv+k}
  $$
for $k=\plancher{\Re\cv}$, which shows that the l.h.s.~is a well-defined holomorphic family of distributions of order $k$,  
\emphasize{uniformly} in $z\in \{\Im z>0\}$. Again by pull-back, this shows that 
for any compactly supported function $\chi$ supported in a compact {set} $K$ which does not meet $0$, we have
$$ \module{\left\langle  (Q(\xi)-z)^{-\cv},\chi\right\rangle}
\leqslant \left(1+ \module{\Im z}\right)^{-\Re\cv} C_K \Vert \chi\Vert_{H^{\plancher{\Re\cv}}(\mathbb{R}^n)}   $$
where the pull-back is well-defined.

\subsubsection{The H\"older--Zygmund estimate on $\Fse{z}$.}
\label{ss:holderestimates}

In this paragraph, we  deal with  Euclidean harmonic analysis
of the holomorphic family $\Fse{z}\in \pazocal{D}^\prime(\mathbb{R}^n)$.

Recall that the Littlewood--Paley decomposition starts from a partition 
of unity $1= \chi_0+\sum_{j=0}^\infty \chi(2^{-j}.)$.
A function $u$ belongs to the \emph{Zygmund class} $\pazocal{C}^r(\rr^n)$~\cite[p.~294]{meyer1981regularite}~\cite[\S 8.6 p.~201]{Hormander-97}~\cite[\S 8 p.~40]{taylorpartial3} iff
\begin{equation}\label{plfj}
 \Vert \chi_0(2^{-j}\sqrt{-\Delta}) u\Vert_{L^\infty}+\sup_{j} 2^{jr} \Vert \chi(2^{-j}\sqrt{-\Delta}) u\Vert_{L^\infty}<+\infty,
\end{equation}
and this also defines a Banach norm $\Vert .\Vert_{\pazocal{C}^r}$ on $\pazocal{C}^r(\rr^n)$ (if $r\geqslant 0$ is not an integer then
$\pazocal{C}^r(\rr^n)$ coincides with the usual H\"older class). The local version of $\pazocal{C}^r(\rr^n)$ is denoted by $\pazocal{C}^r_\loc(\rr^n)$. The equivalence of \eqref{plfj} with a Fourier transform characterization is recalled in \sec{sss:Holderfourierdecay} in the appendix.

We will use the dyadic decomposition
to analyze the family of distributions $\Fse{z}$. 
For $\psi\in C^\infty_\c(\mathbb{R}^n)$, we estimate the norm of $\psi\chi(2^{-j}\sqrt{-\Delta}) F_\cv $ for $\Im(z)\geqslant 0$,  namely:
\beq\label{thenormofjd}
\bea
&\Big\| \psi(x) \int_{\mathbb{R}^n} (Q(\xi)-z)^{-\cv-1} \chi(2^{-j}\vert\xi\vert) e^{ix.\xi}  d^n\xi\Big\|_{L^\infty_x} \\  &=
2^{jn}\Big\| \psi(x) \int_{\mathbb{R}^n} 2^{-2j(\cv+1)}(Q(\xi)-2^{-2j}z)^{-\cv-1} \chi(\Vert\xi\Vert) e^{i(2^jx).\xi}  d^n\xi\Big\|_{L^\infty_x}.
\eea
\eeq
Note the important $2^{-2j}z$ term which explains why at high frequencies, even if $z$ has large imaginary part,
the dyadic scaling will push $2^{-2j}z$ arbitrarily close to the real axis so that for large $j$,
$(Q(\xi)-2^{-2j}z)^{-\cv}$ behaves more and more like the distribution $(Q(\xi)-i0)^{-\cv}$.
 For $k=\plancher{\Re\cv}+1$,  by  \eqref{thenormofjd} we find:
$$\bea
\Vert \chi(2^{-j}\sqrt{-\Delta}) \Fse{z}\psi\Vert_{L^\infty}
&=2^{j(n-2\Re\cv-2)}\Vert \psi(x) \int_{\mathbb{R}^n} (Q(\xi)+2^{-2j}z)^{-\cv-1} \chi(\Vert\xi\Vert) e^{i(2^j x).\xi}  d^n\xi\Vert_{L_x^\infty}\\ 
&\leqslant 
2^{j(n-2\Re\cv-2)}\left(1+ 2^{-2j}\module{\Im z}\right)^{-\Re\cv-1} \fantom \times C \sup_{x\in \supp \psi} \Vert \chi(\Vert \xi\Vert) e^{i(2^jx).\xi}\Vert_{H^k_\xi(\mathbb{R}^n)} \\
&\leqslant C_1
2^{j(n-2\Re\cv-2)}\left(1+ 2^{-2j}\module{\Im z}\right)^{-\Re\cv-1} (1+2^jR)^k\\
&\leqslant C_2 2^{j(n-2\Re\cv+k-2)}\left(1+ 2^{-2j}\module{\Im z}\right)^{-\Re\cv-1}.
\eea $$
In the last two inequalities, we made crucial use of the fact 
that $\chi$ is supported in a  compact ball $\{\vert \xi \vert\leqslant R\}$
and also that the support of $\psi$ is compact so that we  have the simple bound
$\sup_{x\in \supp \psi} \Vert\chi(\Vert\xi\Vert) e^{i(2^jx).\xi}\Vert_{H^k_\xi(\mathbb{R}^n)} \lesssim  (1+2^jR)^k$. 

Let us now {interpolate} the above inequality to show the interplay between decay in $\Im z$ and also decay in the dyadic scaling, which expresses H\"older regularity.
Choose some $a\in [0,1]$, then we get
$$ \bea 
\Vert \chi(2^{-j}\sqrt{-\Delta}) \Fse{z} \psi\Vert_{L^\infty}&\leqslant C_2 2^{j(n-2\Re\cv+k-2)}\left(1+ 2^{-2j}\module{\Im z}\right)^{-\Re\cv-1}\\ &\leqslant  
C_2 2^{j(n-2\Re\cv+k-2)}2^{2ja(\Re\cv+1)}\left(2^{2j}+ \module{\Im z}\right)^{-a(\Re\cv+1)} \fantom \times \left(1+ 2^{-2j}\module{\Im z}\right)^{-(1-a)(\Re\cv+1)}\\
&\leqslant C_2 2^{j(n-2\Re\cv+k-2+2a(\Re\cv+1))}\left(1+ \module{\Im z}\right)^{-a(\Re\cv+1)}. 
\eea $$

To estimate the low energy part $\psi\chi_0(\sqrt{-\Delta}) F_\cv$, we first need that  $\Im(z)\geqslant 0$, $\vert z\vert\geqslant \varepsilon>0$
to avoid the infrared pole in the massless case when $\cv$ is an {integer}. However, for $\cv\in \mathbb{C}\setminus(\mathbb{N}+\frac{n}{2})$, we can still {let $\varepsilon \rightarrow 0$ which means the real part of $z$ is allowed to vanish}. The element 
$\lim_{\Im z\rightarrow 0^+,\vert z\vert\geqslant \varepsilon} (Q(.)-z)^{-\cv-1}$
extends as a distribution weakly homogeneous of degree $-2\cv-2$
hence it extends by~\cite[\S 2]{Meyer} as a distribution of order $p=\plancher{2\Re\cv+2-n}+1$. This implies that: 
$$ \bea 
\Big\|  \psi(x) \int_{\mathbb{R}^n} (Q(\xi)-z)^{-\cv-1} \chi_0(\xi) e^{ix.\xi}  d^n\xi\Big\| _{L^\infty_x} 
\leqslant  C\Vert\psi \Vert_{L^\infty}(1+\module{\Im z})^{-\Re\cv-1} \sup_{x\in \supp \psi}\Vert \chi_0 e^{ix.\xi}  \Vert_{C^{p}_\xi(\mathbb{R}^n)}.
\eea $$

Now we can conclude the H\"older  regularity estimates of our family $\Fse{z}$ for $Im(z)\geqslant 0$, $\vert z\vert\geqslant \varepsilon>0$:
$$ \bea 
\Vert \Fse{z}\psi \Vert_{\pazocal{C}^\varalpha(\mathbb{R}^n)}&=\sup_{j\in \mathbb{N}} 2^{j\varalpha}\Vert \chi(2^{-j}\sqrt{-\Delta}) \Fse{z}\psi\Vert_{L^\infty}+\Vert \chi_0(\sqrt{-\Delta}) \Fse{z}\psi\Vert_{L^\infty}\\
&\leqslant C\left(1+ \module{\Im z}\right)^{-a(\Re\cv+1)} 
\eea $$ 
if $n-2\Re\cv+k-2+2a(\Re\cv+1)+\varalpha\leqslant 0 $, hence if $\varalpha\leqslant (2-2a)(\Re\cv+1)-k-n$. 
So if we want high H\"older regularity, we have to choose large $\Re\cv$.

{The proof also shows that for $\Re\cv>L, L\in \mathbb{R}$, the series $$\sum_{j=1}^\infty \chi(2^{-j}\sqrt{-\Delta}) \Fse{z}\psi+\chi_0(\sqrt{-\Delta}) \Fse{z}\psi$$ converges absolutely in $\pazocal{C}^{ (2-2a-1)(L+1)-n}(\mathbb{R}^n) $ where each term is \emphasize{holomorphic} in $\alpha$. Therefore $\Fse{z}\psi$ is holomorphic in $\cv$ valued in the Banach space $ \pazocal{C}^{ (2-2a-1)(L+1)-n}(\mathbb{R}^n)$.}

We  summarize the estimates as follows for $\cv\in\cc\setminus\{-1,\dots,-k,\dots\}$.

\begin{prop}
Let $k=\plancher{\Re\cv}+1$ and $\Fse{z}\in \pazocal{D}^\prime(\mathbb{R}^n)$ as defined in (\ref{e:defF}).
%Then 
%$\Fse{z}\in \pazocal{C}_{\loc}^{ (2-2a)(\Re\cv+1)-k-n}(\rr^n)$ with  decay in $z$ of order $\pazocal{O}(\module{\Im z}^{-a(\Re\cv+1)})$ for $a\in [0,1]$. 
{For all $\varepsilon>0$, if $z\in \{\Im z\geqslant 0,\, \vert z\vert\geqslant \varepsilon\}$ then 
$\Fse{z}\in \pazocal{C}_{\loc}^{ (2-2a)(\Re\cv+1)-k-n}(\rr^n)$ with  decay in $z$ of order $\pazocal{O}((1+\module{\Im z})^{-a(\Re\cv+1)})$ for $a\in [0,1]$.}

 For $\Re\cv>L, L\in \mathbb{R}_{>-1}$, the family $\Fse{z} $ is holomorphic in $\cv\in \{\Re\cv>L\}$ with values in the Fr\'echet space $\pazocal{C}_{\loc}^{  (2-2a-1)(L+1)-n}(\rr^n)$, with decay in 
$z$ of order $\pazocal{O}(\module{\Im z}^{-a{(L+1)}})$ for $a\in [0,1]$. 
\end{prop}

As expected, we always have to trade regularity for decay in $\Im z$.

\subsection{Microlocal estimates}

To prove  microlocal bounds, we will need to represent the 
distribution $\Fse{z}$ as the sum of two oscillatory integrals which account for  the high  (${\rm UV}$) versus low (${\rm IR}$) frequency parts. Both require a careful treatment of the $\Im z\rightarrow 0^+$ limit.  

\subsubsection{Oscillatory integral representation formula}

We first prove the following important technical lemma.
 
\begin{lemm}\label{l:oscilldec}
Let $\psi\in C^\infty_\c(\mathbb{R};[0,1])$ be such that $\psi=1$ near $0$.
For $ -1< \Re\cv<0$ and $\Im z>0$, we have
$$
 \Fse{z} = I_{{\rm IR}}+I_{{\rm UV}},
$$
where
$$ \bea 
I_{{\rm IR}} &= \frac{e^{-i(\cv+1)\frac{\pi}{2}}}{(4\pi i)^{\frac{n}{2}}(-1)^{\frac{n-1}{2}}}\int_0^\infty    e^{u\frac{Q(x)}{4i}} e^{i\frac{z}{u}} \psi(u)  u^{\frac{n}{2}-\cv-2}du ,\\
 I_{{\rm UV}}&= \frac{e^{-i(\cv+1)\frac{\pi}{2}}}{(4\pi i)^{\frac{n}{2}}(-1)^{\frac{n-1}{2}}}\int_0^\infty   e^{u\frac{Q(x)}{4i}} e^{i\frac{z}{u}} (1-\psi)(u)  u^{\frac{n}{2}-\cv-2}du. 
\eea $$
Furthermore, in the sense of distributions $\pazocal{D}^\prime(\mathbb{R}^n\setminus\{0\})$, the $\cv,z$-dependent oscillatory integral  $I_{{\rm UV}}$  extends uniquely  to a holomorphic family in $\cv\in\cc$, uniformly in $z\in \{\Im z\geqslant 0\}$. The term $I_{{\rm IR}}$ extends uniquely as a distribution in $\pazocal{D}^\prime(\mathbb{R}^n)$,  depending holomorphically in 
$\cv$ in the half-plane $\Re\cv<\frac{n}{2}-1$, uniformly in $z\in \{\Im z\geqslant 0\}$.
\end{lemm}

The difficulty is in proving that the two oscillatory integrals
on the r.h.s.~have well-defined distributional limits for all $\cv\in\cc$. 

\begin{refproof}{Lemma \ref{l:oscilldec}}
We start from an elementary
representation formula for $\Im z>0$:
$$ \bea 
(Q(\xi)-z)^{-\cv-1}
=\frac{e^{-i(\cv+1)\frac{\pi}{2}}}{\Gamma(\cv+1)}\int_0^\infty  e^{-iu(Q(\xi)-z)} u^{\cv}du,
\eea $$
where $Q$ is the quadratic form of signature $(n-1,1)$.
When $\Im z>0$ and $\Re\cv>-1$, the right hand side converges
absolutely and is holomorphic in the variable $\cv$.
Let $\varphi$ be a  Schwartz  function. We 
study the following integral:
$$ \bea 
\int_0^\infty u^\cv\left(\int_{\mathbb{R}^n}e^{-iu(Q(\xi)-z)}\widehat{\varphi}(\xi)d^n\xi\right) du
=\int_0^\infty\int_{\mathbb{R}^n} \varphi(x)   \frac{(2\pi)^n e^{\frac{Q(x)}{4ui}}}{ (4\pi i)^{\frac{n}{2}} (-1)^{\frac{n-1}{2}}} e^{iuz} u^{\cv-\frac{n}{2}} d^nx \,du
\eea $$
where we used the Plancherel formula and the Fourier transform of complex Gaussians to obtain the last equality.
The integral on the r.h.s.~is well-defined for $\Re\cv-\frac{n}{2}>-1$.
So after change of variables in $u$, we get another oscillatory integral representation
for $\Re\cv>\frac{n}{2}-1$, $\Im z>0$:
\begin{equation}\label{e:oscillatory1}
\int_0^\infty u^\cv\left(\int_{\mathbb{R}^n}e^{-iu(Q(\xi)-z)}\widehat{\varphi}(\xi)d^n\xi\right) du
=C(\cv)\int_0^\infty  \left\langle e^{u\frac{Q(.)}{4i}},\varphi\right \rangle e^{i\frac{z}{u}} u^{\frac{n}{2}-\cv-2}du,
\end{equation}
where we absorbed all normalizations in the holomorphic constant $C(\cv)$ for simplicity, since they
play no rôle in the oscillatory bounds. Then we just use the test function $\psi$ to cut the integration in $u\in \mathbb{R}_{\geqslant 0}$ in two parts to separate the {\rm IR} and {\rm UV} problems:
$$ \bea 
\underset{\text{{\rm UV} part}}{\underbrace{\frac{e^{-i(\cv+1)\frac{\pi}{2}}}{(4\pi i)^{\frac{n}{2}}(-1)^{\frac{n-1}{2}}}\int_0^\infty   \left\langle e^{u\frac{Q(.)}{4i}},\varphi\right \rangle e^{i\frac{z}{u}} (1-\psi(u))  u^{\frac{n}{2}-\cv-2}du }}
\\+ \underset{\text{{\rm IR} part}}{\underbrace{\frac{e^{-i(\cv+1)\frac{\pi}{2}}}{(4\pi i)^{\frac{n}{2}}(-1)^{\frac{n-1}{2}}}\int_0^\infty   \left\langle e^{u\frac{Q(.)}{4i}},\varphi\right \rangle e^{i\frac{z}{u}} \psi(u)  u^{\frac{n}{2}-\cv-2}du}}.
\eea $$
The {\rm IR} part is well-defined for all values of $\cv$ and all $z$ s.t.~$\Im z>0$ since $e^{i\frac{z}{u}}=\pazocal{O}(u^\infty)$ as $u\rightarrow 0^+$. Another observation is that we can take the $\Im z\rightarrow 0^+$ limit
when $\Re\cv<\frac{n}{2}-1$, since $u^{\frac{n}{2}-\cv-2}$ is Riemann integrable near $u=0$ so there is no problem to let $\Im z\rightarrow 0^+$ and there are no constraints on the real part $\Re z$.

 Now we need to justify that the integral representation of the {\rm UV} part 
is well-defined as a distribution 
on the half-plane $\Im z\rightarrow 0^+$ which is holomorphic in 
$\cv\in \mathbb{C}$. 
For $\Re\cv>-1$ and $\Im z>0$ and for any test function $\varphi\in C^\infty_\c(\mathbb{R}^n\setminus\{0\})$:
$$ \bea 
\int_0^\infty   \left\langle e^{u\frac{Q(.)}{4i}},\varphi\right \rangle e^{i\frac{z}{u}} (1-\psi)(u) u^{\frac{n}{2}-\cv-2}du  &= \int_0^\infty \left\langle e^{\frac{uQ(.)}{4i}+i\frac{z}{u}}   , \left(\t L\right)^N\varphi\right\rangle (1-\psi)(u) u^{\frac{n}{2}-\cv-2}du 
\eea $$
where $L=\frac{(4i)\left\langle\nabla Q,\nabla\right\rangle }{u\Vert\nabla Q\Vert^2 }$ 
is a well-defined differential operator since the phase $dQ\neq 0$ on $\mathbb{R}^n\setminus\{0\}$, $N$ is an arbitrary integer
and the integral is holomorphic in $\cv$ on the half--plane
$\Re\cv>N+\frac{n}{2}-1$ since $\left(\t L\right)^N\varphi=\pazocal{O}(u^{-N})$ uniformly in $z\in \{\Im z\geqslant 0\}$.
\end{refproof}

Next, we make an observation on the large $\Im z$ behaviour of $\Fse{z}$ outside $\{0\}\subset \mathbb{R}^n$ which follows from the oscillatory integral representation.

\begin{lemm}\label{l:Fsoutsidediag}
For all $\cv\in\cc$, for all $\varphi \in C^\infty_\c(\mathbb{R}^n\setminus \{0\})$ and all $\Im z>0$, we have
$\langle \Fse{z},\varphi\rangle=\pazocal{O}\big(\module{\Im z}^{-\infty} \big)$.
\end{lemm} 
\begin{proof}
Let $\varphi\in C^\infty_\c(\mathbb{R}^n\setminus \{0\})$.
For the {\rm UV} part, outside $x=0$, set $L=4i\frac{\left\langle\nabla Q,\nabla\right\rangle}{u\Vert\nabla Q\Vert^2}$, then:
$$ \bea 
\bigg|\int_0^\infty   \left\langle e^{u\frac{Q(.)}{4i}},\varphi\right \rangle e^{i\frac{z}{u}}& (1-\psi(u))  u^{\frac{n}{2}-\cv-2}du\bigg| \\
&=\module{\int_0^\infty   \left\langle e^{u\frac{Q(.)}{4i}}, \left(\t L\right)^N\varphi\right \rangle e^{i\frac{z}{u}} (1-\psi(u))  u^{\frac{n}{2}-\cv-2}du}\\
&\leqslant C \int_0^\infty e^{-\frac{\Im z}{u}}  (1-\psi(u)) \underset{=\pazocal{O}(u^{\frac{n}{2}-\Re\cv-2-N})}{\underbrace{\Vert \left(\t L\right)^N\varphi \Vert_{L^\infty(\mathbb{R}^n)} u^{\frac{n}{2}-\Re\cv-2}}}du\\ 
&\leqslant  C \module{\Im z}^{\frac{n}{2}-\Re\cv-2-N} \int_0^\infty e^{-u^{-1}}  (1-\psi(u\Im z)) u^{\frac{n}{2}-\Re\cv-2-N} du,
\eea $$
where the integral on the r.h.s.~is uniformly bounded as $\Im z\rightarrow +\infty$. Therefore
we get that $I_{{\rm UV}}=\pazocal{O}_{\pazocal{D}^\prime(\mathbb{R}^n\setminus \{0\})}(\module{\Im z}^{-\infty})$~\footnote{To get large decay in $\Im z$ the distribution is viewed as an element of high order, we need to differentiate the test function $\varphi$ many times.}.

For the infrared part, one immediately deduces from the integral representation that
$$ \module{\int_0^\infty   \left\langle e^{u\frac{Q(.)}{4i}},\varphi\right \rangle e^{i\frac{z}{u}} \psi(u)  u^{\frac{n}{2}-\cv-2}du}\leqslant C\Vert \varphi\Vert_{L^\infty}\int_0^\infty \psi(u) e^{-\frac{\Im z}{u}}  u^{\frac{n}{2}-\Re\cv-2}du, $$ hence
$I_{{\rm IR}}=\pazocal{O}_{\pazocal{D}^\prime(\mathbb{R}^n)}(\module{\Im z}^{-\infty})$~\footnote{Here it is a distribution of order $0$
for all orders of decay in $\Im z$.}.
\end{proof}

\subsubsection{Bounds on the semi-norms $\Vert .\Vert_{N,V,\chi}$.}

Recall that for a closed conic set $\Gamma\subset \coto{\mathbb{R}^{n}}$, the topology of $\pazocal{D}^\prime_\Gamma$ is given by the continuous seminorms $\Vert .\Vert_{N,V,\chi}$,
\begin{equation}
\Vert t\Vert_{N,V,\chi}=\sup_{\xi\in V} (1+\Vert\xi\Vert)^N \vert\widehat{t\chi}(\xi)\vert, 
\end{equation}
where $\supp(\chi)\times V\cap \Gamma=\emptyset$, plus the weak or strong topology of distributions. Note that throughout the paper, we typically control the size of distributions in a stronger topology than the weak or strong topology on $\pazocal{D}^\prime$ because we operate with  H\"older norms.
 
We need to estimate the seminorms $\Vert .\Vert_{N,V,\chi}$ for $\Fse{z}$ 
uniformly in $z\in \{\Im z>0\}$ in the upper half-plane. We will also need to control these seminorms down to
$\Im z\rightarrow 0^+$ with $\Re z\neq 0$.
Now we can bound the wavefront set of $\Fse{z}$ using the oscillatory integral representation of Lemma~\ref{l:oscilldec} which involves oscillatory integrals with complex phase. For $\Im z>0$, they have exponential decay and the oscillatory integrals are well-defined for all $\cv\in \mathbb{C}$. 
But when $\Im z\rightarrow 0^+$, 
they converge to some oscillatory integrals with real phase so we can control the integration 
in $u$ for $\Im z\rightarrow 0^+$ only when $\Re\cv<\frac{n}{2}$. 

\step{1 (ultraviolet part)} We first deal with the {\rm UV} part
$\int_0^\infty  e^{u\frac{Q(x)}{4i}} e^{i\frac{z}{u}} (1-\psi)(u)  u^{\frac{n}{2}-\cv-2}du$.
Assume that $\supp \chi\times V$ does not meet 
$\{(x;\xi) \,|\, Q(x)=0, \, \xi=\tau dQ, \, \tau<0 \}\cup T^*_{0}\mathbb{R}^n$ which in particular 
implies that $0\notin \supp \chi$. 
Our proof is inspired by~\cite[Thm.~0.5.1 p.~38, Thm.~0.4.6 p.~34]{sogge2017fourier}.
We choose some smooth $\psi$ and a smooth bump function $\beta$ supported in $[\frac{1}{2},2]$
s.t. $\psi{(u)} +\sum_{j=0}^\infty \beta(2^{-j}u)=1$; this is a dyadic partition of unity. Set $\beta_j(.)=\beta(2^{-j}.)$. 
Then for $\xi\in V$, we need to consider the series:
$$ \bea 
\sum_{j=0}^\infty \int_0^\infty \left(\int_{\mathbb{R}^n}\chi(x)  e^{-i\left(\langle\xi,x\rangle +\frac{uQ(x)}{4}\right)} d^nx\right) e^{iu^{-1}z}  \beta_j(u)u^{\frac{n}{2}-\cv-2}du.
\eea $$
We fix $j$ and rewrite one term of the series after change of variables:
$$ \bea 
&\int_0^\infty \left(\int_{\mathbb{R}^n}\chi(x)  e^{-i\left(\langle\xi,x\rangle +\frac{uQ(x)}{4}\right)} d^nx\right) e^{iu^{-1}z}  \beta_j(u)u^{\frac{n}{2}-\cv-2}du\\
&= 2^{j(\frac{n}{2}-\cv-1)} \int_0^\infty {\left(\int_{\mathbb{R}^n}\chi(x)  e^{-i\left(\langle\xi,x\rangle +\frac{2^j uQ(x)}{4}\right)} d^nx\right)} e^{i2^{-j}u^{-1}z}  \beta_1(u)u^{\frac{n}{2}-\cv-2}du.
\eea $$
The phase function 
$
\phi(x,\xi,j,u)=\langle\xi,x\rangle +\frac{2^juQ(x)}{4} 
$
is non-degenerate since
$ d_x\phi= \xi+2^ju\frac{dQ(x)}{4} \in C^\infty(\mathbb{R}^n\times V\times \open{0,+\infty};\mathbb{R}^n) $
never vanishes because $\xi\in V$ does not meet $\mathbb{R}_{<0}dQ(x)$ for all $x$ in the support 
of the test function $\chi$.
Define the differential operator $\pazocal{L}=\frac{1+ \langle\nabla_x\phi,\nabla_x \rangle }{1+\langle\nabla_x\phi,\nabla_x\phi \rangle}$.
Observe that the term
$ \langle\nabla_x\phi,\nabla_x\phi \rangle $ in the denominator is bounded from below by
$$\bigg\|\xi+2^ju\frac{dQ(x)}{4}\bigg\|^2\geqslant C(\Vert\xi\Vert+2^j)^2$$
for some $C>0$ uniformly in $\xi\in V$ and $u\in [1,4]$ since $u$ lives in the support of $\beta_1$ and $dQ(x)\neq 0$ because $0\notin \supp \chi$.
By $3N$ integration by parts w.r.t.~$\pazocal{L}$ as in~\cite[Lem.~0.4.7, p.~35]{sogge2017fourier}, since $j\geqslant 1$, we get the bound: 
$$ \sup_{u\in \open{0,+\infty}}\bigg|\int_{\mathbb{R}^n}\chi(x)  e^{-i\left(\langle\xi,x\rangle +\frac{2^j uQ(x)}{4}\right)}  \beta_1(u)  d^n x \bigg|\leqslant C\left(\Vert\xi\Vert+2^j \right)^{-3N}\leqslant C(1+\Vert \xi\Vert)^{-N}2^{-j2N} . $$
Therefore for $j\geqslant 1$,
$$ \bea 
\bigg| 2^{j(\frac{n}{2}-\cv-1)}  \int_0^\infty \left(\int_{\mathbb{R}^n}\chi(x)  e^{-i\left(\langle\xi,x\rangle +\frac{2^j uQ(x)}{4}\right)} d^nx\right) e^{-2^{-j}u^{-1}z}  \beta_1(u)u^{\frac{n}{2}-\cv-2}du\bigg| & \\
\leqslant  C(1+\Vert \xi\Vert)^{-N}\left(\int_1^4 e^{-2^{-j}u^{-1}\Im z}u^{\frac{n}{2}-\Re\cv-2}du \right)  2^{j(\frac{n}{2}-\Re\cv-1-2N)}.  &
\eea $$
Using the elementary estimates:
$$ \bea 
\int_1^4 e^{-2^{-j}u^{-1}\Im z}u^{\frac{n}{2}-\Re\cv-2}du &= \int_{\frac{1}{4}}^1 e^{-2^{-j}u\Im z}u^{\Re\cv-\frac{n}{2}}du
\leqslant C_1  \int_{\frac{1}{4}}^1 e^{-2^{-j}u\Im z} du \\
&\leqslant  \frac{3}{4}C_1 e^{-\frac{\Im z}{42^j}}\leqslant C_{2,N} 
\left(1+\frac{\module{\Im z}}{2^j}\right)^{-N}\\ &\leqslant C_{2,N} 
2^{jN}\left(1+\module{\Im z}\right)^{-N} 
\eea $$
and combining with the above stationary phase estimate, we deduce that
$$ \bea 
\bigg| \int_0^\infty \left(\int_{\mathbb{R}^n}\chi(x)  e^{-i\left(\langle\xi,x\rangle +\frac{uQ(x)}{4}\right)} d^nx\right) e^{iu^{-1}z}  \beta_j(u)u^{\frac{n}{2}-\cv-2}du\bigg| & \\[1mm]
\leqslant C_{3,N}(1+\Vert \xi\Vert)^{-N}(1+\module{\Im z})^{-N}2^{j(\frac{n}{2}-\Re\cv-1-N)}. & 
\eea $$

Now, for all $N>\frac{n}{2}-\Re\cv-1$, the series in $j$ converges absolutely and yields an estimate of the form
$$ \bea 
\bigg| \sum_{j=1}^\infty \int_0^\infty \left(\int_{\mathbb{R}^n}\chi(x)  e^{-i\left(\langle\xi,x\rangle +\frac{uQ(x)}{4}\right)} d^nx\right) e^{iu^{-1}z}  \beta_j(u)u^{\frac{n}{2}-\cv-2}du \bigg|
 \leqslant C_{3,N}(1+\Vert \xi\Vert)^{-N}(1+\module{\Im z})^{-N}.
\eea $$

\step{2 (infrared part)} To conclude the estimate, we still need to deal 
with the infrared part $$ \left\langle I_{{\rm IR}},\varphi e^{i\left\langle\vareta,.\right\rangle}\right\rangle=\frac{ e^{-i(\cv+1)\frac{\pi}{2}}}{(4\pi i)^{\frac{n}{2}}(-1)^{\frac{n-1}{2}}}\int_0^\infty   \left\langle e^{\frac{Q(.)}{4iu}},\varphi e^{i\left\langle\vareta,.\right\rangle}\right \rangle e^{izu} \psi(u^{-1})  u^{\cv-\frac{n}{2}}du $$ for $\vareta\in V$ and $\varphi\in C^\infty_\c(\mathbb{R}^n)$, where we did a variable change $u\mapsto u^{-1}$. 
We first assume that $\Im z\geqslant \varepsilon >0$.
The function $e^{\frac{Q(x)}{4iu}}\varphi(x)$ is smooth in $x$ uniformly in $u\in \supp \psi(u^{-1})$.
Therefore $\vert\left\langle e^{\frac{Q(.)}{4iu}},\varphi e^{i\left\langle\vareta,.\right\rangle}\right \rangle\vert\leqslant C_N (1+\Vert\vareta\Vert)^{-N}$ for all $N\in \mathbb{N}$.
If $\Im z>0$, then $u\in \mathbb{R}\mapsto e^{izu} \psi(u^{-1})  u^{\cv-\frac{n}{2}} $ is Riemann integrable on $\mathbb{R}$ hence we immediately find that for all $N$:
$$ \bea 
\module{\left\langle I_{{\rm IR}},\varphi e^{i\left\langle\vareta,.\right\rangle}\right\rangle}&\leqslant\vert C(\cv)\vert\int_0^\infty   \module{\left\langle e^{\frac{Q(.)}{4iu}},\varphi e^{i\left\langle\vareta,.\right\rangle}\right \rangle } e^{-\Im zu} \psi(u^{-1})  u^{\Re\cv-\frac{n}{2}}du\\
&=\pazocal{O}\left( e^{-\Im z\frac{\delta}{2}}\Vert \vareta\Vert^{-N} \right)=\pazocal{O}\left( \module{\Im z}^{-N} \Vert \vareta\Vert^{-N} \right),
\eea $$
where $\delta>0$ is such that $[0,\delta]\cap \supp(\psi(u^{-1}))=\emptyset $. Now when $\Re\cv<\frac{n}{2}-1$, then the above bound holds true uniformly on $\{\Im z\geqslant 0\}$ since $\psi(u^{-1})  u^{\Re\cv-\frac{n}{2}}$ is Riemann integrable
and 
$$ \bea 
\int_0^\infty   \module{\left\langle e^{\frac{Q(.)}{4iu}},\varphi e^{i\left\langle\vareta,.\right\rangle}\right \rangle } e^{-\Im zu} \psi(u^{-1})  u^{\Re\cv-\frac{n}{2}}du&\leqslant  C_N(1+\Vert\vareta\Vert)^{-N}e^{-\frac{\delta}{2}\Im z}\underset{<+\infty}{\underbrace{\int_0^\infty\psi(u^{-1})  u^{\Re\cv-\frac{n}{2}}du}}\\
&\leqslant  C_{2,N}(1+\Vert\vareta\Vert)^{-N}(1+\module{\Im z})^{-N} .
\eea $$

\step{3 (conclusion)}
Let 
\beq\label{eq:deflbo}
\varlambda_0=\{ (x;\xi)\,|\,\xi=\tau dQ(x), \, Q(x)=0, \, \tau<0 \}\cup (\coto[0]{\mathbb{R}^n})\subset T^*\mathbb{R}^n.
\eeq
For all $\chi\in C^\infty_\c(\mathbb{R}^n\setminus \{0\})$ and all cones $V$ s.t.~$\supp \chi \times V$ does not meet 
$\{(x;\xi) \,|\, Q(x)=0, \ \xi=\tau dQ, \ \tau<0 \}\cup (\coto[0]{\mathbb{R}^n})$, for $\Im z\geqslant \varepsilon>0$,
we deduce an estimate of the form
$$\vert\pazocal{F}\left( \Fse{z}\chi \right)(\xi)\vert \leqslant C(1+\Vert\xi\Vert)^{-N}\module{\Im z}^{-N} $$
uniformly in $\xi\in V$, where $\pazocal{F}$ denotes the Fourier transform.
In other words, in terms of the continuous seminorms $\Vert.\Vert_{N,V,\chi}$ of the $\pazocal{D}^\prime_{\varlambda_0}$ topology, the above estimate reads
$\Vert  \Fse{z}\Vert_{N,V,\chi}\leqslant C \module{\Im z}^{-N} $ for   $\Im z\geqslant \varepsilon>0$.
When $\Re\cv< \frac{n}{2}-1 $, we have a stronger estimate which holds true on $\Im z> 0$:
\begin{eqnarray}\label{ineq:fsmuloc1}
\Vert \Fse{z} \Vert_{N,V,\chi}\leqslant C\left(1+\module{\Im z}\right)^{-N}.
\end{eqnarray}

Combining this with the H\"older estimates of \sec{ss:holderestimates} and Lemma~\ref{l:Fsoutsidediag}, we get the following result.

\begin{prop}\label{p:fsholder}
Let $\varlambda_0=\{ (x;\xi) \,|\, \xi=\tau dQ(x), \, Q(x)=0, \, \tau< 0 \}\cup (\coto[0]{\mathbb{R}^n})$.
Then:
\ben
\item\label{sholder1} The family 
$\left(1+\module{\Im z}\right)^{\Re\cv+1}  \Fse{z}$, $\Im z\geqslant 0$, $\vert z\vert\geqslant \varepsilon>0$, is bounded in $\pazocal{D}^\prime_{\varlambda_0}(\mathbb{R}^n)$.
\item\label{sholder2}  For all $\varepsilon>0$ the family 
$\module{\Im z}^{N}  \Fse{z}$,  ${\Im z\geqslant \varepsilon}$, is bounded in $\pazocal{D}^\prime_{\varlambda_0}(\mathbb{R}^n\setminus \{0\})$ for all $N\in \mathbb{N}$.
\item\label{sholder3} If $\Re\cv<\frac{n}{2}-1$ then the family $\left(1+\module{\Im z}\right)^{\Re\cv+1}  \Fse{z}$, $\Im z\geqslant 0$, is bounded in $\pazocal{D}^\prime_{\varlambda_0}(\mathbb{R}^n)$.
\item\label{sholder4} If $\Re\cv<\frac{n}{2}-1$ then the family 
$\left(1+\module{\Im z}\right)^{N} \Fse{z}$, $\Im z\geqslant 0$,  is bounded in $\pazocal{D}^\prime_{\varlambda_0}(\mathbb{R}^n\setminus \{0\})$ for all $N\in \mathbb{N}$. 
\een
\end{prop}

Using the notation introduced in Definition \ref{def:o},  the statement \eqref{sholder1} is equivalent to
$\Fse{z}=\pazocal{O}_{\pazocal{D}^\prime_{\varlambda_0}}(\module{\Im z}^{-\Re\cv-1})$ in $\Im z>0$, and we can rephrase \eqref{sholder2}--\eqref{sholder4} similarly.

\subsection{The holomorphic family of distributions \texorpdfstring{$\Fse{z}$}{F(z)} for \texorpdfstring{$\Re\cv\geqslant 0$}{Re alpha>0}}

We need to verify algebraic relations satisfied by the 
holomorphic family of Lorentz invariant distributions $\Fse{z}\in \pazocal{D}^\prime(\mathbb{R}^n)$ 
which will appear in the asymptotic expansion of Feynman powers.   

Let $\gamma_\varepsilon$ be the contour in the upper half-plane introduced 
in \sec{ss:restocp}.  We will need the following lemma when inserting the parametrix for $(P-z)^{-1}$ in contour integrals along $\gamma_\varepsilon$.  It is precisely the $\Fse{z}$ family of distributions that will contribute to the singularities near the diagonal of the  Schwartz kernel of the complex powers.

\begin{lemm}\label{l:holofamilyFm}
Let $\Fse{z}\in \pazocal{D}^\prime(\mathbb{R}^n)$, $\cv\in \mathbb{C}$, $\Im z>0$ be the family of distributions defined by \eqref{e:defF}. 
For all $\varm\in\nn$, {$m\in \mathbb{R}$, $\varepsilon>0$}, they satisfy the contour integral identity:
\begin{equation}
\boxed{\frac{1}{2\pi i}\int_{\gamma_\varepsilon} (z{\pm i\varepsilon})^{-\cv} 
\Fse[\varm]{ { z-m^2 }}dz=\frac{(-1)^\varm\Gamma(-\cv+1)}{\Gamma(-\cv-\varm+1)\Gamma(\cv+\varm)}  \Fse[\cv+\varm-1]{-m^2\mp i\varepsilon}.}
\end{equation} 
where both sides converge in $\pazocal{D}^\prime(\mathbb{R}^{n})$ for $\Re\cv>0$.
For $\square_\eta=\eta^{ij}\partial_{x^i}\partial_{x^j}$, we also have the relation
\begin{equation}
\boxed{(\square_\eta-z) \Fse{z}=\cv \Fse[\cv-1]{z}.}
\end{equation}
\end{lemm}
\begin{proof}
We claim that by density of compactly supported functions in $L^2(\mathbb{R}^{n})$ and  Cauchy residue formula,\footnote{Beware that our contour $\gamma_\varepsilon$ is oriented counterclockwise but we integrate against $(Q(\xi)-z)^{-1}$ instead of $(z-Q(\xi))^{-1}$.}
$$ \frac{1}{2\pi i}\int_{\gamma_\varepsilon} (z+i\varepsilon)^{-\cv-1}(Q(\xi)-z)^{-1}dz=(Q(\xi)+i\varepsilon)^{-\cv-1}:L^2(\mathbb{R}^{n})\to L^2(\mathbb{R}^{n}), $$ 
{where the l.h.s.~is norm convergent in  $\pazocal{B}(L^2(\mathbb{R}^{n}))$ when $\Re\cv>0$.}

Indeed, when we multiply $\frac{1}{2\pi i}\int_{\gamma_\varepsilon} (z+i\varepsilon)^{-\cv-1}(Q(\xi)-z)^{-1}dz$ by some compactly supported $\varphi(\xi)$, the values of $Q(\xi)$ when multiplied in $(Q(\xi)-z)^{-1}\varphi(\xi)$ lie in a compact $K\subset \mathbb{C}$. We can enclose $K$ with a large piece of $\gamma_\varepsilon\cap B(0,R)$, which we close up with an arc circle of the form $\{ Re^{i\theta} \,|\, \theta\in [-\omega,\omega] \}$. This arc has size $\sim R $ but the integrand on it decays like $R^{-\Re\cv-1}$ so this large portion tends to $0$ as $R\rightarrow +\infty$. The part $\gamma_\varepsilon\cap B(0,R)^\c$ on the complement of the ball of radius $R$ also decays when $R\rightarrow +\infty$. Therefore, Cauchy's formula tells us that
the identity 
$$ \frac{1}{2\pi i}\int_{\gamma_\varepsilon} (z+i\varepsilon)^{-\cv-1}(Q(\xi)-z)^{-1}dz=(Q(\xi)+i\varepsilon)^{-\cv-1}$$ 
holds true as operators acting on compactly supported smooth functions of $\xi$. But since these are dense in $L^2$, this identity extends in the sense of operators in
$\pazocal{B}(L^2(\mathbb{R}^{n}))$.
By inverse Fourier transform, this yields
$$
 \frac{1}{2\pi i}\int_{\gamma_\varepsilon} (z+i\varepsilon)^{-\cv-1}\Fse[0]{z}dz=\Gamma(\cv+1)^{-1}
\Fse{-i\varepsilon}.
$$

We have to extend the above discussion to the case 
$$ \frac{1}{2\pi i}\int_{\gamma_\varepsilon} (z+i\varepsilon)^{-\cv} \Fse[\varm]{z}dz$$
still in the region $\Re\cv>1$. 
Recall that in Fourier space: 
$${\cF{\Fse[\varm]{z}}}(\xi)=\varm!(Q(\xi)-z)^{-\varm-1}.$$
For every holomorphic $f$, Cauchy's formula says that
$\frac{i}{2\pi}\int_\gamma f(z)(z-z_0)^{-\varm-1}dz=\frac{f^{(\varm)}(z_0)}{\varm!}$ if $\gamma$ is a clockwise contour 
around $z_0$.
Therefore arguing as above yields:
$$ \bea  
 &\frac{1}{2\pi i}\int_{\gamma_\varepsilon} (z+i\varepsilon)^{-\cv} \varm!(Q(\xi)-z)^{-\varm-1}dz\\
&=(-1)^\varm \frac{i}{2\pi}\int_{\gamma_\varepsilon} (z+i\varepsilon)^{-\cv} \varm!(z-Q(\xi))^{-\varm-1}dz\\ 
&=(-1)^\varm(-\cv)\dots(-\cv-\varm+1)(Q(\xi)+i\varepsilon)^{-\cv-\varm}
\eea $$ 
where both sides converge when $\Re\cv>0$ as multiplication operators 
in $\pazocal{B}(L^2(\mathbb{R}^{n}))$.
By inverse Fourier transform and using the definition of $\Fse[\cv+\varm-1]{-i\varepsilon}$
yields
%$$ \bea  
% &\frac{1}{2\pi i}\int_{\gamma_\varepsilon} (z+i\varepsilon)^{-\cv} \Fse[\varm]{z} dz\\
%&=(-1)^\varm \frac{i}{2\pi}\int_{\gamma_\varepsilon} (z+i\varepsilon)^{-\cv} \frac{1}{\varm!}(z-Q(\xi))^{-\varm-1}dz\\ 
%&=\frac{(-1)^m(-\cv)\dots(-\cv-\varm+1)\Gamma(\cv+\varm)}{(\varm!)^2}\Fse[\cv+\varm-1]{-i\varepsilon}
%\eea $$ 
%hence:
$$
\frac{1}{2\pi i}\int_{\gamma_\varepsilon} (z+i\varepsilon)^{-\cv} \Fse[\varm]{z}=\frac{(-1)^\varm\Gamma(-\cv+1)}{\Gamma(-\cv-\varm+1)\Gamma(\cv+\varm)}  \Fse[\cv+\varm-1]{-i\varepsilon},
$$
where the integral makes sense as a bounded operator on acting on $L^2(\mathbb{R}^{n})$.
\end{proof}

\subsubsection{Analytic continuation of the microlocal estimates and Bernstein--Sato polynomial.}
Our next goal is to prove an analytic continuation of the microlocal estimates on $\Fse{z}$ for all $\cv\in \mathbb{C}$ and $z\in \{\Im z\geqslant 0,\, \vert z\vert\geqslant {\varepsilon} > 0\}$. The idea is to prove the existence of a functional equation satisfied by $\Fse{z}$ involving  Bernstein--Sato polynomials~\cite{bernstein1971, bernstein1972, sabbah}.

If $z\neq 0$, observe that one has
$$ \bea 
\partial_{\xi_i}^2(Q(\xi)-z)^{-\cv}=-2\cv\,\eta_{ii} (Q(\xi)-z)^{-\cv-1}+4\cv(\cv+1)\eta_{ii}^2\,\xi_i^2(Q(\xi)-z)^{-\cv-2}, 
\eea $$
which implies after division by $\eta_{ii}$ and summation over $i$:
$$ 
\sum_{i=1}^n \frac{ \partial_{\xi_i}^2}{\eta_{ii}}(Q(\xi)-z)^{-\cv}=-2\cv (Q(\xi)-z)^{-\cv-1}+4\cv(\cv+1)Q(Q(\xi)-z)^{-\cv-2},
$$
hence
$$
\sum_{i=1}^n \frac{ \partial_{\xi_i}^2}{\eta_{ii}}(Q(\xi)-z)^{-\cv}=-2\cv (Q(\xi)-z)^{-\cv-1}+4\cv(\cv+1)(Q(\xi)-z)^{-\cv-1}+4\cv(\cv+1)z(Q(\xi)-z)^{-\cv-2},
$$
and consequently
$$
\left(Q(\xi)Q(\partial_\xi)+2\cv-4\cv(\cv+1)\right)(Q(\xi)-z)^{-\cv}=4\cv(\cv+1)z(Q(\xi)-z)^{-\cv-1}.
 $$
By inverse Fourier transform this yields the functional equation satisfied by the family
$\Fse{z}$
\begin{eqnarray}\label{e:bernsteinsatoFs}
\boxed{A(\cv,x,D_x)\Fse{z}=F_{\cv+1}(z,.),}
\end{eqnarray}
where $A$ is the differential operator with polynomial coefficients 
\begin{equation}\label{e:Pbernstein}
A(\cv,x,D_x)=\frac{\cv\left(Q(\partial_x^2)Q(x)+2(\cv+1)-4(\cv+1)(\cv+2)\right)}{4(\cv+1)(\cv+2)z} 
\end{equation}
which is holomorphic on the half--plane $\Re\cv>-1$.
Using the functional equation~(\ref{e:bernsteinsatoFs}), denoting by $\pazocal{F}$ the Fourier transform, we deduce the identity:
$$ \bea 
&\pazocal{F}\left(F_{\cv+1}(z,.)\chi \right)(\xi)= 
\int_{\mathbb{R}^n} \chi(x)e^{i\left\langle\xi,x\right\rangle} A(\cv,x,D_x)\Fs{z} d^nx\\
&=\int_{\mathbb{R}^n}  \Fs{z}  \, \t A(\cv,x,D_x)\left(\chi(x)e^{i\left\langle\xi,x\right\rangle}\right) d^nx\\
&=\sum \int_{\mathbb{R}^n}  \Fs{z}   \left(A_{(1)}(s,x,D_x)\chi(x)\right) \left(A_{(2)}(s,x,D_x)e^{i\left\langle\xi,x\right\rangle}\right) d^nx \\
&=\sum \int_{\mathbb{R}^n}  \Fs{z}   \left(A_{(1)}(s,x,D_x)\chi(x)\right) \left(A_{(2)}(s,x,\xi) e^{i\left\langle\xi,x\right\rangle}\right) d^nx
\eea $$
where we split the differential operator $\t A$ in two pieces in all possible ways using the Leibniz rule~\footnote{This can also be stated in terms of  
the coproduct $\Delta \,\t A=\sum A_{(1)}\otimes A_{(2)} $ on the coalgebra of differential operators with polynomial coefficients,  following Sweedler's notation.}. The above sum is finite and the degree of $A_{(2)}$ in both $x$ and $\xi$ is always less than $2$.  
 Therefore for all $z\in \{\Im z\geqslant 0, \,\vert z\vert\geqslant \varepsilon> 0\}$, for $\Re\cv<\frac{n}{2}-1$ and all $N_2\in \mathbb{N}$, we can bound the seminorms
of $\Fse[\cv+1]{z}$ in terms of those of $\Fse{z}$: 
$$ \bea 
&\Vert \Fse[\cv+1]{z} \Vert_{N,V,\chi}=\sup_{\xi\in V} \left(1+\Vert\xi\Vert\right)^N \vert \pazocal{F}\left(F_{\cv+1}(z,.)\chi \right)(\xi) \vert\\
&=\sup_{\xi\in V} \left(1+\Vert\xi\Vert\right)^N \bigg|\sum \int_{\mathbb{R}^n}  \Fs{z}   \left(A_{(1)}(s,x,D_x)\chi(x)\right) \left(A_{(2)}(s,x,\xi) e^{i\left\langle\xi,x\right\rangle}\right) d^nx \bigg|\\
&\leqslant  C\sum \Vert \Fse{z}  \Vert_{N+2,V,\chi_{1,2}}\leqslant C(1+\module{\Im z})^{-N_2} 
\eea $$
where all above sums are \emphasize{finite} and the smooth test functions $\chi_{1,2}$ depend on the operators
$A_{(1)},A_{(2)}$. 
Integrating these bounds, we propagate the microlocal estimates from the half-plane $\Re\cv<\frac{n}{2}-1$ to $\cv\in \mathbb{C}$. We deduce that for any continuous seminorm $\Vert.\Vert_{N,V,\chi}$ of $\pazocal{D}^\prime_{\varlambda_0}
$, for all $z\in \{\Im z\geqslant 0,\, \vert z\vert\geqslant {\varepsilon}> 0\}$ and all $\cv\in \mathbb{C}$,   $N_2\in \mathbb{N}$,
$$
 \Vert \Fse[\cv+1]{z} \Vert_{N,V,\chi}\leqslant C(1+\module{\Im z})^{-N_2}.
$$

We  summarize this in the next theorem, together with other results from this section.

\begin{thm}[Analytic properties of the family $F_\cv$]
\label{t:fsanalyticflat}
Let $k=\plancher{\Re\cv}+1$ and $\Fse{z}\in \pazocal{D}^\prime(\mathbb{R}^n)$ as defined in \eqref{e:defF}.
Then for $a\in [0,1]$ {and $\varepsilon>0$},
$\Fse{z}\in \pazocal{C}_{\loc}^{(2-2a)(\Re\cv+1)-k-n}(\rr^n)$ with  decay in {$z\in \{\Im z \geqslant 0, \,\vert z\vert\geqslant \varepsilon\}$} of order $\pazocal{O}({\left(1+ \module{\Im z}\right)^{-a(\Re\cv+1)}})$. 

Let $\varlambda_0=\{ (x;\xi)\,| \, \xi=\tau dQ(x), \, Q(x)=0, \, \tau<0 \}\cup (\coto[0]{\mathbb{R}^n})\subset T^*\mathbb{R}^n$.
For all $z\in \{\Im z\geqslant 0,\, \vert z\vert\geqslant {\varepsilon}>0  \}$ and all $\cv\in \mathbb{C}$, the family $(1+\module{\Im z})^{1+\Re\cv}\Fse{z}$
is bounded in $\pazocal{D}^\prime_{\varlambda_0}(\mathbb{R}^n)$.
Moreover, for every $N\in \mathbb{N}$ and $\varepsilon>0$,
$\big\{\module{\Im z}^N\Fse{z}\big\}_{z\in \{\Im z\geqslant \varepsilon\}}$ is bounded in $\pazocal{D}^\prime_{\varlambda_0}(\mathbb{R}^n\setminus \{0\})$.
\end{thm}

\section{Formal Hadamard parametrix for the resolvent}
\label{s:hadamardformal}

\subsection{Pull-back by exponential maps}

Next, we introduce the main ingredient in the construction of the formal Hadamard parametrix, namely, the pull-back of the distributions $\Fse{z}$ near the diagonal $\diag\subset M\times M$  using the exponential map.

\subsubsection{Moving frame}\label{movingframe} We use the notation $(x;v)$ for elements of $TM$, where $x\in M$ and $v\in T_xM$.
Let $\pazocal{N}$ be 
a neighborhood
of the zero section
$\zero $ in $TM$
for which 
the 
map
$(x;v)\in\pazocal{N}\mapsto ({x},\exp_x(v))\in M^2 $
is a local
diffeomorphism 
onto its image
($\exp_x:T_xM\to M$ is the exponential geodesic map).

The construction of the exponential
in the pseudo-Riemannian
setting is explained in detail in~\cite[App.~A]{Hormander-97}.
The subset 
$\pazocal{U}=\exp\pazocal{N}\subset M^2$
is a neighborhood
of $\diag$ and
the inverse map
$\pazocal{U}\ni (x_1,x_2) \mapsto (x_1;\exp_{x_1}^{-1}(x_2))\in\pazocal{N}$
is a well-defined diffeomorphism.
Let
$\Omega$ be an open subset of $M$
and let 
$(e_0,\dots,e_n)$ be \emphasize{a {time oriented} orthonormal moving frame}
on $\Omega$ 
(i.e.~$\forall x\in \Omega$, $g_x(e_\mu(x),e_\nu(x))=\eta_{\mu\nu}${, and $e_0$ is future directed}), 
and $(\varalpha^\mu)_\mu$ 
the corresponding 
orthonormal moving coframe.

\subsubsection{Pull-back}\label{sss:pullback} We denote by
$\epsilon_\mu$ the
canonical basis of
$\mathbb{R}^{n}$. The data of the orthonormal moving coframe $(\varalpha^\mu)_\mu$
allows us to define for $(x_1,x_2)\in\pazocal{U}$ the submersion
\begin{equation}\label{applipullback}
 G: (x_1,x_2)\mapsto 
G^\mu(x_1,x_2)\epsilon_\mu=
\underset{\in T_{x_1}^* M}{\underbrace{\varalpha^\mu_{x_1}}}\underset{\in T_{x_1}M}{\underbrace{(\exp_{x_1}^{-1}(x_2))}}\epsilon_\mu\in\mathbb{R}^{n}. 
\end{equation}
For any distribution $f$ in
$\pazocal{D}^\prime(\mathbb{R}^{n})$, 
the composition 
$\pazocal{U}\ni(x_1,x_2)\mapsto G^*f(x_1,x_2)$
defines the pull-back of $f$ on $\pazocal{U}\subset M^2$.
If $f$ is $O(1,n-1)_+^\uparrow$-invariant, 
then
the pull-back defined above
\emphasize{does not depend on the choice
of orthonormal moving frame} $(e_\mu)_\mu$ 
and is thus \emphasize{intrinsic}
(since all orthonormal moving frames 
are related by gauge transformations 
in $C^\infty(M;O(1,n-1)_+^\uparrow)$).

This 
allows us to {canonically}
pull-back
$O(1,n-1)_+^\uparrow$-invariant 
distributions to distributions
defined on a neighborhood $\pazocal{U}$
of $\diag$.

\begin{defi}\label{d:fsneardiag}
We apply this construction to the family 
$\Fse{z}\in \pazocal{D}^\prime\left(\mathbb{R}^{n}\right)$ 
constructed in Proposition \ref{prop:flatfeynm},
and we obtain 
the distribution
$\Fe{z}= G^*\Fse{z}\in \pazocal{D}^\prime(\pazocal{U})$.
\end{defi}

\begin{lemm}\label{Wavefrontpullback}
Let $(M,g)$ be a globally hyperbolic Lorentzian manifold, $\pazocal{U}$  the neighborhood of the diagonal $\diag\subset M\times M$
defined in \sec{movingframe}, $G:\pazocal{U}\to \mathbb{R}^n$  the map defined in  \eqref{applipullback} and let $\Fse{z}\in \pazocal{D}^\prime(\mathbb{R}^n)$ be the family of distributions defined in \eqref{e:defF}. Then 
the wavefront set 
of the distribution
$\Fe{z}=G^* \Fse{z}$ 
is contained in the \emph{Feynman wavefront}\footnote{Note that in the literature on Quantum Field Theory on curved spacetime, the opposite convention is often used for Feynman propagators and for the Feynman wavefront. } $\varlambda\subset (T^*M \setminus \zero)\times (T^*M \setminus \zero)$, defined by 
\beq\label{eq:wavefrontfeynman}
\varlambda'= \{  (q_1,q_2) \in \Sigma\times \Sigma \, | \, q_1 {\succ} q_2 \}\cup \ediag.
\eeq
using the notation introduced in Definition \ref{anewdef}.
\end{lemm}

The proof of Lemma \ref{Wavefrontpullback} will be given in \sec{ss:pullback} in the appendix.

We conclude this section by our main result on the regularity of the family $\mathbf{F}_\cv(z)=G^*F_\cv(z)\in \pazocal{D}^\prime_\varlambda(\pazocal{U})$
which follows from continuity in the normal topology~\cite[Prop.~5.1 p.~211]{Viet-wf2} of the pull-back $G^*:\pazocal{D}^\prime_{\varlambda_0}(\mathbb{R}^n)\to \pazocal{D}^\prime_\varlambda(\pazocal{U})$ and Proposition \ref{p:fsholder}.

\begin{prop}[Boundedness of family $\mathbf{F}_\cv$]\label{p:boundednessFs}
Let $(M,g)$ be a globally hyperbolic Lorentzian manifold,  $\pazocal{U}$  the neighborhood of the diagonal $\diag\subset M\times M$
defined in \sec{movingframe}, $G:\pazocal{U}\to \mathbb{R}^n$  the map defined by \eqref{applipullback} and let $\Fse{z}\in \pazocal{D}^\prime(\mathbb{R}^n)$ be the family of distributions \eqref{e:defF}.
For every  $\cv\in \mathbb{C}$, the family of distributions $\F{z}=\Fe{z}=G^*\Fse{z}$, 
{$z\in \{\Im z\geqslant 0,\, \vert z\vert\geqslant \varepsilon >0\}$},
has the property that
$ \left(1+\module{\Im z}\right)^{\Re\cv+1}\F{z}$ is bounded in $\pazocal{D}^\prime_\varlambda(\pazocal{U})$.
\end{prop}

\subsubsection{Preliminary identities.}

Recall that our differential operator of interest is of the form
\beq\label{lkjlkjff}
P-z=(\partial_{x^j}g^{jk}\partial_{x^k}+m^2-z)+b^j \partial_{x^j}.
\eeq
In the formal calculus\,\footnote{We remark here that in applications it could be advantageous to make the connection with a systematic calculus tailored to computations near the diagonal, see \cite{Derezinski2020}.} used in the Hadamard parametrix construction,    the  part in parentheses on the r.h.s.~of \eqref{lkjlkjff} has weight $2$, in particular the parameter $z$ is included in the weight $2$ part.  

We first state the key identities  satisfied    by  the family $\Fse{z}$ on $\mathbb{R}^n$.

\begin{lemm}
For all $z\in \{\Im z\geqslant 0, z\neq 0\}$, the family
$\Fse{z}$ of distributions on $\rr^n$ satisfies the identities: 
\begin{eqnarray}\label{identity1}
(\eta^{\mu\nu}\partial_{x^\mu}\partial_{x^\nu}-z)\Fse{z}=\cv\Fse[\cv-1]{z} \text{ if }\cv\neq 0, \ (\eta^{\mu\nu}\partial_{x^\mu}\partial_{x^\nu}-z)F_0(z,\vert.\vert_g)=\delta_0,
\end{eqnarray}
\begin{eqnarray}\label{identity2}
2\partial_{x^\mu}\Fse{z}=\eta_{\mu\nu}x^\nu F_{\cv-1}(z,\vert.\vert_g).
\end{eqnarray}
\end{lemm}
\begin{proof}
The first identity follows from Lemma~\ref{l:holofamilyFm}. 
The second identity follows from the representation formula 
from Lemma~\ref{l:oscilldec}, namely
$$\Fs{z}=
\frac{e^{-i(\cv+1)\frac{\pi}{2}}}{{(4\pi i)^{\frac{n}{2}}} (-1)^{\frac{n-1}{2}} }  \int_0^\infty  e^{{\frac{uQ(x)}{4i}+izu^{-1}}}  u^{\frac{n}{2}-\cv-2}du,$$ 
and we differentiate under the integral and use the chain rule to obtain the desired result.
In more detail, this reads~\footnote{{The minus sign comes from $Q(x)=-\eta_{ij}x^ix^j$.}}
{
$$ \bea 
&\partial_{x^j}\frac{e^{-i(\cv+1)\frac{\pi}{2}}}{ (4\pi i)^{\frac{n}{2}} (-1)^{\frac{n-1}{2}} }  \int_0^\infty  e^{\frac{uQ(x)}{4i}+iu^{-1}z}  u^{\frac{n}{2}-\cv-2}du\\
&=  -2\eta_{ji}x^i    \frac{e^{-i(\cv+1)\frac{\pi}{2}}}{ (4\pi i)^{\frac{n}{2}} (-1)^{\frac{n-1}{2}} }  \int_0^\infty  \frac{u}{4i} e^{\frac{uQ(x)}{4i}+iu^{-1}z}  u^{\frac{n}{2}-\cv-2}du\\
&= \frac{\eta_{ji}x^i}{2}    \frac{e^{i\frac{\pi}{2}}e^{-i(\cv+1)\frac{\pi}{2}}}{ (4\pi i)^{\frac{n}{2}} (-1)^{\frac{n-1}{2}} }  \int_0^\infty  e^{\frac{uQ(x)}{4i}+iu^{-1}z}  u^{\frac{n}{2}-\cv-1}du=\frac{\eta_{ji}x^i}{2} \Fs[\cv-1]{z}.
\eea $$
}
The fact that we can differentiate under the integral is justified 
for $\Im z>0$ and  $\Re\cv>\frac{n}{2}-2$ (this guarantees all integrals converge absolutely and we can differentiate under the integral argument) and the general result follows from analytic continuation
of the identity $ 2\partial_{x^j}\Fse{z}=\eta_{ji} x^i \Fse[\cv-1]{z}$ in $\cv$ and the weak convergence of both sides in the distribution sense when $\Im z\rightarrow 0^+$.
\end{proof}

\subsubsection{Identities in normal coordinates.}

We consider the family of distributions $\Fe{z}\in \pazocal{D}^\prime(\pazocal{U})$
 introduced in Definition~\ref{d:fsneardiag}, which plays the role of the buiding blocks of the parametrix.
 
The parametrix is constructed in normal charts. 
This means that
we fix a point $\varvarm\in M$, then we express the distribution $x\mapsto \mathbf{F}_\cv(z,\varvarm,x)$
in normal coordinates centered at $\varvarm$. The fact that we can freeze $\varvarm$ and view $x\mapsto \mathbf{F}_\cv(z,\varvarm,x)$ as a \emph{distribution of the second variable} $x$ comes from the wavefront set of 
$\mathbf{F}_\cv(z,.,.)\in \pazocal{D}^\prime(\pazocal{U})$, which is contained in $\varlambda\subset T^*\left(M\times M\right)$. Near the element $(\varvarm,\varvarm)$ on the diagonal, the set $\varlambda$ is close to the conormal $N^*\diag$ and therefore $\varlambda$ is locally 
transverse to the conormal $N^*\left( \{\varvarm\}\times M\right)=T_\varvarm^*M\times \zero $ of the submanifold $\{\varvarm\}\times M$ near the diagonal $(\varvarm,\varvarm)$. 
Hence, the pull-back theorem of H\"ormander allows us to restrict the distribution $\mathbf{F}_\cv(z,.,.)$ to 
$\{\varvarm\}\times M$, which means in practice that we freeze the variable $\varvarm$ and consider 
$\mathbf{F}_\cv(z,\varvarm,.)$ as a distribution of the second variable.
In the sequel, we work in normal coordinates centered at $\varvarm$.
\begin{defi}\label{d:Fsnormal}
Instead of using the rather heavy notation $T_{\varvarm}  M \supset U  \ni v \mapsto \mathbf{F}_\cv(z,\varvarm,\exp_\varvarm(v))$, we use the simplified notation
$\Feg{z}\in \pazocal{D}^\prime(U)$ to denote the distribution $T_{\varvarm}M \supset U \ni  v \mapsto \mathbf{F}_\cv(z,\varvarm,\exp_{\varvarm}(v))$, where
the notation $\vert y\vert^2_g$ is the pseudodistance squared of $y$ w.r.t.~$0\in T_{\varvarm}M$ which represents the point $\varvarm$ in the normal chart around $\varvarm$ and $g$ is the metric pulled-back on $T_{\varvarm}M$ by the exponential map. 
\end{defi}
The fundamental equation satisfied by the normal coordinates
reads~\cite[A 2.3 p.~271]{Hormander-97} 
\begin{eqnarray}\label{e:normalcoord}
\boxed{g_{jk}(x)x^k=g_{jk}(0)x^k=x^j}
\end{eqnarray}
and this very general result is valid in pseudo-Riemannian geometry.
This implies that $\vert y\vert^2_g=g_{jk}(0)y^jy^k=\eta_{jk}y^jy^k$.
The second key observation is  the statement of the next lemma. 

\begin{lemm}\label{l:identitiesnormal}
Let $\Feg{z}\in \pazocal{D}^\prime(U)$ be the family of distributions from Definition~\ref{d:Fsnormal}. In the normal coordinate system $(x^j)_{j=0}^n$ on $U$ defined in \eqref{e:normalcoord}, we have the
identities: 
\begin{eqnarray}
2g^{jk}(x)\partial_{x^k}\Fxg{z}=x^j\Fxg[\cv-1]{z},\\
\left(\partial_{x^j}g^{jk}\partial_{x^k}-z\right)\Fxg[0]{z}={ \module{g(x)}^{-\frac{1}{2}}}\delta_0(x), \\
\left(\partial_{x^j}g^{jk}\partial_{x^k}-z\right)\Fxg[\cv]{z}=\cv \Fxg[\cv-1]{z}.
\end{eqnarray}
\end{lemm}
\begin{proof} The proof is completely analogous to the Riemannian case, see \cite[(17.4.2), (17.4.3) p.~31--32]{HormanderIII}.
The important property  used to derive these identities is 
that in the normal coordinate system, for any function $f(\vert.\vert_g^2)$ of the square geodesic length, 
we have
$$g^{jk}(x)\partial_{x^k}f(\vert x\vert_g^2)=g^{jk}(0)\partial_{x^k}f(\vert x\vert_g^2).$$ 
The proof of the first equation follows from the fact that
$$ \bea 
2g^{jk}(x)\partial_{x^k}\Fxg{z}=2g^{jk}(0)\partial_{x^k}\Fxg{z} =g^{jk}(0)\eta_{ki}x^i\mathbf{F}_{\cv-1}(z,\vert x\vert_g)=x^j\mathbf{F}_{\cv-1}(z,\vert x\vert_g)
\eea $$
where we used \eqref{identity2} and $g^{jk}(0)=\eta^{jk}$.
The second equation follows from the first equation, \eqref{identity1} and properties of the normal coordinate chart:
$$ \bea 
\left(\partial_{x^j}g^{jk}\partial_{x^k}-z\right)\mathbf{F}_0(z,\vert x\vert_g)={\module{g(x)}^{-\frac{1}{2}}}\delta_0(x).
\eea $$
\end{proof}

\subsection{Deriving the transport equations.}

Recall that $\square_g$ is the Lorentzian Laplace--Beltrami operator and $P=\square_g+m^2$ is the Klein--Gordon operator. The parametrix construction involves  transport equations because even though the operator $(\partial_{x^j}g^{jk}(x)\partial_{x^k}+m^2-z)$
 has a fundamental solution which is 
$\mathbf{F}_0(-m^2+z,\vert x\vert^2_g)$, 
the operator $P-z$ is  not necessarily of the form
$\left(\partial_{x^j}g^{jk}(x)\partial_{x^k}+m^2-z\right)$ but is rather given by the more general expression
$$ P-z = \partial_{x^j}g^{jk}(x)\partial_{x^k}+b^j(x)\partial_{x^j}+m^2-z $$
with a \emphasize{non-trivial subprincipal part}
$b^j(x)\partial_{x^j}$. This subprincipal part will be responsible for the appearance of the scalar curvature as we will
later see. 

%\begin{rema}\label{rem:vector2}
%We can also work in the bundle case for a smooth vector bundle $E\to M$,
%we treat operators acting on sections $C^\infty(E)$ of $E$ of the form $P=\square_g(x,D_x)\otimes \id_{\End(E_x)}+\text{first order terms}$. 
%It is a classical result in global analysis that for any such operator $P$, we can always find a connection $\nabla^E$ on $E$ s.t. $P$ reads~\cite[Prop 2.5 p.~66]{BGV}
%\begin{eqnarray}
%P=Tr\left(\nabla^{T^*M\otimes E}\nabla^E . \right)+V
%\end{eqnarray}
%where $Tr:C^\infty(T^*M\otimes T^*M\otimes E)\mapsto C^\infty(E)$ is defined by contraction with the metric tensor and $V\in C^\infty(End(E))$. It is also a well--known result that
%when $\nabla^E$ is a Hermitian connection on the Hermitian bundle $E$ then 
%$Tr\left(\nabla^{T^*M\otimes E}\nabla^E . \right)$ is symmetric w.r.t. the scalar product 
%defined on $C^\infty_c(E)$ as:
%\begin{eqnarray}
%\int_M\left\langle \overline{s_1}(x),s_2(x) \right\rangle_{E_x} dv_g
%\end{eqnarray}
%where $\left\langle .,. \right\rangle_{E_x}$ is the fiberwise scalar product on $E$ and $dv_g$ is the volume form on $M$.

%In the applications to the square of the Lorentz Dirac operator $P=-D^2$ acting on spinors, 
%the Lichnerowicz Theorem tells us that
%\begin{eqnarray}
%P=-\nabla^{S,\dagger}\nabla^S+V
%\end{eqnarray}
%where $\nabla^S$ is the Levi--Civita connection acting on spinors, $\nabla^{S,\dagger}$ is the Lorentz adjoint and
%$V=-\frac{R_g}{4} \id_{\End(E_x)}$ for $R_g\in C^\infty(M)$ is the multiplication by the scalar curvature.
%} 
%\end{rema}

For $\Im z>0$, let $f(z,.)$ be the unique Schwartz distribution such that 
$$\Fs[0]{z}=\frac{1}{(2\pi)^{n}} \int e^{i\left\langle x,\xi \right\rangle}\left(\vert\xi\vert_\eta^2-i0-z \right)^{-1}d^{n}\xi=f(z,\vert x\vert_\eta^2).$$
We have the following Fourier integral representation for $f$~\footnote{It is related to Bessel--Macdonald $K$ functions, sometimes called modified Bessel functions of the second kind.}:
\begin{equation}
f(z,q)=\frac{e^{-i\frac{\pi}{2}}}{(4\pi i)^{\frac{n}{2}} (-1)^{\frac{n-1}{2}} }  \int_0^\infty  e^{\frac{q}{4ui}+iuz}  u^{-\frac{n}{2}}du.
\end{equation} 
The existence of $f(z,.)\in \pazocal{D}^\prime(\mathbb{R})$ and of $f(z,\vert .\vert_g^2)\in \pazocal{D}^\prime(U)$ 
for $\Im z>0$ follows from the oscillatory integral proofs of Lemma~\ref{l:oscilldec}. 
We have the following key lemma, which again parallels the Riemannian case \cite[(17.4.5) p.~32]{HormanderIII}.
\begin{lemm}\label{l:pretransport1}
Let $\Feg{z-m^2}\in \pazocal{D}^\prime(U)$ be the family of distributions from Definition~\ref{d:Fsnormal},
in the normal coordinate system $(x^j)_{j=0}^n$ on $U\subset T_{\varvarm} M$ defined in \eqref{e:normalcoord},
and let  $\Re\cv>0$.
For any $u\in C^\infty(U)$, 
\begin{eqnarray}
\boxed{ (P-z)\left(u \mathbf{F}_{\cv}\right)={\cv} u\mathbf{F}_{\cv-1}+(Pu)\mathbf{F}_{\cv}+(hu+ 2 \rho u)\frac{\mathbf{F}_{\cv-1}}{2} .}
\end{eqnarray} 
where
\begin{equation}
h(x)=b^j(x)\eta_{jk}x^k\text{ and } 
\rho=x^k\partial_{x^k}.
\end{equation}
For $\cv=0$ and all $u_0(x)\in C^\infty(U)$,
\begin{eqnarray}\label{eq1:mainpart}
\boxed{ (P-z)u_0\mathbf{F}_0=u_0 { \module{g(x)}^{-\frac{1}{2}}}\delta_0(x)+(Pu_0)\mathbf{F}_0+2hu_0 f^\prime(z,\vert.\vert_g)+4x^j\frac{\partial u_0}{\partial x_j} f^\prime(z,\vert.\vert_g) }
\end{eqnarray} 
where $f^\prime$ is the distributional derivative of $f$.
\end{lemm}
\begin{proof}
By definition and using all identities from Lemma~\ref{l:identitiesnormal}:
$$ \bea 
(P-z)\left(u \mathbf{F}_\cv\right)&=(Pu) \mathbf{F}_\cv+u\left(\left(\partial_{x^k}g^{jk}(x)\partial_{x^k}-z\right) \mathbf{F}_\cv\right)+2(\partial_{x^j}u) g^{jk}(x)\left(\partial_{x^k} \mathbf{F}_\cv\right)+ub^j(x)\left(\partial_{x^j}\mathbf{F}_\cv\right)\\
&=(Pu) \mathbf{F}_\cv+u \cv\mathbf{F}_{\cv-1}+(x^j\partial_{x^j}u)  \mathbf{F}_{\cv-1}+\frac{ub^j(x)\eta_{jk}x^k}{2}  \mathbf{F}_{\cv-1}
\eea $$
since 
$$ 
{2g^{jk}(x)\partial_{x^k}\mathbf{F}_\cv=  x^j\mathbf{F}_{\cv-1},}
$$
which implies
$$
 2\partial_{x^i}\mathbf{F}_\cv=2g_{ik}(x)g^{kj}(x)\partial_{x^j}\mathbf{F}_\cv= g_{ik}(x) x^k\mathbf{F}_{\cv-1}= g_{ik}(0) x^k\mathbf{F}_{\cv-1}=\eta_{ik}x^k\mathbf{F}_{\cv-1}.
$$
The second equation is obtained in the same way.
\end{proof}

%\begin{rema}
%In the bundle case { for $P= Tr\left(\nabla^{T^*M\otimes E}\nabla^E . \right)+C^\infty(End(E)) $ where $\nabla^E$
%is some given connection  on $E$}, we have the exact same equation except that
%$u\in C^\infty(\pazocal{U},E\boxtimes E^\vee)$ and the operator $\rho$ of Lie derivative along $\rho$ should be replaced by 
%${ \nabla^E_\rho}$, where we use the covariant derivative ${ \nabla^E_\rho}$ based on the connection and ${ h \in C^\infty(End(E))}$ is viewed as a diagonal operator. 
%\end{rema}

The existence of $f^\prime(z,\vert .\vert_g)$ follows from the same arguments as in the proof of Lemma
~\ref{l:oscilldec}.

\subsubsection{Parametrix from transport equations.}
In this paragraph we  construct the formal parametrix in the normal coordinate chart $U\subset T_{\varvarm}M$ centered around $\varvarm\in M$.
 We start from equation (\ref{eq1:mainpart}). We need to solve away the term
in front of $f^\prime$ which reads
$ 4x^j\frac{\partial u_0}{\partial x_j}+2hu_0$, so we must look for $u_0\in C^\infty(U)$ solution of the first \emphasize{transport equation}
\begin{equation}\label{eq1:transport}
2\rho u_0+hu_0=0, 
\end{equation}
with initial condition $u_0(0)=1$.
So we see immediately that there is a potential problem since there is still a term
$(Pu_0)\mathbf{F}_0$ which is singular. To kill the singular term $(Pu_0)\mathbf{F}_0$,
we look for 
$u_1\in C^\infty(U)$ satisfying
\begin{equation}
\rho u_1 + u_1 +\frac{h}{2}u_1 =-Pu_0,
\end{equation}
since for such pair of smooth functions $(u_0,u_1)\in C^\infty(U)^2$, we would immediately find that 
$$ \bea 
(P-z)\left(u_0\mathbf{F}_0+ u_1 \mathbf{F}_1\right)&=u_0 { \module{g}^{-\frac{1}{2}}}\delta_0(x)+(Pu_0)\mathbf{F}_0+ u_1 \mathbf{F}_{0}+(Pu_1)\mathbf{F}_1+(hu_1+ 2 \rho u_1)\frac{\mathbf{F}_{0}}{2}\\
&={ \module{g}^{-\frac{1}{2}}}\delta_0(x)+(Pu_1)\mathbf{F}_1.
\eea $$

Applying the above algorithm recursively,  
we see that at order $N$  we have to look for a parametrix $H_N(z)$ of the form
\begin{eqnarray}
\boxed{H_N(z)=\sum_{k=0}^N u_k \Feg[k]{z-m^2} \in \pazocal{D}^\prime(U) }
\end{eqnarray}
where the sequence of functions 
$(u_k)_{k=0}^\infty$ in $C^\infty(U)$ solves the well-known \emphasize{hierarchy of transport equations}
\begin{eqnarray}
\boxed{2k u_k+hu_k+2\rho u_k+2Pu_{k-1}=0}
\end{eqnarray}
where for $k=0$, we choose the convention that $u_{k-1}=0$.
This would kill all terms in front of $f^\prime,\mathbf{F}_0,\dots,\mathbf{F}_{N-1}$
therefore in the normal chart near $x$, $H_N(z,.)$ satisfies the equation:
\begin{equation}
\left(P-z\right)H_N(z,.)={\module{g}^{-\frac{1}{2}}}\delta_0+(Pu_N)\mathbf{F}_N.
\end{equation}
{Note that the solutions $(u_k)_{k=0}^\infty$ of the transport equations do not depend on $z$ or on the mass term $m$. This dependence is absorbed in the distributions $\Feg[k]{z-m^2}$.}

The next lemma is fully analogous to \cite[Lem.~17.4.1 p.~33]{HormanderIII}:
\begin{lemm}
The hierarchy of transport equations always has solutions in $C^\infty(U)$ where $U\subset T_{\varvarm}M $ is any open neighborhood of $0\in T_{\varvarm}M$
such that $\exp_{\varvarm}:U\to M$ is \emphasize{injective}. 
\end{lemm}

The formulation of H\"ormander is practical  for proving estimates and fairly general, but to extract the scalar curvature we will later have to specialize  to the case of the pseudo-Riemannian Laplace--Beltrami operator.

\subsubsection{Going back to a neighborhood of the diagonal $\diag$.}

For the moment we have constructed a parametrix $T_{\varvarm} M \supset U \ni x \mapsto H_N(z,x)$ around some fixed $\varvarm\in M$. Now we need to treat $\varvarm$ as a parameter and prove that
everything depends nicely on $\varvarm\in M$. First, observe that the solutions $(u_k)_k$ of the transport equations
are smooth in $C^\infty(U)$. Recall that $(\varalpha^\mu)_\mu$ is the coframe
from \sec{movingframe} and $\pazocal{U}\subset M\times M$ is a neighborhood of the diagonal $\diag\subset M\times M$.
Therefore,  $\pazocal{U} \ni (x_1,x_2) \mapsto u_k(\varalpha(\exp_{x_1}^{-1}(x_2))) $ is smooth in both arguments by composition and smoothness of the inverse exponential map on $\pazocal{U}$. The distributions $\mathbf{F}_{\cv}(z,.)=G^*\Fse{z}\in \pazocal{D}^\prime(\pazocal{U})$ are also well-defined on the neighborhood 
$\pazocal{U}$ of the diagonal (with wavefront set in $\varlambda$).   
Therefore the parametrix $$H_N(z,x_1,x_2)=\sum_{k=0}^N u_k(\varalpha(\exp_{x_1}^{-1}(x_2)))G^*F_\cv(z-m^2,x_1,x_2)$$ describes in fact an element of $\pazocal{D}^\prime_\varlambda(\pazocal{U})$.
For the sake of brevity, by slight abuse of notation we simply write 
$u_k$  for the solution of the transport equation (the inverse exponential map is dropped), and the parametrix $H_N(z,x)$ depending on the variable $x$ in normal chart around $\varvarm$ or its pull-back $H_N(z,\varalpha(\exp_{x_1}^{-1}(x_2)))$ on $\pazocal{U}$ are both denoted by $H_N$. So from now on, one should always be aware that all objects are defined in terms of the exponential map. With these conventions the \emph{Hadamard parametrix} $H_N(z,.)$ reads
\begin{equation}
\boxed{H_N(z,.)=\sum_{k=0}^N u_k \Feg[k]{z-m^2}\in \pazocal{D}^\prime(\pazocal{U}).}
\end{equation}

{In the sequel, we shall use the notation $\delta_\Delta$ for the distribution defined locally by pull-back as
$\delta_\Delta=G^*\delta_0$ which we can extend globally by partition of unity. By construction, $\delta_\Delta $ is a conormal distribution supported by the diagonal $\Delta\subset M\times M$ and $\delta_\Delta(x,y)\module{g}^{-\frac{1}{2}} $ is the Schwartz kernel of the identity map~\footnote{We need to multiply by $\module{g}^{-\frac{1}{2}}$ in order to take into account integration against the volume element.}.}

%\begin{rema}
%In the bundle case, the distributions $\mathbf{F}_\varm$ are defined identically, the only difference being that the solutions
%$u_k$ of the transport equations { are defined} in terms of the connection { $\nabla^E$} and are therefore elements of
%$C^\infty(\pazocal{U},E\boxtimes E^\vee)$.  
%Consequently, $H_N(z,.)=\sum_{k=0}^N u_k \Feg[k]{z-m^2}\in \pazocal{D}^\prime(\pazocal{U},E\boxtimes E^\vee)$
%where $u_k\in C^\infty(\pazocal{U},E\boxtimes E^\vee)$ and $\Feg[k]{z-m^2}$ are the same scalar distributions.
%The equation satisfied by the parametrix reads
%$\left(P-z\right)H_N(z,.)={ \module{g}^{-\frac{1}{2}}} \one_{\End(E_x)} \delta_{\diag}+(Pu_N)\mathbf{F}_N
%\in \pazocal{D}^\prime(\pazocal{U},E\boxtimes E^\vee) $.
%\end{rema}

\section{The Hadamard parametrix approximates the resolvent}\label{sec:gluing}

\subsection{Summary} From now on, we assume that $(M,g)$ is globally hyperbolic with non-trapping Lorentz scattering metric $g$.

The goal of this section is to prove that the formal parametrix $H_N(z,.)$ constructed in \sec{s:hadamardformal} truly approximates the
resolvent $(P-z)^{-1}$  in the functional space $\pazocal{D}^\prime_\varlambda(\pazocal{U})$ of distributions
defined near the diagonal, whose wavefront set is the Feynman wavefront $\varlambda$.

By Lemma~\ref{l:holderpullback} proved in the appendix
which shows that the H\"older regularity of $\Fse{z}$ is preserved under pull-back by $G$ 
and Theorem~\ref{t:fsanalyticflat} giving the microlocal properties of the family $\Fse{z}$, combined with Proposition~\ref{p:boundednessFs}, we have the following bounds.
\begin{lemm}\label{l:regFs}
Let $\pazocal{U}\subset M\times M$ be the neighborhood of the diagonal $\diag\subset M\times M$ as defined in \sec{movingframe}, and let $a\in [0,1]$. The
 family of distributions $\braket{\Im z}^{\Re\cv+1}\mathbf{F}_{\cv}(z,.)$ is bounded in $\pazocal{D}_\varlambda^\prime\left( \pazocal{U}\right)$ and the family $\braket{\Im z}^{a(\Re\cv+1)}\mathbf{F}_{\cv}(z,.)$ is bounded in
$\pazocal{C}^{\varalpha}_{\loc}\left( \pazocal{U}\right)$ uniformly in {$z\in \{\Im z\geqslant 0, \,\vert z\vert\geqslant \varepsilon>0\}$} for all 
$\varalpha\leqslant (2-2a)(\Re\cv+1)-k-n$ where $k=\plancher{\Re\cv}+1$.

Moreover, for every $N\in \mathbb{N}$ and $\varepsilon>0$,
$(\module{\Im z}^N\mathbf{F}_{\cv}(z,.))_{z\in \Im z>\varepsilon}$ is bounded in $\pazocal{D}_\varlambda^\prime\left( \pazocal{U}\setminus \diag\right)$
uniformly in $\cv\in \mathbb{C}$. If $z\in \{\Im z\geqslant 0, \,\Re z\geqslant m^2>0\}$, then
$(\braket{\Im z}^N\mathbf{F}_{\cv}(z,.))_{z\in \{\Im z\geqslant 0, {\vert z\vert\geqslant \varepsilon>0}\}}$ is bounded in $\pazocal{D}_\varlambda^\prime\left( \pazocal{U}\setminus \diag\right)$
uniformly in $\cv\in \mathbb{C}$.
\end{lemm}

\subsection{Resolvent approximation}

Let $\chi\in C^\infty(M\times M;[0,1])$ 
be  such that $\chi=1 $ near the diagonal $\diag\subset M\times M$ and 
$\chi(x,y)=0$ outside the neighborhood $\pazocal{U}$ defined in \eqref{applipullback}.
Recall that the Feynman wavefront set is defined in \eqref{eq:wavefrontfeynman}. 

We interpret the family
$$ \bea 
H_N(z-m^2,.)\chi=\sum_{k=0}^N u_k\mathbf{F}_k(z-m^2,.)\chi\in \pazocal{D}^\prime_{\varlambda}(\pazocal{U}),
\eea $$
as a family of Schwartz kernels, and we show that  near the diagonal, the corresponding family of operators is a parametrix  which approximates the resolvent $(P-z)^{-1}$.

Let $\tilde{\pazocal{U}}$  be  
a neighborhood of the 
diagonal $\diag$ such that $\chi|_{\tilde{\pazocal{U}}}=1$ and $\tilde{\pazocal{U}}\subset \pazocal{U}$.
Let 
$$
\varlambda_{\chi}=\{(x,y;\xi,\eta) \st (x,y;\xi,\eta)\in \varlambda, \, (x,y)\notin \tilde{\pazocal{U}}  \},
$$
so  $\varlambda_{\chi}=\varlambda\cap T^*\big((M\times M)\setminus\tilde{\pazocal{U}} \big)$ is a 
\emphasize{truncation of the Feynman wavefront set} $\varlambda$ where we removed the 
neighborhood $\tilde{\pazocal{U}}$
near the diagonal. 

\begin{lemm}\label{lemfirg}
Assume  $(M,g)$ is globally hyperbolic and set $P=\square_g+m^2$.
For every $\varalpha\in \mathbb{R}_{\geqslant 0}$, $p\in \mathbb{Z}_{\geqslant 0}$, $m\geqslant 0$ there exists $N$ large enough s.t.
for every $z\in\{\Im z\geqslant 0,\, \vert z\vert\geqslant \varepsilon>0\}$
\begin{eqnarray}
\left(P-z\right)\left(\sum_{k=0}^N u_k \mathbf{F}_k(z-m^2,.)\chi\right)={ \module{g}^{-\frac{1}{2}}}\delta_{\diag}+(Pu_N)\mathbf{F}_N(z-m^2,.)\chi+r_N(z),
\end{eqnarray}
where:
\ben
\item\label{manul1} ${\module{g}^{-\frac{1}{2}}}\delta_{\diag}\in \pazocal{D}^\prime(M\times M)$ is the Schwartz kernel of the identity map,
\item\label{manul2} $\bra z\ket^{p}(Pu_N)\mathbf{F}_N(z-m^2,.)\chi$ is bounded in $\pazocal{C}^\varalpha_{\loc}(\pazocal{U})$, 
\item\label{manul3}
$r_N(z,.)\in \pazocal{D}^\prime(M\times M)$ vanishes on $\tilde{\pazocal{U}}\subset\pazocal{U}$ and outside $\pazocal{U}\subset M\times M$. In particular, $r_N(z,.)$ is the Schwartz kernel of a family of {proper{ly supported } operators}. Furthermore,  $r_N(z,.)$  is bounded in $\pazocal{D}^\prime_{\varlambda_{\chi}}(M\times M)$
uniformly in $z\in \{\Im z\geqslant 0, \,\vert z\vert\geqslant \varepsilon>0\}$, and $ r_N(z,.)=\pazocal{O}_{\pazocal{D}^\prime_{\varlambda_{\chi}}}( \braket{\Im z}^{-\infty} )  .$
\een
\end{lemm}

  Note that in the above formul\ae, the coefficients $(u_k)_{k=0}^\infty$ of the transport equations do not depend on the mass $m$ nor on the spectral parameter $z$.

\begin{refproof}{Lemma \ref{lemfirg}}
We start from the result from \sec{s:hadamardformal} which says that 
in the sense of distributions in $\pazocal{D}^\prime(\pazocal{U})$,
\beq \label{efps}
\left(P-z\right)\left(\sum_{k=0}^N u_k\mathbf{F}_k(z-m^2,.)\right)={\module{g}^{-\frac{1}{2}}}\delta_{\diag}+(Pu_N)\mathbf{F}_N(z-m^2,.),
\eeq
where $\mathbf{F}_N(z,.)\in \pazocal{C}^{(2-2a)(N+1)-N-n}_{\loc}(\cU)$ with decay in $z$ of the form $\braket{\Im z}^{-a(N+1)}$ for $a\in [0,1]$ by Lemma~\ref{l:regFs} and $\mathbf{F}_N(z,.)$ 
is bounded in $\pazocal{D}^\prime_\varlambda(\pazocal{U})$ by Proposition~\ref{p:boundednessFs}.
Now,
multiplying the parametrix defined near the diagonal by the cut--off
function $\chi$ creates an \emphasize{additional term}. 
Namely, it turns \eqref{efps} into:
$$ \bea 
\left(P-z\right)\left(\sum_{k=0}^N u_k \mathbf{F}_k(z-m^2,.)\chi\right)={\module{g}^{-\frac{1}{2}}}\delta_{\diag}+(Pu_N)\mathbf{F}_N(z-m^2,.)\chi+r_N(z,.) 
\eea $$
where $$r_N(z,.)=2\left\langle \nabla \chi, \nabla\left(\sum_{k=0}^N u_k \mathbf{F}_k(z-m^2,.)\chi\right)\right\rangle + (P\chi) \left(\sum_{k=0}^N u_k \mathbf{F}_k(z-m^2,.)\right).$$ 
The term $r_N(z,.)$ vanishes whenever either $\chi=1$ or $\chi=0$, since 
it involves products with {derivatives} of the diagonal cut-off function $\chi$. This means that $r_N(z,x,y)=0$ for $(x,y)\in \tilde{\pazocal{U}}$, which is near the diagonal $\diag$, and 
also $r_N(z,x,y)=0$ for $(x,y)\notin \pazocal{U}$. This implies 
that the Schwartz kernel $r_N(z,.)\in \pazocal{D}^\prime(M\times M)$ defines a \emphasize{proper operator}. 
Therefore, $r_N(z,.)$ is bounded in $\pazocal{D}^\prime_{\varlambda_{\chi}}$ with a bound of the form
$r_N(z,.)=\pazocal{O}_{\pazocal{D}^\prime_{\varlambda_{\chi}}(M\times M)}(\module{\Im z}^{-\infty})$ since it is defined 
in terms of the distributions
$\mathbf{F}_k(z,.)=\pazocal{O}_{\pazocal{D}^\prime_\varlambda(\pazocal{U}\setminus \diag)}(\module{\Im z}^{-\infty})$ restricted \emphasize{outside the diagonal} and because of its support that we just discussed.
\end{refproof}

We are now ready to conclude that near the diagonal, our parametrix is a good approximation of the  resolvent $(P-z)^{-1}$. 

We denote by $r_N(z)$ the operator with Schwartz kernel $r_N(z,.)$, and similarly for other Schwartz kernels. Recall that $\gamma_\varepsilon$ is the integration contour needed to define the complex powers, see \sec{ss:restocp} and Figure \ref{fig:contour} therein. 

\begin{proposition}\label{l:decompositionresolvent}
Assume that $(M,g)$ is a globally hyperbolic non-trapping Lorentzian scattering space and let $\varepsilon>0$. Set $P=\square_g+m^2$. For every $s\in \mathbb{R}_{\geqslant 0}$, $p\in \mathbb{Z}_{\geqslant 0}$, $m\geqslant0$, 
there exists $N$ large enough s.t.~uniformly in $z\in\gamma_\varepsilon$, we have the identity 
\begin{eqnarray}\label{eq:toinsert}
(P-z)^{-1}=\left(\sum_{k=0}^N u_k \mathbf{F}_k(z-m^2,.)\chi\right)+E_{N,1}(z)+E_{N,2}(z)
\end{eqnarray}
in the sense of operators $\Onorm{1}$, where $E_{N,1}(z)=(P-z)^{-1} r_N(z)$ satisfies
\beq\label{manixmanix}
\wfl{7}(E_{N,1}(z))\subset \varlambda_{\chi}',
\eeq
and $E_{N,2}(z)=(P-z)^{-1}(Pu_N)\mathbf{F}_N(z)\chi$ satisfies $
E_{N,2}(z)=\Oregsh({\braket{z}^{-p}})$. Furthermore if $m\neq0$ and assuming in addition injectivity and non-trapping at $\sigma=m^2$, then the estimates and \eqref{eq:toinsert} hold uniformly in $z\in \gamma_0$.
\end{proposition}

In particular, \eqref{manixmanix} implies that $E_{N,1}(z)$  is \emph{smooth near the diagonal}.

As  explained earlier, there are inevitable losses in decay in $z$ in the high regularity estimates, and the $\pazocal{O}(\bra z\ket^{-p})$ bound requires to choose $N$ extremely large. %Moreover, we insist on the fact 
%that all the distributions defined above depend on $z$ in a controlled way when $\Im z\rightarrow 0^+$ but $\Re z$ stays away from $0$.

\begin{refproof}{Proposition \ref{l:decompositionresolvent}}
Recall that   $(P-z)^{-1}=\Onorm{1}$ along $\gamma_\varepsilon$ by Lemma \ref{lem:pika}.  
{  
 By \eqref{manul2} of Lemma \ref{lemfirg}, for every $(p,a)\in \mathbb{N}\times \mathbb{R}$, for $N$ large enough, {in $\cU$} we have $\braket{z}^{p} (Pu_N)\mathbf{F}_N(z)\chi\in H_{\loc}^{a+n-0}(M\times M)$ uniformly in $z$ along $\gamma_\varepsilon$ 
by the Sobolev embeddings $\pazocal{C}_{\loc}^{a} (M\times M)\hookrightarrow H_{\loc}^{a+n-0}(M\times M)$ recalled in Lemma~\ref{l:holdersobolev} in the appendix. 
For every $b\in \mathbb{R}$, $s\in \mathbb{R}_{\geqslant 0}$, the exterior product $ H^b_{\c}(M)\times H_\c^{-b-s}(M)\ni (v_1,v_2)\mapsto v_1 \otimes v_2\in H^{\inf(b,-b-s,-s)}_\c(M\times M)$
is linear continuous~\cite[Thm.~3.2 p.~140]{joshi}, therefore choosing $N$ large enough so that $a+n+\inf(b,-b-s,-s)>0 $, we find that
$$ 
\bea
 H^b_{\c}(M)\times H_\c^{-b-s}(M)\ni(v_1,v_2)\mapsto &\left\langle v_2, \braket{z}^{p}(Pu_N)\mathbf{F}_N(z)\chi v_1\right\rangle_{M}\\ &= \left\langle \braket{z}^{p}(Pu_N)\mathbf{F}_N(z)\chi, v_1\otimes v_2\right\rangle_{M\times M}
 \eea
  $$
is bilinear continuous by Sobolev duality uniformly in $z$ along $\gamma_\varepsilon$. 
Therefore we have (for $N$ large enough):
\beq\label{njn}
(Pu_N)\mathbf{F}_N(z)\chi = \Oregsh({\braket{z}^{-p}}).
\eeq
Since the operators in \eqref{njn} are proper{ly supported},
$$
E_{N,2}(z)=(P-z)^{-1}(Pu_N)\mathbf{F}_N(z)\chi=\Oregsh({\braket{z}^{-p}})$$ follows by composition.
}

Next, by \eqref{manul3} of Lemma \ref{lemfirg}, $r_N(z, .)=\pazocal{O}_{\pazocal{D}^\prime_{\varlambda_{\chi}}}( \braket{ z}^{-\infty} )$ along $\gamma_\varepsilon$. Consequently,  in terms of the operator wavefront set  we have $\wfl{7}(r_N(z))\subset \varlambda_{\chi}'$ by Lemma \ref{ditoop}. We also know from Theorem \ref{thm:wf} that
\beq
\wfl{12}  \big( ( P-z)^{-1} \big)\subset \varlambda'. 
\eeq
By the composition rule for operator wavefront sets, i.e., by Lemma \ref{lem:composition}, we obtain
\beq\label{manix1}
\wfl{7}(E_{N,1}(z))\subset \varlambda' \circ \varlambda_{\chi}'.
\eeq
It is easy to show using the {transitivity} of the $q_1{\succ} q_2$ relation that $\varlambda$ and $\varlambda_\chi$ satisfy the remarkable property
\beq\label{manix2} 
\varlambda' \circ \varlambda'\subset \varlambda' \mbox{ and }  \varlambda' \circ \varlambda_{\chi}'\subset \varlambda_{\chi}'.
\eeq
From \eqref{manix1} and the second property in \eqref{manix2} we conclude \eqref{manixmanix} immediately.  The case of $\gamma_0$ is fully analogous.
\end{refproof}

\section{Parametrix for complex powers and functions of the wave operator}

%\subsection{Summary} Plugging the resolvent parametrix in contour integrals and holomorphic dependence.

\subsection{The case of complex powers \texorpdfstring{$(P-i\varepsilon)^{-\cv}$}{}} \label{ss:cpparam} Our next objective is to study analytic properties of complex powers in a neighborhood $\pazocal{U}\subset M\times M$ of the diagonal  using the relationship between the resolvent and the Hadamard parametrix shown in Proposition \ref{l:decompositionresolvent}.

Let $\varepsilon>0$. We already know that the contour integral 
$$
(P-i\varepsilon)^{-{\cv}}=\frac{1}{2\pi i}\int_{\gamma_\varepsilon} (z-i\varepsilon)^{-\cv} (P-z)^{-1} dz
$$ makes sense as an operator
in  $\pazocal{B}(L^2(M))$ for $\Re\cv>0$. Using the decay along $\gamma_\varepsilon$ of the various terms stated in Lemma \ref{l:regFs} and Proposition \ref{l:decompositionresolvent}, we can insert the r.h.s.~of \eqref{eq:toinsert} into the above contour integral. For $\Re \cv>0$ this yields
\beq\label{cornell}
\bea 
\left(P-i\varepsilon \right)^{-\cv}&=\frac{1}{2\pi i}\int_{\gamma_{\varepsilon}}(z-i\varepsilon)^{-\cv} (P-z)^{-1} dz\\
&=\sum_{\varm=0}^N \chi u_\varm\frac{1}{2\pi i}\int_{\gamma_{\varepsilon}}(z-i\varepsilon)^{-\cv}  \mathbf{F}_\varm(z) dz+  \frac{i}{2\pi}\int_{\gamma_{\varepsilon}}(z-i\varepsilon)^{-\cv}\left(E_{N,1}(z)+E_{N,2}(z)\right)dz,
\eea 
\eeq
and this extends to $\cv\in\cc$ provided we check that the summands have an analytic continuation.

By Lemma ~\ref{l:holofamilyFm} and continuity of the pull-back by $G$, 
we know that
$$\frac{1}{2\pi i}\int_{\gamma_\varepsilon}(z-i\varepsilon)^{-\cv} \mathbf{F}_\varm(z,.) dz =\frac{(-1)^\varm\Gamma(-\cv+1)}{\Gamma(-\cv-\varm+1)\Gamma(\cv+\varm)}   \mathbf{F}_{\varm+\cv-1}(i\varepsilon,.)$$
which is a well-defined holomorphic family of distributions in $\pazocal{D}^\prime(\pazocal{U})$.
Therefore the finite sum
$\sum_{\varm=0}^N \chi u_\varm\frac{1}{2\pi i}\int_{\gamma_\varepsilon}(z+i\varepsilon)^{-\cv} \mathbf{F}_\varm(z,.) dz$ in fact reads
$$ \sum_{\varm=0}^N \chi u_\varm\frac{(-1)^\varm\Gamma(-\cv+1)}{\Gamma(-\cv-\varm+1)\Gamma(\cv+\varm)} \mathbf{F}_{\varm+\cv-1}(-i\varepsilon,.)$$
and is a well-defined holomorphic family of distributions in the parameter $\cv\in \mathbb{C}$.

The error term
$$
R_N(z,\cv)\defeq  \frac{i}{2\pi}\int_{\gamma_{\varepsilon}}(z-i\varepsilon)^{-\cv}\left(E_{N,1}(z)+E_{N,2}(z)\right)dz
$$
is smooth near the diagonal.

It follows that 
we have a decomposition of the Schwartz kernel of $(P-i\varepsilon)^{-\cv}$ which has to be understood in the sense of germs of distributions defined near the diagonal $\diag\subset M\times M$. The germ is the only information we need to take the diagonal restrictions: 
\begin{lemma} \label{l:decfeynmanpowers}  Let $(M,g)$ and $P$ be as in Proposition \ref{l:decompositionresolvent}. Then for every $\varalpha\in \mathbb{R}_{\geqslant 0}$, $p\in \mathbb{N}$, there exists $N\geqslant 0$ s.t.~we have the decomposition in $\pazocal{D}^\prime(\tilde{\pazocal{U}})$:
\begin{eqnarray}
(P-i\varepsilon)^{-\cv}=\sum_{\varm=0}^N u_\varm\frac{(-1)^\varm\Gamma(-\cv+1)}{\Gamma(-\cv-\varm+1)\Gamma(\cv+\varm)} \mathbf{F}_{\varm+\cv-1}(-m^2+i\varepsilon)+R_N(i\varepsilon,\cv) \\
{R_N(i\varepsilon,\cv)\in C^\varalpha(\pazocal{U})}
\end{eqnarray}
where the terms on the r.h.s.~depend 
holomorphically on $\cv$ in the half-plane $\Re\cv>-p$ and the r.h.s.~is well-defined as an element of $\pazocal{D}^\prime(\tilde{\pazocal{U}})$.
\end{lemma}

\begin{coro}\label{c:cororegfeynpowers}  
For $\Re\cv\geqslant n+\varalpha$ (where $n=\dim M$) and $\varalpha>0$, we find that
$(P-i\varepsilon)^{-\cv}$ has H\"older regularity $\pazocal{C}^{\varalpha}_{\loc}(\pazocal{U})$
in a neighborhood of the diagonal $\diag\subset M\times M$, and under the non-trapping and injectivity assumption at $\sigma=m^2\neq 0$  we get existence of  $\lim_{\varepsilon\rightarrow 0^+} (P-i\varepsilon)^{-\cv}$ in the sense of $\pazocal{C}^{\varalpha}_{\loc}(\pazocal{U})$.
\end{coro}

\begin{rema}[Checking the combinatorial factors]
In the limit $\cv\rightarrow 1$, we get 
$$ \frac{(-1)^\varm\Gamma(-\cv+1)}{\Gamma(-\cv-\varm+1)\Gamma(\cv+\varm)}\rightarrow  1$$
since the poles of $\Gamma(-\cv+1)$ and $\Gamma(-\cv-\varm+1) $ compensate. Therefore,  Lemma \ref{l:decfeynmanpowers} is consistent with the formula $(P-i\varepsilon)^{-1}=\sum_{\varm=0}^N u_\varm   \mathbf{F}_{\varm}(i\varepsilon)+R_N(i\varepsilon,\cv)$ 
as expected.
\end{rema}

\subsection{The case of \texorpdfstring{$f(\frac{P+i\varepsilon}{\scalambda^2})$}{functions of P}} With the spectral action in mind we now discuss other functions of $P$. For that purpose it is actually slightly more convenient to work with $P+i\varepsilon$ instead of $P-i\varepsilon$, which in practice amounts to considering $-P$ instead of $P$. Note that $(P+i\varepsilon)^{-1}$ and the corresponding Hadamard parametrix have \emph{anti-Feynman} rather than Feynman wavefront set because the sinks and sources are interchanged, but this has no practical significance in the discussion below. 

\begin{defi} We denote by $S^{-\infty}_+(\mathbb{R})$ the set of Schwartz functions $f$ such that $\widehat{f}\in C^\infty_\c(\open{0,\infty})$. 
\end{defi}

First, observe that by the Paley--Wiener theorem, each $f\in S^{-\infty}_+(\mathbb{R})$ has a unique 
holomorphic extension to the upper half--plane $\{\Im z\geqslant 0\}$ and that the analytic extension, 
still denoted by $f$, has exponential decay when $\Im z\rightarrow +\infty$. Let us  also note that
$f(.+i\varepsilon)\in L^\infty(\mathbb{R})\cap C^\infty(\mathbb{R})$.
%For every $f\in S^{-\infty}_+(\mathbb{R})$, we would like to define and to study the asymptotics of $f(\frac{P+i0}{\scalambda^2})$ when $\scalambda\rightarrow +\infty$.
 Recall that the Mellin transform
of $\widehat f\in C^\infty_\c(\open{0,\infty})$
is by definition the function
$$ \bea 
\pazocal{M}\widehat{f}(\cv)=\int_0^\infty \tau^{\cv-1}\widehat f(\tau) d\tau.
\eea $$
By assumption on $\widehat{f}$, the Mellin transform $\pazocal{M}\widehat{f}$ has fast decay, and the Mellin inversion yields
$\widehat{f}(\tau)=\frac{1}{2\pi i}\int_{\Re\cv=c} \tau^{-\cv}\pazocal{M}\widehat{f}(\cv)ds$ where the integral is absolutely convergent 
uniformly in $\tau\in K\subset \open{0,+\infty}$ for any compact $K$.
By inverse Fourier transform of $\tau^{-\cv}_+$, for every $\varepsilon>0$, 
we have the formula:
$$ \bea 
f(t+i\varepsilon)=\frac{1}{2\pi i}\int_{\Re \cv=c} e^{i\cv\frac{\pi}{2}}(t+i\varepsilon)^{-\cv}\Gamma(\cv) \pazocal{M}\widehat{f}(\cv)d\cv.
\eea $$
The left hand side makes sense when $\varepsilon>0$ since $f$ has an analytic continuation to the upper half plane $\{\Im z\geqslant 0\}$. 
Note that for $\varepsilon>0$ the integral on the r.h.s.~converges absolutely and that 
for $t\in K$ in some compact $K\subset \mathbb{R}\setminus \{0\}$ away from $0$, the integral $$f(t)=\lim_{\varepsilon\rightarrow 0^+}\frac{1}{2\pi i}\int_{\Re \cv=c} e^{i\cv\frac{\pi}{2}}(t+i\varepsilon)^{-\cv}\Gamma(\cv) \pazocal{M}\widehat{f}(\cv)d\cv$$
also converges absolutely. 
This allows to give a representation formula for $f(P+i\varepsilon)$ involving the complex powers: 
\begin{eqnarray}\label{e:mellinf}
f(P+i\varepsilon)=\frac{1}{2\pi i}\int_{\Re\cv=c} e^{i\cv\frac{\pi}{2}}(P+i\varepsilon)^{-\cv}\Gamma(\cv) \pazocal{M}\widehat{f}(\cv)ds.
\end{eqnarray}
{As long as $\varepsilon>0$, the integral on the r.h.s.~is norm convergent in $\pazocal{B}(L^2(M))$ and the identity follows
 by  Borel functional calculus.} It should be noted that one can choose $c>0$ arbitrarily 
large on the r.h.s.; this does not affect 
the convergence properties. 

Recall $\pazocal{U}$ is  a  neighborhood of the diagonal $\diag\subset M\times M$.       Setting now $c>\dim(M)+\varalpha$ for $\varalpha>0$,
we know by Corollary~\ref{c:cororegfeynpowers} 
that the Schwartz kernel $(P+i\varepsilon)^{-\cv}$ belongs to $\pazocal{C}^\varalpha_{\loc}(\pazocal{U})$
uniformly in $\cv\in \{\Re\cv=c\}$. Note that to take the limit $\varepsilon\rightarrow 0^+$, we 
need to assume $m\neq 0$ and non-trapping and injectivity at $m^2$, in which case $\lim_{\varepsilon\rightarrow 0^+}(P+i\varepsilon)^{-\cv}=(P+i0)^{-\cv}$ has Schwartz kernel in
$\pazocal{C}^\varalpha_{\loc}(\pazocal{U})$
uniformly in $\cv\in \{\Re\cv=c\}$.

Therefore by 
the fast decay of $\pazocal{M}\widehat{f}(\cv)$ on 
the vertical line $\{\Re\cv=c\}$, 
we find that $f(P+i\varepsilon)$ 
has Schwartz kernel which belongs to $\pazocal{C}^\varalpha_{\loc}(\pazocal{U})$. If $m^2$ is non-trapping, the same result holds for
$\lim_{\varepsilon\rightarrow 0^+} f(P+i\varepsilon)=f(P+i0)$  
which implies we can take $\varepsilon\rightarrow 0^+$ on the r.h.s.~of 
\eqref{e:mellinf} which makes sense as Schwartz kernel near 
$\diag\subset M\times M$ and one can take the restriction on the diagonal 
as $f(P+i0)(x,x)$. 
If $\varepsilon>0$, we do not need the mass term
and both sides of \eqref{e:mellinf}
make sense as operators acting on $L^2$ whose Schwartz kernel is $\pazocal{C}^\varalpha$ 
near the diagonal $\diag$.
 
So if $m^2$ non-trapping, in the limit $\varepsilon\rightarrow 0^+$, 
we take the formula on the r.h.s.~of \eqref{e:mellinf} 
as definition
of $f(P+i0)$ and both sides are no longer 
viewed as operators but as
germs of Schwartz kernels defined near 
the diagonal $\diag\subset M\times M$:
\begin{eqnarray}\label{e:mellinf2}
f(P+i0)(.,.)\overset{\text{def}}{=}\lim_{\varepsilon\rightarrow 0^+}\frac{1}{2\pi i}\int_{\Re\cv=c} e^{i\cv\frac{\pi}{2}}(P+i\varepsilon)^{-\cv}(.,.)\Gamma(\cv) \pazocal{M}\widehat{f}(\cv)ds\in \pazocal{C}^\varalpha_\loc(\pazocal{U}).
\end{eqnarray}
Since $\varalpha$ and hence the parameter $c$ can be chosen arbitrarily large, 
we proved:
\begin{lemm}
Let $(M,g)$ be a globally hyperbolic non-trapping Lorentzian scattering space and let $\varepsilon>0$, $P=\square_g+m^2$, $m\geqslant 0$.
Then for all $f\in S^{-\infty}_+(\mathbb{R})$, 
$f(P+i\varepsilon):L^2(M)\to L^2(M)$ exists and has
\emphasize{smooth} Schwartz kernel in some neighborhood
$\pazocal{U}$ of the diagonal $\diag\subset M\times M$.

Moreover, if  $m\neq 0$ and assuming injectivity and  non-trapping at energy $\sigma=m^2$,  for all $f\in S^{-\infty}_+(\mathbb{R})$, 
the limit:
\begin{eqnarray}\label{eqthislimit}
f(P+i0)(.,.)\defeq\lim_{\varepsilon\rightarrow 0^+}\frac{1}{2\pi i}\int_{\Re\cv=c} e^{i\cv\frac{\pi}{2}}(P+i\varepsilon)^{-\cv}(.,.)\Gamma(\cv) \pazocal{M}\widehat{f}(\cv)ds
\end{eqnarray}
exists in $\pazocal{D}^\prime_\varlambda(M\times M)$ and is \emphasize{smooth} in a neighborhood
$\pazocal{U}$ of the diagonal $\diag\subset M\times M$.
\end{lemm}

\section{Diagonal restriction, poles and residues}

\subsection{Summary}  

We make the central observation that
$\Fso{z}=\mathbf{F}_{\cv}(z,x,x)$ for every $x\in M$ by construction, therefore to study  
the restriction on the diagonal we only need to study the analytic continuation in $\cv$ of $F_\cv(z, 0)$.

\subsection{Meromorphic continuation of \texorpdfstring{$\Fso{z}$}{F(z)}}

The value at $0\in \mathbb{R}^n$ of the distribution $\Fse{z}$ that we denote by $\Fso{z}$ is studied.
By the H\"older regularity shown in Proposition~\ref{p:fsholder}, $\Fse{z}\in \pazocal{C}^\varalpha(\rr^n)$ for $\varalpha\leqslant \Re\cv+1-n$, therefore
the value $\Fso{z}$ at $0$ is well-defined when $\Re\cv>n-1$ and depends holomorphically on $\cv\in \{\Re\cv>n-1\}$, and also the limit when $\Im z\rightarrow 0^+$, $\vert z\vert\geqslant m^2>0$ is well-defined.
We prove it admits an 
analytic continuation as a meromorphic function with a simple pole
at $\cv=\{\frac{n}{2}-1,\dots,\frac{n}{2}-k,\dots\}$. 

\subsubsection{A warm-up calculation.}
\label{sss:warmup}
The pole of $\Fso{z}$ comes from its  representation as an integral of symbols on 
cones, where the decay of the symbol approaches the critical dimension of the cone.
We will also have to take into account the $\Gamma$ function factor.  
The typical example reads:
$$ \bea 
\int_{\mathbb{R}^n}(\Vert\xi\Vert^2-z)^{-\cv}d^n\xi
\eea $$
where $\norm{\xi}$ is the Euclidean norm of $\xi$. We assume for the moment $z<0$ so that $z$ acts as a mass squared to regulate infrared divergences 
since we only want to deal with {\rm UV} problems.

To extract residues of such integrals, 
observe that
the poles of 
$$\int_{\mathbb{R}^n}(\Vert\xi\Vert^2-z)^{-\cv}d^n\xi=\frac{1}{\Gamma(\cv)} \int_0^\infty \left( \int_{\mathbb{R}^n} e^{-t(\Vert\xi\Vert^2-z)}d^n\xi\right) t^{\cv-1} dt$$
are 
the poles of
$$ \bea 
&\frac{1}{\Gamma(\cv)}  \int_0^1 \left( \int_{\mathbb{R}^n} e^{-t(\Vert\xi\Vert^2-z)}d^n\xi\right) t^{\cv-1} dt \\ &=\frac{(2\pi)^n}{\Gamma(\cv)(4\pi)^{\frac{n}{2}}}
 \sum_{k=0}^\infty \frac{z^k}{k!} \int_0^1  t^{\cv-\frac{n}{2}+k-1} dt=\frac{\pi^{\frac{n}{2}}}{\Gamma(\cv)}
 \sum_{k=0}^\infty \frac{z^k}{k!(\cv-\frac{n}{2}+k)}. 
\eea $$
All poles are simple and located at
$\cv=\frac{n}{2}, \dots, 1 $ and have $z^k$ in the {numerator}; there are compensations for $\cv\in -\mathbb{N}$ due to the $\Gamma$ factor. 
So the residue at $\cv=k$, $k\in \{1,\dots,\frac{n}{2}-1\}$
is
\begin{equation}
\Res_{\cv=k}\int_{\mathbb{R}^n}(\Vert\xi\Vert^2-z)^{-\cv}d^n\xi=\frac{z^{\frac{n}{2}-k}\pi^{\frac{n}{2}}}{(\frac{n}{2}-k)! \Gamma(k)}.
\end{equation}

\subsection{The Wick rotation by homological methods}\label{ss:Wick}

We  need to deal with similar integrals as in the above paragraph but with 
the Minkowski quadratic form instead of the Euclidean one. 
We present a geometric approach
to the analytic continuation of the residue which is close
to the Wick rotation in the physics literature 
but is fully rigorous. 

Consider $\mathbb{C}^n$ viewed as a K\"ahler manifold with coordinates $(z_1,\dots,z_n)$, and some complex parameter $u\in \mathbb{C}$ that will  take value in the upper half-plane $\{\Im u>0\}$. Set $Q(z)=\sum_{i=1}^n z_i^2$ and consider
the complex valued $n$-form
$$\omega_{\cv}=\left(\sum_{i=1}^n z_i^2-u\right)^{-\cv}dz_1\wedge\dots\wedge dz_n\in \Omega^{n,0}\left( U  ;\mathbb{C}\right) $$
which is well-defined on the Zariski open set $U=\{z\in\cc^n \st Q(z)-u\notin \opencl{-\infty,0}  \}$ since we chose the usual branch of the $\log$ which avoids 
the negative reals.

For all $\theta\in \open{-\frac{\pi}{2},\frac{\pi}{2}}$ let $P_\theta=\{ (e^{i\theta}z_1,\dots,z_n) \st (z_1,\dots,z_n)\in \mathbb{R}^n \}$,
considered as an oriented $n$--chain. 

\begin{prop}[Stokes' theorem]\label{p:stokes}
For all $\cv\in \mathbb{C}$,
the differential form $\omega_{\cv}$ is closed and 
for all $\theta\in \open{-\frac{\pi}{2},\frac{\pi}{2}}$ and all $\Re\cv>\frac{n}{2}$:
\begin{eqnarray}
\int_{P_\theta} \omega_{\cv}=\int_{P_0} \omega_{\cv}.
\end{eqnarray}
\end{prop}
\begin{proof}
It is obvious that $d\omega_{\cv}=0$.
One has to be careful when applying Stokes' theorem since our ``cycles'' have their boundaries along sectors at infinity. Denote by 
$B(R)=\{\Vert z\Vert^2\leqslant R^2\}$ the ball of radius $R$ in $\mathbb{C}^n$.

After intersecting our chains $P_\theta$ with $B(R)$, the usual Stokes' theorem yields
$$ \bea 
\int_{P_\theta\cap B(R)} \omega_{\cv}-\int_{P_0\cap B(R)} \omega_{\cv}=\underset{=0}{\underbrace{\int_{D_\theta} d\omega_{\cv}}}-\int_{R_\theta}\omega_{\cv},
\eea $$
where $D_\theta$ is the angular sector  $D_\theta=\{ (z_1,\dots,z_n) \st 0\leqslant \arg(z_1)\leqslant \theta, \, (z_2,\dots,z_n)\in \mathbb{R}^n  \}\cap B(R)$ and $R_\theta=\{ (z_1,\dots,z_n) \st  0\leqslant \arg(z_1)\leqslant \theta,\, (z_2,\dots,z_n)\in \mathbb{R}^n , \, \Vert z\Vert=R \} $.
Let us bound the integral on the arc $R_\theta$,
$$ \bea 
\int_{R_\theta}\omega_{\cv}&= \int_{\{ \sum_2^nz_i^2\leqslant R^2 \}} \left( \int_{0}^\theta \left( e^{2ia}(R^2-\sum_{i=2}^n z_i^2-u)+ \sum_{i=2}^n z_i^2 \right)^{-\cv} ie^{ia}da \right) dz_2\dots dz_n\\
&\leqslant  C R^{n-1} R^{-2\Re\cv}
 \eea $$
which tends to $0$ as $R\rightarrow +\infty$. Since for all $\theta\in \open{-\frac{\pi}{2},\frac{\pi}{2}}$ the integral
$\int_{P_\theta} \omega_{\cv} $ converges absolutely when $\Re\cv>\frac{n}{2}$, we can take the limit $R\rightarrow +\infty$ which yields
\begin{eqnarray}
\lim_{R\rightarrow +\infty}\int_{P_\theta\cap B(R)} \omega_{\cv}=\lim_{R\rightarrow +\infty}\int_{P_0\cap B(R)} \omega_{\cv}.
\end{eqnarray}  
\end{proof}

It follows from
the identity $\int_{P_\theta} \omega_{\cv}=\int_{P_0} \omega_{\cv}$ for $\Re\cv>\frac{n}{2}$ as holomorphic functions, and from the fact that
$\int_{P_0} \omega_{\cv}$ is a meromorphic function with simple poles at
$\cv\in \{\frac{n}{2},\frac{n}{2}-1,\dots,1\}$, that both sides coincide
in the sense of meromorphic functions for all $\cv\in \mathbb{C}\setminus\{\frac{n}{2},\dots,1\}$
by analytic continuation in $\cv$.
Define the linear invertible holomorphic map $\Phi_\theta:(z_1,\dots,z_n)\mapsto (e^{i\theta}z_1,\dots,z_n)$.
Since $\Phi_\theta$ is invertible and does not reverse orientations, we get by the pull-back Theorem:
$$ \bea 
\int_{P_\theta}\omega_{\cv}=\int_{\Phi_\theta(P_0)}\omega_{\cv}=\int_{P_0}\Phi_\theta^*\omega_{\cv}.
\eea $$
Combining with the equality $\int_{P_\theta} \omega_{\cv}=\int_{P_0} \omega_{\cv} $, this means that
$$ \bea 
&\int_{\mathbb{R}^n}\bigg(\sum_{i=1}^n z_i^2-u\bigg)^{-\cv}dz_1\wedge\dots\wedge dz_n  =\int_{P_0} \omega_{\cv}=\int_{P_0}\Phi_\theta^*\omega_{\cv}\\ &=e^{i\theta} \int_{\mathbb{R}^n}\bigg( e^{2i\theta}z_1^2+ \sum_{i=2}^n z_i^2-u\bigg)^{-\cv}dz_1\wedge\dots\wedge dz_n. 
\eea $$
When $\theta\rightarrow -\frac{\pi}{2}^+$, the term $\left( e^{2i\theta}z_1^2+ \sum_{i=2}^n z_i^2-u\right)$ has non-positive 
imaginary part and the integrand  
$\left( e^{2i\theta}z_1^2+ \sum_{i=2}^n z_i^2-u\right)^{-\cv}$ converges to 
$(-z_1^2+  \sum_{i=2}^n z_i^2-u-i0)^{-\cv}$ in the sense of distributions by Lemma~\ref{l:wick} proved in the appendix. By weak homogeneity at infinity and using a Littlewood--Paley 
decomposition {$1=\sum_j\chi_0(\vert z\vert ) +\psi(2^{-j}\vert z\vert)$ as in \sec{ss:holderestimates}}, 
one can show that:
\begin{eqnarray}\label{e:wick}
\lim_{\theta\rightarrow -\frac{\pi}{2}^+}\int_{\mathbb{R}^n}\bigg( e^{2i\theta}z_1^2+ \sum_{i=2}^n z_i^2-u\bigg)^{-\cv}dz_1\wedge\dots\wedge dz_n =\int_{\mathbb{R}^n} (-z_1^2+  \sum_{i=2}^n z_i^2-u-i0)^{-\cv}d^nz,
\end{eqnarray}
where the bound 
$$\bea
&\sup_{\theta\in \opencl{-\frac{\pi}{2},0}} \bigg| \int_{\mathbb{R}^n}\bigg( e^{2i\theta}z_1^2+ \sum_{i=2}^n z_i^2-u\bigg)^{-\cv} \psi(2^{-j}\vert z\vert) dz_1\wedge\dots\wedge dz_n \bigg|\\
&=\sup_{\theta\in \opencl{-\frac{\pi}{2},0}} 2^{j(n-2\Re\cv)} \bigg| \bigg\langle (e^{2i\theta}z_1^2+ \sum_{i=2}^n z_i^2-\frac{u}{2^{2j}}-i0)^{-\cv},\psi \bigg\rangle \bigg|  \leqslant C 2^{j(n-2\Re\cv)} 
\eea
$$
ensures that both sides of the above equality (\ref{e:wick}) can be written as convergent series and both sides are \emphasize{holomorphic} in $\cv$ when $\Re\cv>\frac{n}{2}$.

By Proposition~\ref{p:stokes} and the warm-up calculation~\sec{sss:warmup},
for all $\theta\in \opencl{-\frac{\pi}{2},0}$, 
$$\int_{\mathbb{R}^n}\bigg( e^{2i\theta}z_1^2+ \sum_{i=2}^n z_i^2-u\bigg)^{-\cv}dz_1\wedge\dots\wedge dz_n=e^{-i\theta}\int_{P_0}\omega_\alpha$$ extends as a meromorphic function in $\cv$ with simple poles at
$\cv=\frac{n}{2},\dots,1$. Hence the limit on the l.h.s of equality~(\ref{e:wick}) which equals $i\int_{P_0}\omega_\alpha$ also does.
Therefore the r.h.s of equation~(\ref{e:wick}) equals $i\int_{P_0}\omega_\alpha$ which is meromorphic with simple poles at $\cv=\frac{n}{2},\dots,1$, this finally yields

$$ \bea 
\boxed{\Res_{\cv=k}\int_{\mathbb{R}^n} \bigg(-z_1^2+  \sum_{i=2}^n z_i^2-u-i0\bigg)^{-\cv}d^nz=i \Res_{\cv=k}\int_{\mathbb{R}^n}\bigg(\sum_{i=1}^n z_i^2-u\bigg)^{-\cv}d^nz=\frac{i\pi^{\frac{n}{2}}{u}^{\frac{n}{2}-k}}{\Gamma(k) (\frac{n}{2}-k)! }. }
\eea $$

\subsubsection{Conclusion and  structure of  residues}

%In general, for every $p\in \{0,\dots,\frac{n}{2}-1\}$,
%$$
%\Res_{\cv=\frac{n}{2}-p} \int_{\mathbb{R}^n}(Q(\xi)-z-i0)^{-\cv}d^n\xi=\pazocal{O}(z^p).
%$$
Therefore, we can go back to the residue of the diagonal restriction of $\mathbf{F}_{\cv}(z,x,x)$, for all $z\neq 0$, $\Im z>0$, 
we find that
$$ \bea 
 \Res_{\cv=\frac{n}{2}-\varm}\Gamma(\cv+\varm)^{-1} \mathbf{F}_{\cv+\varm-1}(\pm z,x,x)&=\frac{1}{(2\pi)^n}\Res_{\cv=\frac{n}{2}-\varm} \int_{\mathbb{R}^n}(Q(\xi)\mp(z+i0))^{-\cv-\varm}d^n\xi \\
  &= \pm\frac{ i }{2^n\pi^{\frac{n}{2}} (\frac{n}{2}-1)!}.
\eea $$
The poles of $\Gamma(\cv+\varm)^{-1}\mathbf{F}_{\cv+\varm-1}(\pm z,x,x)$ occur at $\cv=\{1-\varm,\dots,\frac{n}{2}-\varm\}$. For $j\in \{1-\varm,\dots,\frac{n}{2}-\varm\}$:
\begin{eqnarray*}
 \Res_{\cv=j}\Gamma(\cv+\varm)^{-1} \mathbf{F}_{\cv+\varm-1}(\pm z,x,x)&=& \pm \frac{i(\pm z)^{\frac{n}{2}-\varm-j}}{\Gamma(j+\varm)2^n\pi^{\frac{n}{2}}(\frac{n}{2}-\varm-j)! }.
\end{eqnarray*}

 Applying this result to the 
parametrix of the Feynman powers, 
we note that $$\frac{\Gamma(-\cv+1)}{\Gamma(-\cv-k+1)}=(-\cv)\dots (-\cv-k+1),$$
which means that
this term has no poles 
for every $\cv\in \mathbb{C}$ and \emphasize{does not contribute} to the residues of $(P\pm i\varepsilon)^{-\cv}(x,x)$.
We would like to study the residue in the variable $s$ of $\left(P\pm i\varepsilon\right)^{-\cv}$
at $p=\frac{n}{2}, \frac{n}{2}-1, \dots$ in two cases.

\case{1} When $p=\frac{n}{2},\dots,1$, we find that for $\cv$ near $p$
$$ \bea 
(P\pm i\varepsilon)^{-\cv}=\sum_{k=0}^\infty \underset{{\neq 0}}{\underbrace{(-1)^k(-\cv)\dots (-\cv-k+1) }}\underset{\text{simple poles at }\{1-k,\dots,\frac{n}{2}-k\}}{\underbrace{\Gamma(\cv+k)^{-1} \mathbf{F}_{\cv+k-1}( -m^2 \mp i\varepsilon ,x,x)}} 
\eea $$
which implies that
$$ \bea 
\Res_{\cv=p}(P\pm i\varepsilon)^{-\cv}=\sum_{k=0}^{\frac{n}{2}-p}(-1)^k(-p)\dots (-p-k+1)  \Res_{\cv=p}\Gamma(\cv+k)^{-1} \mathbf{F}_{\cv+k-1}( -m^2 \mp i\varepsilon ,x,x).
\eea $$
The only residue which is independent of $\varepsilon,m$ reads:
$$ \bea 
&\Res_{\cv=\frac{n}{2}-k} (-1)^k \frac{ u_k(x,x)\mathbf{F}_{k+s-1}(-m^2 \mp i\varepsilon,x,x)\Gamma(-\cv+1)}{\Gamma(\cv+k)\Gamma(-\cv-k+1)}\\
&=\mp u_k(x,x)(k-\frac{n}{2})\dots (1-\frac{n}{2})\frac{ (-1)^{k}i }{2^n\pi^{\frac{n}{2}} \Gamma(\frac{n}{2})}\\
&=\mp  u_k(x,x)\frac{(\frac{n}{2}-1)! }{(\frac{n}{2}-k-1)!}\frac{ i }{2^n\pi^{\frac{n}{2}} \Gamma(\frac{n}{2})}\\
&=\mp  \frac{ u_k(x,x)}{(\frac{n}{2}-k-1)!}\frac{ i }{2^n\pi^{\frac{n}{2}}}.  
\eea $$

\case{2} When $p\leqslant 0$, we find for $\cv$ near $p$ that:
$$ \bea 
(P\pm i\varepsilon)^{-\cv}=\sum_{k=0}^\infty \underset{\text{simple zeroes for }k\geqslant 1-p }{(-1)^k\underbrace{ (-\cv)\dots (-\cv-k+1) }}\underset{\text{simple poles at }\{1-k,\dots,\frac{n}{2}-k\}}{\underbrace{\Gamma(\cv+k)^{-1} \mathbf{F}_{\cv+k-1}(-m^2 \mp i\varepsilon,x,x)}} 
\eea $$
therefore
$$ \bea 
\Res_{\cv=p}(P\pm i\varepsilon)^{-\cv}=\sum_{k=0}^{-p}(-1)^k (-p)\dots (-p-k+1) \Res_{\cv=p}\underset{\text{no poles at }s=p}{\underbrace{\Gamma(\cv+k)^{-1} \mathbf{F}_{\cv+k-1}(-m^2 \mp i\varepsilon,x,x)}} 
=0
\eea $$
where $\Gamma(\cv+k)^{-1}\mathbf{F}_{\cv+k-1}$ has
no poles at $\cv=p$ because $p+k\leqslant 0$ and 
$\Gamma(\cv+k)^{-1}
F_{\cv+k-1}$ has no poles in the region $\Re\cv+k\leqslant 0$.

\begin{rema}
We recover the well-known fact that on a compact Riemannian manifold, if $0\notin\sp(-\Delta_g)$ then the meromorphic continuation of $\{ \Re \cv \gg 0 \} \ni\cv\mapsto \Tr((-\Delta_g)^{-\cv})$ has no pole at $\cv=0,1,\dots,k,\dots$. 
\end{rema}

In summary, we have proved the following result. 

\begin{thm}\label{mainmainthm}
Let $(M,g)$ be a globally hyperbolic  non-trapping Lorentzian scattering space of even dimension $n$ and let {$P=\square_g$}. Then 
the Schwartz kernel $K_s(.,.) $ of $ (P\pm i\varepsilon)^{-\cv} $ exists as a family of distributions near the diagonal depending holomorphically in $\cv$ on the half-plane $\Re\cv>-1$.
Its restriction on the diagonal
$K_\cv(x,x)$ exists and is holomorphic for $\Re\cv>\frac{n}{2} $, and it extends as a meromorphic 
function of $\cv$ with simple poles along the arithmetic progression
$\{\frac{n}{2}, \frac{n}{2}-1,\dots,1\}$. Furthermore,
$$
 \bea 
\lim_{\varepsilon\rightarrow 0^+} \Res_{\cv=\frac{n}{2}-\varm}(P\pm i\varepsilon)^{-\cv}(x,x)= 
\mp\frac{ i \,u_\varm(x,x)}{2^n\pi^{\frac{n}{2}}(\frac{n}{2}-\varm-1)!}.
\eea $$
\end{thm}

\subsection{Residues for the spectral action}

To recover the spectral action principle, we must study the pole
structure of the restriction $\Gamma(\cv)(P+i\varepsilon)^{-\cv}(x,x)$ where we must take into account
the non trivial effect of the $\Gamma$ factor for  $\cv\leqslant 0$. 
For usual applications of the spectral action principle, we need the
first three poles at $\cv=\frac{n}{2},\frac{n}{2}-1,\frac{n}{2}-2$ of 
$\Gamma(\cv)(P\pm i\varepsilon)^{-\cv}(x,x)$ that we will explicitely calculate in terms of the mass term $m^2$ and
the regulator $\varepsilon$, after tedious bookkeping of all the formulas from the previous paragraph we find:
$$
\bea
\Res_{\cv=\frac{n}{2}}\Gamma(\cv) (P\pm i\varepsilon)^{-\cv}(x,x)&= 
\mp\frac{ i }{2^n\pi^{\frac{n}{2}}}\\
\Res_{\cv=\frac{n}{2}-1}\Gamma(\cv)(P\pm i\varepsilon)^{-\cv}(x,x)&= 
\mp \frac{i(-m^2\mp i\varepsilon)}{2^n\pi^{\frac{n}{2}}}\mp  \frac{ iu_1(x,x)}{2^n\pi^{\frac{n}{2}}}\\
\Res_{\cv=\frac{n}{2}-2}\Gamma(\cv)(P\pm i\varepsilon)^{-\cv}(x,x)&= \mp \frac{i(-m^2\mp i\varepsilon)^2}{2^{n+1}\pi^{\frac{n}{2}}}\mp \frac{i(-m^2\mp i\varepsilon)u_1(x,x)}{2^n\pi^{\frac{n}{2}}}  \mp   \frac{ iu_2(x,x)}{2^{n}\pi^{\frac{n}{2}}}.
\eea
$$
In conclusion, this yields the following result.

\begin{thm}\label{thm:spectralaction}
Let $(M,g)$ be a globally hyperbolic non-trapping Lorentzian scattering space of even dimension $n$ and let $P=\square_g+m^2$, $m\geqslant 0$, the corresponding 
Klein--Gordon or wave operator. Then for every $f\in S^{-\infty}_+(\mathbb{R})$,
the Schwartz kernel 
$f(\frac{P {+} i\varepsilon}{\scalambda^2})(.,.) $ is smooth near the diagonal and admits an asymptotic expansion of the form:
$$
\bea
f\left(\frac{P  + i\varepsilon}{\scalambda^2}\right)(x,x)&= \frac{  e^{i\frac{n\pi}{4}} c_0 }{i2^n\pi^{\frac{n}{2}}}  \scalambda^n+
  \frac{ e^{i\frac{(n-2)\pi}{4}}c_1 }{i2^n\pi^{\frac{n}{2}}} \left((-m^2- i\varepsilon)+  u_1(x,x)\right) \scalambda^{n-2} \fantom
  +\frac{ e^{i\frac{(n-4)\pi}{4}} c_2}{i2^n\pi^{\frac{n}{2}}} \left( \frac{(-m^2- i\varepsilon)^2}{2}+(-m^2- i\varepsilon)u_1(x,x) + u_2(x,x)  \right)\scalambda^{n-4} \fantom  +\cO(\scalambda^{n-5}),
\eea
$$
where $u_k$ are the Hadamard coefficients and
$
c_k=\int_0^\infty \widehat{f}(t)t^{\frac{n}{2}-k-1}dt$.
\end{thm}

We remark that in the case when  $m\neq 0$,  assuming injectivity and  non-trapping at energy $\sigma=m^2$ we have an analogous result with $\varepsilon=0$ and with $f\left(\frac{P  + i 0}{\scalambda^2}\right)(x,x)$ defined by \eqref{eqthislimit}.

\begin{refproof}{Theorem \ref{thm:spectralaction}}
Since we are interested  in the three first terms of the asymptotic 
expansion, we  choose $ \frac{n}{2}-3< c_2<\frac{n}{2}-2$.
By Lemma~\ref{l:decfeynmanpowers},
we can start from the diagonal expansion for 
$(P {+} i\varepsilon)^{-\cv}$:
\begin{eqnarray}
(P{+} i\varepsilon)^{-\cv}=\sum_{\varm=0}^N (-1)^\varm u_\varm \frac{\Gamma(-\cv+1)}{\Gamma(\cv+\varm)\Gamma(-\cv-\varm+1)}  \mathbf{F}_{\varm+\cv-1}(-m^2{-} i\varepsilon)+R_N({-} i\varepsilon,\cv).
\end{eqnarray}
By Lemma~\ref{l:decfeynmanpowers}, for any $p\in \mathbb{N}$ s.t. $-p<c_2$ and $\varalpha>0$, we may always choose $N$ large enough so that the remainder term
$R_N({-}i\varepsilon,\cv)$ has $\pazocal{C}^\varalpha_{\loc}$ regularity  hence they have well-defined diagonal restriction which is holomorphic and bounded on $\Re\cv\geqslant -p$. We also proved that
the terms $\mathbf{F}_{\varm+\cv-1}(-m^2{-} i\varepsilon,x,x)$ have well-defined diagonal restriction which is holomorphic 
on the vertical line $\Re\cv=c_2$ since this line does not meet the poles of $ \mathbf{F}_{\varm+\cv-1}(-m^2{-} i\varepsilon,x,x)$.
Then it means that the integrand $e^{i\cv\frac{\pi}{2}}(P{+} i\varepsilon)^{-\cv}(x,x)\scalambda^{2\cv}\Gamma(\cv)$ in 
$$ \bea 
f\bigg(\frac{P{+} i\varepsilon}{\scalambda^2}\bigg)(x,x)=\frac{1}{2\pi i}\int_{\Re\cv=c} e^{i\cv\frac{\pi}{2}}(P{+} i\varepsilon)^{-\cv}(x,x)\scalambda^{2\cv}\Gamma(\cv) \pazocal{M}\widehat{f}(\cv)ds
\eea $$
has simple poles
at $\frac{n}{2},\frac{n}{2}-1, \frac{n}{2}-2$.

Then the result follows by moving the contour from $\Re\cv=c$ to $\Re\cv=c_2$
and using the Cauchy residue formula (we are allowed to do so because of the fast decay of ${\cM}\widehat{f}(\cv)$ when $\module{\Im \cv}\rightarrow +\infty$)
to get
$$ \bea 
f\bigg(\frac{P{+} i\varepsilon}{\scalambda^2}\bigg)(x,x)&= \sum_{k=0}^2 \big(\Res_{\cv=\frac{n}{2}-k}\Gamma(\cv)(P{+} i\varepsilon)^{-\cv}(x,x)\big)e^{i\frac{\pi}{2}
(\frac{n}{2}-k)} \scalambda^{n-2k}{\cM}\widehat{f}(\textstyle\frac{n}{2}-k)   \fantom +\underset{\pazocal{O}(\scalambda^{2c_2})}{\underbrace{\frac{1}{2\pi i}\int_{\Re\cv=c_2} e^{i\cv\frac{\pi}{2}}(P{+} i\varepsilon)^{-\cv}(x,x)\scalambda^{2\cv}\Gamma(\cv) \pazocal{M}\widehat{f}(\cv)ds}},
\eea $$
where the underbraced term  is $\pazocal{O}(\scalambda^{2c_2})$, which is of lower order smaller than the preceding ones.
\end{refproof}

%{ 
%\subsubsection{Remarks on residues and the spectral action in the bundle case}

%When $P$ acts on sections of some vector bundle $E\mapsto M$ and satisfies the condition of remarks~\ref{rem:vector1} and \ref{rem:vector2}, then the results of Theorems~\ref{thm:spectralaction} and ~\ref{mainmainthm} are identical except 
%the diagonal restrictions of the Schwartz kernels
%$\lim_{\varepsilon\rightarrow 0^+} \Res_{\cv=\frac{n}{2}-\varm}(P\pm i\varepsilon)^{-\cv}(x,x)$ and
%$f\bigg(\frac{P{ +} i\varepsilon}{\scalambda^2}\bigg)(x,x)$
%take value in $End(E_x)$.
%In order to extract numbers, it suffices to take
%the fiberwise trace $\lim_{\varepsilon\rightarrow 0^+} \Res_{\cv=\frac{n}{2}-\varm} Tr_E\left((P\pm i\varepsilon)^{-\cv}\right)(x,x)$ and
%$Tr_E\left(f\bigg(\frac{P{ +} i\varepsilon}{\scalambda^2}\bigg)\right)(x,x)$.
%}

\subsection{Extraction of the scalar curvature} 
Finally, we specialize the discussion of the formal parametrix construction in  \sec{s:hadamardformal} to the Laplace--Beltrami operator
$P=\square_g$ to explain how one can extract the scalar curvature from the residue of $(P{\pm}i0)^{-\cv}(x,x)$ at $\cv=\frac{n}{2}-1$ .

To that end we need to understand the geometric nature of the term $b^j\partial_{x^j}$
appearing in $P$, and also to interpret geometrically the first transport equation on $u_0,u_1$.   
Recall that the operator $P$ is defined in any coordinate system as~\cite[p.~270]{Hormander-97}:
\begin{equation}
Pu=\module{g}^{-\frac{1}{2}}\partial_{x^j}\left( \module{g}^{\frac{1}{2}}g^{jk}\partial_{x^k}u \right)
\end{equation}
for all $u\in C^\infty(M)$ and
where we sum over repeated indices~\footnote{Our convention follows
H\"ormander~\cite[p.~270]{Hormander-97}.}.
Therefore $P=\partial_{x^j}g^{jk}\partial_{x^k}+b^k(x)\partial_{x^k}$ where by
~\cite[p.~270]{Hormander-97}, 
$b^k(x)=\module{g(x)}^{-\12} g^{jk}(x)(\partial_{x^j}\module{g(x)}^{\frac{1}{2}} )$.
This leads us to the identity
\begin{equation}
\boxed{P=\partial_{x^k} g^{kj}(x)\partial_{x^j}+g^{jk}(x)(\partial_{x^j}\log \module{g(x)}^{\frac{1}{2}} ) \partial_{x^k} }
\end{equation}
which holds true in normal coordinates centered at an arbitrary point $\varvarm$.
These formul{\ae} are completely analogous to the well-known ones for the Laplace--Beltrami operator on Riemannian manifolds~\cite[p.~41--42]{soggeHangzhou}.

Recall that when we introduced the transport equations to study the parametrix, in Lemma~\ref{l:pretransport1} 
there was a function $h$ defined
in normal coordinates as $h(x)=b^j(x){\eta_{jk}}x^k$. It can be written as 
$$ \bea {
h(x)=b^j\eta_{jk}x^k=g^{lj}(x)\left(\partial_{x^l}\log\module{g}^\frac{1}{2}\right)\eta_{jk}x^k=x^k\partial_{x^k}\log\module{g}^\frac{1}{2}=\rho\log\module{g}^\frac{1}{2} }
\eea $$
where $\rho$ is the Euler vector field induced by the pseudo-Riemannian metric.
The first transport equation $2\rho u_0+hu_0=0 $ now reads~\cite[(2.4.18) p.~43]{soggeHangzhou}:
$$ \bea 
2\rho u_0=-\rho\log\module{g}^\frac{1}{2}u_0, {\,\  u_0(0)=1}
\eea $$
hence $u_0(x)=\vert g(0)\vert^\frac{1}{4} \vert g(x)\vert^{-\frac{1}{4}}$. 
The second transport equation is:
$$ \bea 
\rho u_1+u_1+\frac{h}{2}u_1=-P u_0.
\eea $$
Since both $\rho u_1$ and ${h=\rho\log(\module{g})^\frac{1}{2}}$ vanish at the origin, this implies that 
$$u_1(0)=-Pu_0{(0)}=- P(\module{g({0})}^\frac{1}{4} \module{g(x)}^{-\frac{1}{4}})|_{x=0} .$$ 
Now in normal coordinates $\vert g(0)\vert^\frac{1}{4}=1$ and from the 
Taylor expansion of the metric in normal coordinates
~\cite[(5.2) p.~82]{Baereinstein}~\cite[Prop.~1.28 p.~37]{BGV}:
$$g_{ij}(x)=\eta_{ij}+\frac{1}{3}R_{ikjl}x^kx^l +\pazocal{O}(\vert x\vert^3).$$ 
Therefore~\cite[p.~84]{Baereinstein} we get the Taylor expansion of $\module{g(x)}^{-\frac{1}{4}}$
in normal coordinates:
$$  
\bea
\module{g(x)}^{-\frac{1}{4}}&= \bigg\vert\module{\eta}\exp\left(\Tr\log\left(\delta_{ij}+\eta^{i_1}_i\frac{1}{3}R_{i_1kjl}(0)x^kx^l +\pazocal{O}(\vert x\vert^3)  \right)  \right) \bigg\vert^{-\frac{1}{4}} 
\\
&=\left(1+ \frac{1}{3}\Tr\left(\eta^{i_1}_iR_{i_1kjl}(0)x^kx^l   \right)  \right)^{-\frac{1}{4}}+\pazocal{O}(\vert x\vert^3)
=1+\frac{1}{12}\mathbf{Ric}_{kl}(0)x^kx^l+\pazocal{O}(\vert x\vert^3), 
\eea
$$
since $\Tr\left(\eta^{i_1}_iR_{i_1kjl}(0)x^kx^l   \right)=\delta^{ij}\eta^{i_1}_iR_{i_1kjl}(0)x^kx^l=\eta^{ij}R_{ikjl}(0)x^kx^l=-\mathbf{Ric}_{kl}(0)x^kx^l   $
where $\mathbf{Ric}_{kl}$ is the {Ricci} tensor. This implies that
$$-P \module{g(x)}^{-\frac{1}{4}}=-\frac{1}{6} \underset{=R_g(0)}{\underbrace{ g^{kl}\mathbf{Ric}_{kl}(0)}}+\pazocal{O}(\vert x\vert), $$ where we recognize
$R_g(0)=g^{kl}\mathbf{Ric}_{kl}$ to be the scalar curvature. Finally $u_1(x,x)=-\frac{R_g(x)}{6}$.

%{ 
%\subsubsection{Scalar curvature and $D^2$.}

%In the case $P=-D^2$ where $D=e_i.\nabla^S_{e_i}$ is the Dirac operator acting on spinors, the first transport equation 
%reads $ 2\nabla^S_\rho u_0+hu_0=0, u_0(0)=Id|_{S_0}  $ hence $u_0(x)=\vert g(0)\vert^\frac{1}{4} \vert g(x)\vert^{-\frac{1}{4}} \mathcal{P}_{x} $ where $\mathcal{P}_{x}\in Hom(S_x,S_0) $ is the \textbf{parallel transport} w.r.t.
%the spin connection $\nabla^S$ along the radial geodesic $t\in [0,1] \mapsto (1-t)x$ connecting $x$ and $0$ in normal coordinates. By the Lichnerowicz Theorem~\cite[Thm 3.52 p.~126]{BGV}~\cite[Thm 8.8]{lawsonspin}:
%\begin{eqnarray}
%P=-\nabla^{S\dagger}\nabla^S-\frac{R_g}{4}.
%\end{eqnarray}
%It follows that
%$$u_1(0)=-Pu_0(0)=\left(-\frac{R_g(0)}{6}+\frac{R_g(0)}{4} \right) Id|_{S_0}=\frac{R_g(0)}{12}Id|_{S_0}.$$
%}

We are done with extracting the scalar curvature from the coefficient $u_1(0)$ of the transport equations.  We conclude therefore the following result.

\begin{proposition}\label{mainmainprop} As a particular case  of Theorem \ref{mainmainthm}, 
if in addition the dimension  of $M$ is $n\geqslant 4$ then
$$ \bea 
\boxed{\lim_{\varepsilon\rightarrow 0^+} \Res_{\cv=\frac{n}{2}-1}(P\pm i\varepsilon)^{-\cv}(x,x)={\pm }
\frac{ {i}  R_g(x)}{6 {(4 \pi)^{\frac{n}{2} }}(\frac{n}{2}-2)!}}
\eea $$
where $R_g(x)$ is the scalar curvature at $x$.
%{
%If moreover $M$ is spin and $P=-D^2$ where $D$ is the Dirac operator on spinors $S\mapsto M$.
%Then 
%$$ \bea 
%\boxed{\lim_{\varepsilon\rightarrow 0^+} \Res_{\cv=\frac{n}{2}-1} Tr_E\left((P\pm i\varepsilon)^{-\cv}\right)(x,x)={\mp }
%\frac{ { i \mathbf{rk}\left(S\right) }  R_g(x)}{12 { (4 \pi)^{\frac{n}{2} }}(\frac{n}{2}-2)!} .}
%\eea $$
%where $\mathbf{rk}\left(S\right)$ is the rank of the Spin bundle.}
\end{proposition}

Put  together with Theorem \ref{mainmainthm} this proves our main result stated in the introduction, i.e.~Theorem \ref{mainthm}.

\appendix
 
 \section{Propagation estimates}\label{appLie}\init
 
 \subsection{Summary} The purpose of this appendix is to  supplement the material   in \secs{ss:kg}{ss:jj}  with a very brief summary on scattering calculus and propagation estimates.   
  
  Propagation estimates in the scattering setting are due to Melrose \cite{melrosered}. The generalization to variable weight orders presented here is due to Vasy \cite{vasygrenoble,vasyessential}, see \cite[\S\S2--3]{hassell} for a concise introduction, cf.~\cite[\S E.4]{DZ} .
 The scattering calculus in the model case $\overline{\rr^n}$ was earlier developed  among others by Shubin \cite{Shubin1978} and Parenti \cite{Parenti1972}. 
 
 \subsection{Scattering calculus}\label{ss:scattering} We use the notation already introduced in \secs{ss:kg}{ss:jj}; recall in particular that  $\rho$ is a boundary-defining function and $y$ are local coordinates on $\p\M$, extended in a collar neighborhood of $\p\M$. Let $(\rho, y, \varrho, \eta)$ be local coordinates on $\co$ such that  $(\varrho,\eta)$ are the dual coordinates of $(\rho,y)$. Recall that we introduced the  formal notation $\bra \xi \ket^{-1}$  for the boundary defining function of fiber infinity.
 
 The class of \emph{scattering symbols} of order $s,\ell\in \rr$, denoted by $S^{s,\ell}_{\rm sc}(T^*M)$, is defined  away from $\p\M$ in the same way as the usual symbol class $S^{s}(T^*M)$, whereas near the boundary, any $a\in S^{s,\ell}_{\rm sc}(T^*M)$ is a smooth section of $T^*M$ that satisfies the estimate
 \beq\label{sym1}
\forall  j,k\in\nn, \ \alpha,\beta\in\nn^{n-1}, \  \  \Big| (\rho \p_\rho)^j \p_y^\alpha \p_\varrho^k \p_\eta^\beta a(\rho,y,\varrho,\eta)    \big|\leqslant  C_{jk\alpha\beta}  \, \rho^{-\ell}   \bra \xi\ket^{s-k-|\beta|}.
 \eeq 
The model example  is  as always  $\overline{\rr^n}$ with standard coordinates $(x,\xi)$ on $T^*\rr^n$. In this case, by using spherical coordinates $x=(r,y)$ and setting $\rho=r^{-1}$, $\xi=(\varrho,\rho^{-1}\eta)$, $\bra x \ket=(1+|x|^2)^\12$ and $\bra \xi \ket=(1+|\xi|^2)^\12$, one finds that   \eqref{sym1} is equivalent to 
 $$ 
 \forall\alpha,\beta\in\nn^{n}, \ \   \big| \p_x^\alpha   \p_\xi^\beta a(x,\xi)    \big|\leqslant  C_{\alpha\beta}   \bra x\ket^{\ell-|\alpha|}   \bra \xi\ket^{s-|\beta|}. 
 $$
The class  of \emph{scattering pseudo-differential operators} $\Psi^{s,\ell}_{\rm sc}(M)$ is obtained from $\sc$-symbols $a\in S^{s,\ell}_{\rm sc}(T^*M)$ by reduction to quantization of symbols on $T^*\rr^n$. This requires to choose  a partition of unity $\{\psi_i\}_i$ subordinate to a finite chart covering of $M$ as well as suitable diffeomorphisms close to $\p\M$ (one can show that in the end different choices give the same operator modulo an element of $\Psi^{s-1,\ell-1}_{\rm sc }(M)$). One also includes in the definition of $\Psi^{s,\ell}_{\rm sc}(M)$ a class of regularizing operators in the sense that their Schwartz kernels $K_A(x,x')$ are smooth and decrease rapidly (with all derivatives) as the distance between $x\in M$ and $x'\in M$ tends to infinity. We refer the reader to, e.g., \cite{vasygrenoble} and \cite[\S 2]{Uhlmann2016} for an introduction, cf.~\cite{melrosered} for the original, more geometric description of the Schwartz kernels of scattering pseudo-differential operators.

 In the $\sc$-calculus, the \emph{principal symbol}  of $A\in\Psi^{s,\ell}_{\rm sc }(M)$ is the equivalence class of the symbol of $A$  in $S^{s,\ell}_{\rm sc}(T^*M)/ S^{s-1,\ell-1}_{\rm sc}(T^*M)$. It is often useful to consider the more narrow  class of \emph{classical pseudo-differential operators} which is obtained from \emph{classical symbols}, i.e.~from symbols of the form  $a=\bra \xi\ket^s \rho^{-\ell} \tilde a$ with $\tilde a\in\cf(\co)$. In the simplest case of $A\in \Psi^{0,0}_{\sc}(M)$ classical, it is possible to identify the principal symbol with the restriction of $a\in \cf(\co)$ to $\p \co$. For  classical  $A\in \Psi^{s,\ell}_{\sc}(M)$ of arbitrary order  there is also a natural identification of the principal symbol with a function on $\p\co$, see for instance \eqref{eqpri} for the explicit formula for the principal symbol $p_z$ of $\square_g-z$. %Let us also recall that we have denoted by  $\Sigma_z$ the  \emph{characteristic set} of $\square_g-z$, i.e.~the closure of $p_z^{-1}(\{0\})$.
 
 The \emph{microsupport} $\wf'_\sc(A)$ of $A\in \Psi^{s,\ell}_\sc(M)$ is the complement of the set of points $q\in \p\co$ such that the (full) symbol of $A$ coincides in a neighborhood of $A$ with a symbol in $S^{-N,-L}_\sc(T^*M)$ for all $N,L\in\rr$. If $A$ is classical, then its \emph{elliptic set} is the  complement $\elll_\sc(A)=\p\co\setminus \Sigma_\sc(A)$ of the \emph{characteristic set} $\Sigma_\sc(A)$, defined as the closure of the zero set of the principal symbol. 
 
In the context of propagation estimates it is useful to allow for weight orders $\ell$ that vary on $\co$. On top of the obvious modifications of the definitions of $S^{s,\ell}_{\rm sc}(T^*M)$ and $\Psi^{s,\ell}_{\rm sc }(M)$, the cost to pay is that one needs to slightly relax the decay stated in \eqref{sym1}, and require instead that
\beq\label{eqlklkk}
\Big| (\rho \p_\rho)^j \p_y^\alpha \p_\varrho^k \p_\eta^\beta a(\rho,y,\varrho,\eta)    \big|\leqslant  C_{jk\alpha\beta}  \, \rho^{-\ell-\delta(j+k+|\alpha|+|\beta|)}   \bra \xi\ket^{s-k-|\beta|}.
\eeq
for some $\delta>0$. This circumvents logarithmic losses one would otherwise have when differentiating $\ell$.  {Apart from that,} the change of definition \eqref{eqlklkk} has however no big practical significance and will be disregarded in the notation entirely.
 
Now, if $s\geqslant 0$ and  $\ell\in \cf(\co )$, we define the \emph{weighted Sobolev space of variable weight order}:
 $$
 H^{s,\ell}_\sc(M)= \{ u \in L^2(M) \ | \ A u \in L^2(M)\},
 $$
 where $A\in\Psi^{s,\ell}_\sc(M)$ is a classical elliptic operator (i.e., $\elll_\sc(A)=\p\co$) which can be chosen arbitrarily. One can fix in particular an invertible $A$, and  the norm can be then defined as  $\| u\|_{s,\ell}= \|A u \|$ (different choices of $A$ give equivalent norms). This agrees with the definition for $s\in\zz_{\geqslant 0}$ given in the main part of the paper. For $s\leqslant  0$, $H^{s,\ell}_\sc(M)$ can be defined as the dual of $H^{-s,-\ell}_\sc(M)$. Then, one has for all $s,\ell\in\rr$ and any elliptic $A\in\Psi^{s,\ell}_\sc(M)$
 $$
  H^{s,\ell}_\sc(M)= \big\{ u \in \textstyle\bigcup_{s',\ell'} H^{s',\ell'}_\sc(M)  \ \big|\big. \ A u \in L^2(M)\big\}.
 $$
 \subsection{Propagation estimates}\label{ss:pe} We now consider the setting of the wave or Klein-Gordon operator $P-z$ on non-trapping Lorentzian scattering spaces introduced in \secs{ss:kg}{ss:bi}. We review various microlocal estimates for $P-z$, following \cite{vasygrenoble,vasyessential}, with a particular emphasis on the dependence on the complex parameter $z$, which we assume to vary in some set $\complex\subset \cc$. 
 
 Recall that with the notation from \sec{ss:bi},  the characteristic set of $P-z$  is $\Char_\sc(P-z)=(\Sigma_0\cap\fibinf)\cup (\Sigma_z \cap\basinf)$.

We first state the analogue of H\"ormander's propagations of singularities theorem in our setting. The fixed $z$ version is due to Melrose \cite{melrosered};  see \cite[Thm.~5.4]{vasygrenoble} and the remarks in \cite{vasyessential} for the uniform version below including the $(\Im z)$ term. As in the main part of the text, we write $q\sim q'$ if $q$ and $q'$ are connected by a bicharacteristic in $\Char_\sc(P-z)$, and we denote {the closed bicharacteristic segment from $q$ to $q'$} by $\gamma_{q\sim q'}$. The notation $q\succ q'$ means that $q\sim q'$  and $q$ comes after $q'$ along the flow.
  
   \bep[Propagation of singularities]\label{pos} Let $s\in\rr$ and let $\ell\in\cf(\co)$ be non-decreasing along the Hamilton flow. Let $A_1,A_2,B\in\Psi^{0,0}_{\sc}(M)$ be such that $\wf'_\sc(A_1)\subset \elll_\sc(B)$ and for all $z\in\complex$ the following control condition is satisfied:
   \beq\label{eq:forward}
   \bea
   &\forall q\in \wf'_\sc(A_1)\cap \Char_\sc(P-z), \\ &\exists\, q'\in\elll_\sc(A_2) \mbox{ s.t. } q\succ q' \mbox{ and } \gamma_{q\sim q'} \subset \elll_\sc(B).  
  \eea
   \eeq
   Suppose $A_2 u \in\Hsc{s,\ell}$ and $\{B(P-z)u\}_{z\in\complex}$ is bounded in $H^{s-1,\ell+1}_\sc(M)$. Then for all $u\in H^{\SL}_\sc(M)$,
\beq\label{eq:thepos}
   \norm{A_1 u}_{s,\ell} + (\Im z )^{\12}    \norm{A_1 u}_{s-\12,\ell+\12}  \leqslant C (   \norm{A_2 u}_{s,\ell} +  \| B(P-z) u \|_{s-1,\ell+1}+ \err{u})
   \eeq
  uniformly in $z\in\complex\cap \{\Im z \geqslant 0\}$.
   \eep 
  
The control condition \eqref{eq:forward} means in particular that the knowledge about $u$ being in $H^{s,\ell}_\sc(M)$  microlocally is propagated \emph{forward} (from $\elll_\sc(A_2)$ to $\elll_\sc(A_1)$), consistently with the sign of $\Im z$.
  
% \begin{remark} Away from base infinity $\basinf$, $z$ does not enter the principal symbol of $P-z$ and the estimates can be propagated in both directions. In consequence, if $ \wf'_\sc(A_1)\cap\basinf=\emptyset$ then we can replace $q\succ q'$ in \eqref{eq:forward} by $q\sim q'$  and obtain an estimate uniform in the whole set $Z$.
 %\end{remark} 
  
Beside propagation of singularities one can also show a uniform version of the simpler \emph{elliptic estimate} \cite[Cor.~5.5]{vasygrenoble}.  
  
Next, recall that in our setting, $L_+$ are the sinks and  $L_-$ the sources. Below, $\ell\in \cf(\p\co)$ and $\ell_\pm=\ell|_{L_\pm}$ as in the main part of the text. 

 We now state the radial estimates for $P-z$. The \emph{low decay radial estimate} can be used to propagate decay properties of $u$ into $L_+$ from a punctured neighborhood $U\setminus U_1$. The \emph{higher decay radial estimate} serves to gain decay properties in a neighborhood of $L_-$ provided it is already better than the \emph{threshold value} $-\12$.
 
  We refer  again to \cite{melrosered} and  \cite[Prop.~5.27]{vasygrenoble} for the fixed $z$ version, and to \cite{vasyessential} for the modifications in the proof needed to accomodate for the $(\Im z)$ term. 
 
 \bep[Low decay radial estimate  {\cite[(5)]{vasyessential}}]\label{radial2}  Let $s\in\rr$ and assume that $\ell$ is non-decreasing along the Hamilton flow and $\ell_+<-\12$. Let $A,B,B_1\in\Psi^{0,0}_{\sc}(M)$ and let $U_1, U$ be open neighborhoods of $L_+$ in  $\p\co$ and assume $U_1\subset \elll_\sc(A)$, $\wf'_\sc(A)\subset \elll_\sc(B)\subset U$ and $\wf'_\sc(B_1)\subset U\setminus U_1$. Assume  that for all $z\in\complex$ the following control condition is satisfied:
    $$ 
    \bea
    &\forall q\in \wf'_\sc(A)\cap \Char_\sc(P-z)\setminus L_+, \\ &\exists\, q'\in\elll_\sc(B_1)  \mbox{ s.t. } q\succ q' \mbox{ and }  \gamma_{q\sim q'} \subset \elll_\sc(B).  
    \eea
    $$
    Suppose $B_1 u \in\Hsc{s,\ell}$ and $\{B(P-z)u\}_{z\in\complex}$ is bounded in $H^{s-1,\ell+1}_\sc(M)$. Then for all $u\in H^{\SL}_\sc(M)$, 
    $$
    \norm{A u}_{s,\ell} +  (\Im z )^\12    \norm{A u}_{s-\12,\ell+\12}   \leqslant C (   \norm{B_1 u}_{s,\ell} +  \| B(P-z) u \|_{s-1,\ell+1}+ \err{u})
    $$
   uniformly in $z\in\complex\cap \{\Im z \geqslant 0\}$.
 \eep

    \bep[Higher decay radial estimate {\cite[(4)]{vasyessential}}]\label{radial1}  Let $s\in\rr$ and assume that $\ell$ is non-decreasing along the Hamilton flow and  $\ell_->-\12$. Let $A,B\in\Psi^{0,0}_{\sc}(M)$ and let $U$ be a sufficiently small open neighborhood of $L_-$ in $\p\co$. Assume   $L_-\subset \elll_\sc(A)$ and $\wf'_\sc(A)\subset \elll_\sc(B)\subset U$. 
       Suppose $\{B(P-z)u\}_{z\in\complex}$ is bounded in $H^{s-1,\ell+1}_\sc(M)$. Then for all $s'\in \rr$, $\ell'\in\open{-\12,\ell}$ and $u\in H^{\SL}_\sc(M)$ such that $B u \in\Hsc{s',\ell'}$,
       \beq\label{erff}
       \norm{A u}_{s,\ell} +  (\Im z )^\12    \norm{A u}_{s-\12,\ell+\12}   \leqslant C (   \norm{B u}_{s',\ell'} +  \| B(P-z) u \|_{s-1,\ell+1}+ \err{u})
       \eeq
      uniformly in $z\in\complex\cap \{\Im z \geqslant 0\}$.
    \eep
    
    %Note that in practice, if $u\in \Hsc{s',\ell'}$ then when applying \eqref{erff} in global estimates it is possible to absorb the term $\norm{B u}_{s',\ell'}$ into the right hand side

   \section{Complex powers via functional calculus}\label{sec:global}
   
    \subsection{Contour integrals}\label{ss:contours} Suppose $P$ is a (possibly unbounded) self-adjoint operator acting in a Hilbert space $\cH$. 
    
    If $\cv\in\cc$ and $\varepsilon>0$, or if  $\Re\cv <0$ and $\varepsilon\geqslant 0$, then  the operator $(P-i\varepsilon)^{-\cv}$ is well-defined by the Borel functional calculus for self-adjoint operators.  In the particular case $\Re \cv > 0$ and $\varepsilon >0$, it satisfies
   $$
   (P-i\varepsilon)^{-\cv}=\frac{e^{-i\frac{\pi}{2}\cv}}{\Gamma(\cv)}\int_0^\infty s^{\cv-1} e^{-\varepsilon s} e^{ i P s}ds.
   $$
   in the sense of convergence of the integral in the strong operator  topology.

From our point of view it is more useful to express  $(P-i\varepsilon)^{-\cv}$ in terms of  the resolvent $(P-z)^{-1}$ as a contour integral, the precise form of which we briefly recall.  It is actually  more instructive to work with $A=i P$ instead of $P$. The reason is that $A$ and $A+\varepsilon$ are \emph{sectorial operators} (of angle $\frac{\pi}{2}$), and therefore their complex powers are special cases of a large and systematically studied functional calculus based on contour integrals, see \cite[\S 2]{sectorial} and references therein.  By \cite[Prop.~7.1.3]{sectorial}, the Borel functional calculus definition of $(A+\varepsilon)^{-\cv}$  coincides with the  sectorial calculus definition. This has the following immediate consequences.
   
   First, let $\varepsilon>0$. If $\Re \cv >0$, then consistency with the sectorial calculus implies the formula
    \beq\label{ctin0}
     ( A+ \varepsilon)^{-\cv}=  \frac{1}{2\pi i}\int_{\eta_\delta} z^{-\cv} \big(z-(A+\varepsilon)\big)^{-1} dz,
    \eeq
  where $0<\delta<\varepsilon$ and $\eta_\delta$ is an arbitrary  contour going from $\Im z\gg 0$ to  $\Im z\ll 0$ of the form
    $$
    \eta_\delta = e^{i\theta}\opencl{+\infty,\delta}\cup \{\delta e^{i\omega}\, | \, -\theta<\omega<\theta\}\cup  e^{-i\theta}\clopen{\delta,+\infty}
    $$  
    for some $\theta\in\open{\pid,\pi}$. More generally, for any $\cv\in \cc$
   \beq\label{ctin}
    ( A+ \varepsilon)^{-\cv}=  (1+A)^N \frac{1}{2\pi i}\int_{\eta_\delta}\frac{z^{-\cv}}{(1+z)^{N}} \big(z-(A+\varepsilon)\big)^{-1} dz,
   \eeq
  where $N\in\nn_{\geqslant 0}$ is an arbitrary number such that $N>-\Re \cv$. Furthermore, $\Dom A^N$ is a core for $( A+ \varepsilon)^{-\cv}$ \cite[Prop.~3.1.1]{sectorial}. Observe that if $\Re\cv<1$ then it is not necessary to surround $0$ in the integral, and so the contour $\eta_\delta$ in \eqref{ctin} can be replaced by
   $$
   \eta_0 =  e^{i\theta}\opencl{+\infty,0}\cup  e^{-i\theta}\clopen{0,+\infty}.
   $$
For $\Re\cv <0$, $( A+ \varepsilon)^{-\cv}$ is in general not bounded, it is however  a still a closed operator, with domain independent on $\varepsilon>0$.
   
   Let now $\varepsilon =0$. If $\Re\cv<0$ then again
    \beq\label{ctin2}
     A^{-\cv}=  (1+A)^N \frac{1}{2\pi i}\int_{\eta_0}\frac{z^{-\cv}}{(1+z)^{N}} (z-A)^{-1} dz,
     \eeq
   and the  domain is $\Dom A^{-\cv}=\Dom(A+ \varepsilon)^{-\cv}$  for $\varepsilon>0$ arbitrary, see \cite[Prop.~3.1.9]{sectorial}. Furthermore, 
     $$
     A^{-\cv} u = \lim_{\varepsilon\to 0^+} (A+ \varepsilon)^{-\cv} u, \ \   u\in \Dom A^{-\cv}.
    $$
In the special situation $0\notin \sp (A)$, $A^{-\cv}$ is well-defined for all $\cv\in \cc$, and
   $$
     A^{-\cv}=  (1+A)^N \frac{1}{2\pi i}\int_{\eta_\delta}\frac{z^{-\cv}}{(1+z)^{N}} (z-A)^{-1} dz
   $$
   for all  sufficiently small $\delta\geqslant 0$.
   
 Using back the relation $A=i P$ and changing the integration variable $z\to i (z+\varepsilon)$ one finds integrals with $(P-z)^{-1}$ over the contour $\gamma_\varepsilon$ used in the main part of the text (see  \ref{ss:restocp}).
   
  %\fi

\section{The ultrastatic case} \label{s:static}

\subsection{Resolvent bounds and Feynman wavefront sets for ultrastatic spacetimes} \label{app:static}

Let $(Y,h)$ be a complete Riemannian manifold of dimension $n-1$ and let
\beq\label{thelpp}
 M=\mathbb{R}\times Y,  \quad g=dt^2-h
 \eeq
 be the corresponding \emph{ultrastatic}
Lorentzian manifold of dimension $n$. 

The wave operator {is then}
$\square_g=\p_t^2 - \Delta_h$, where $\Delta_h$ is the  Laplace--Beltrami operator on $(Y,h)$ (with the convention  $-\Delta_h\geqslant 0$).  As explained in~\cite{derezinski}, the essential self-adjointness of  $\square_g$ in that case can be shown using Nelson's commutator theorem.

 For $s\in \mathbb{R}$ and $p>\frac{1}{2}$, we 
recall the definition of weighted Sobolev spaces $L^{2,p}_t H^s(Y)$ on $M$:
$$ \Vert u\Vert_{ L^{2,p}_tH^s(Y)}= \left(\int_{\mathbb{R}} \left\langle t\right\rangle^{2{ p}}\Vert u(t,.)\Vert_{H^s\left(Y \right)}^2  \right)^{\frac{1}{2}},$$
where $H^s(Y)= \bra -\Delta_h \ket^{-s} L^2\left(Y\right)$ is the usual Sobolev space on $Y$.

\begin{thm}\label{thm:us}
Let $(Y,h)$ be a complete Riemannian manifold, let $(M,g)$ be as in \eqref{thelpp} and let $P=\square_g$. Let $s\in \mathbb{R}$ and $p>\frac{1}{2}$. Then for $z\in \{\Im z\geqslant 0, \, \module{\Re z}\geqslant \varepsilon>0\}$, $(P-z)$ admits a Feynman inverse
$(P-z)^{-1}:L^{2,-p}_tH^s_y(\Cauchy)\to L^{2,p}_t H^{s+1}_y(\Cauchy) $ which satisfies a bound of the form 
\begin{equation} \label{e:boundfeynmanreolv}
\Vert (P-z)^{-1}u\Vert_{L^{2,{-p}}_tH^s_y}\leqslant C \Vert u\Vert_{L^{2,{p}}_tH^{s-1}_y}.
\end{equation}

Furthermore, if $\Re\cv>0$ and $\Im z>0$ then
$ (P-z)^{-\cv}:H^s(M) \to H^s(M) $ is well-defined for all $s\in \mathbb{R}$.
\end{thm}

 In particular, the above bound (\ref{e:boundfeynmanreolv}) holds true when $\Im z \rightarrow 0^+$, $\Re z\neq 0$ which yields a limiting absorption principle for the Klein--Gordon resolvent $(\square_g {+} m^2-i0)^{-1}$ as also proved in~\cite{derezinski}.

\begin{refproof}{Theorem \ref{thm:us}} The starting point is the well-known ansatz for $\Im z\geqslant 0$:
\beq
\big((P-z)^{-1} u\big)(t,.)=-\12 {\int_{\mathbb{R}}} \frac{e^{{-}i\vert t-s\vert\sqrt{-\Delta_h{-}z}}}{\sqrt{-\Delta_h{-}z}}u(s,.)ds.
\eeq
We first prove a rough bound for low values of $\Im z\geqslant 0$, $\module{\Re z}\geqslant \varepsilon$:
$$ \bea 
\Vert (P-z)^{-1}u\Vert^2_{L^{2,{ -p}}_tH^s_y}&=\frac{1}{4}\int_{\mathbb{R}}\left\langle t\right\rangle^{-2p}\Big\| \int_{\mathbb{R}} \frac{e^{{-}i\vert t-s\vert\sqrt{-\Delta_h{-}z}}}{\sqrt{-\Delta_h{-}z}}u(s,.)ds\Big\|^2_{H^s_y(\Cauchy)}dt \\
&\leqslant \frac{1}{4}\int_{\mathbb{R}}\left\langle t\right\rangle^{-2p}dt \, \sup_t\left(\int_{\mathbb{R}} \left\langle s\right\rangle^{-p}  \left\langle s\right\rangle^{p}\Big\| \frac{e^{{-}i\vert s-t\vert\sqrt{-\Delta_h{-}z}}}{\sqrt{-\Delta_h{-}z}}u(s,.)\Big\|_{H^s_y(\Cauchy)}ds\right)^2.
\eea $$
Then by Cauchy--Schwartz inequality in $s$:
$$ \bea 
&\int_{\mathbb{R}} \left\langle s\right\rangle^{-p}  \left\langle s\right\rangle^{p}\Big\| \frac{e^{{-}i\vert s-t\vert\sqrt{-\Delta_h{-}z}}}{\sqrt{-\Delta_h{-}z}}u(s,.)\Big\|_{H^s_y(\Cauchy)}ds \\ &\leqslant \left(\int_{\mathbb{R}} \left\langle s\right\rangle^{-2p}ds\right)^{\frac{1}{2}} \sup_t \left(\int_{\mathbb{R}} \left\langle s\right\rangle^{2p}\Big\| \frac{e^{{-}i\vert s-t\vert\sqrt{-\Delta_h{-}z}}}{\sqrt{-\Delta_h{-}z}}u(s,.)\Big\|^2_{H^s_y(\Cauchy)}ds\right)^{\frac{1}{2}} \\
&\leqslant  C\left(\int_{\mathbb{R}} \left\langle s\right\rangle^{2p}\Vert u(s,.)\Vert^2_{H^{s-1}_y(\mathbb{R}^{n-1})}ds\right)^{\frac{1}{2}}=C\Vert u\Vert_{L^{2,{p}}_tH^{s-1}_y}
\eea $$
using the fact that 
$$
\frac{e^{{-}i\vert t-s \vert\sqrt{-\Delta_h{-}z}}}{\sqrt{-\Delta_h{-}z}}: H^{s}_y(\Cauchy)\to H^{s+1}_y(\Cauchy) 
$$
 is {bounded for} all $s\in \mathbb{R}$ uniformly in $\Im z\geqslant 0$, $\norm{\Re z}\geqslant \varepsilon$.
Finally, for small $\Im z$, we get
$$ \bea 
\Vert (P-z)^{-1}u\Vert^2_{L^{2,{-p}}_tH^s_y(\mathbb{R}^{n-1})}\leqslant C^2 \Vert u\Vert^2_{L^{2,{p}}_tH^{s-1}_y}
\eea $$
which shows that $(P-z)^{-1}: L^{2,{p}}_t H^s_y(\Cauchy)\to L^{2,{-p}}_t H^{s+1}_y(\Cauchy) $ is invertible on the half-plane $\Im z\geqslant 0$, $\module{\Re z}\geqslant \varepsilon$. 

Next, we refine the above bounds for large $\vert z\vert$
along the contour $\gamma_\varepsilon$ defined in \sec{ss:restocp} to get decay in $z$. We denote by $E(\lambda)d\lambda$ the projection-valued measure associated to the functional
calculus of $-\Delta_h$, which is well-known to
be self-adjoint by completeness of $Y$ \cite{Chernoff1973,Strichartz1983}. For $u\in C^\infty_\c(M)$, we define $\widehat{u}=\int_{\mathbb{R}} e^{-i\tau t}E_\lambda\left(u(t,.)\right)dt
$. Then,
we get
$$ \Vert (P-z)^{-1}u\Vert_{H^s(M)}^2=\int_{\mathbb{R}\times \mathbb{R}_{\geqslant 0}}  \frac{(1+\vert\tau\vert^2+\lambda)^{s}\Vert \widehat{u}(\tau,\lambda)\Vert_{L^2(Y)}^2 }{ \vert {-}\tau^2{+}\lambda-z\vert^2}  d\tau d\lambda \lesssim \frac{1}{\module{\Im z}^2} \Vert u\Vert_{H^s(M)}^2.  $$
For $\varepsilon>0$, this implies by a contour integration argument  like in the proof of Lemma~\ref{l:holofamilyFm} that the complex powers $(P-i\varepsilon)^{-\cv}$, $\Re \cv > 0$ are well
defined and can be represented as
\begin{equation}\label{eq:repsfeynmanpowers}
(P-i\varepsilon)^{-\cv}u=\frac{1}{2\pi}\int_{\mathbb{R}\times \mathbb{R}_{\geqslant 0}} e^{i\tau t}  \left({-}\tau^2{+}\lambda-i\varepsilon \right)^{-\cv}\widehat{u}(\tau,\lambda) d\tau d\lambda.
\end{equation}
for all $ u\in H^s(M)$.

The bound on the wavefront set of $\left(P-z\right)^{-1}$ is an immediate consequence of the explicit formula for the Feynman inverse and follows the discussion in \sec{ss:firstparametrix}.
\end{refproof}

{
\subsection{Limiting absorption principle for Feynman powers}
\label{a:LAPfeynman}

In this second part of  Appendix~\ref{s:static}, for the sake of illustration  we specialize to ultrastatic Lorentzian manifolds $\mathbb{R}\times Y$ which are \emph{space-compact}, i.e.~such that $Y$ is compact Riemannian.
We give an elementary proof of the limiting absorption principle for Feynman powers, by which we mean that the limit
$\lim_{\varepsilon\rightarrow 0^+}(P-i\varepsilon
)^{-\cv}:C^\infty_\c(M) \to \pazocal{D}^\prime(M)$ exists for $P=\square_g+m^2$ with $m>0$.
This amounts to showing that one can take the $\varepsilon\rightarrow 0^+$ limit in \eqref{eq:repsfeynmanpowers}, and then to finding suitable  weighted Sobolev spaces on which the complex Feynman powers are well-defined.

\begin{defi}[Weighted anisotropic Sobolev spaces]\label{d:sobolevaniso}
We define the following product-type weighted Sobolev norms on $M$ depending on three indices, 
$H^{(s,\ell),p}(M)$ where $s$ is the time regularity, $\ell$ is a weight on the time variable and $p$ is the space regularity:
\begin{equation}
\Vert u\Vert_{H^{(s,\ell),p}(M)}=\left(\int_{M} \big| \bra D_t\ket^s \bra t \ket^\ell \bra  -\Delta_h\ket^{\frac{p}{2}}u\big|^2 d \vol_g \right)^{\frac{1}{2}}.
\end{equation}
\end{defi}

An important property of these spaces is that Fourier transform in the time variable exchanges the first two indices, i.e.~$u\in H^{(s,\ell),p}(M)$ implies $\cF_t(u)\in H^{(\ell,s),p}(M)$.

\begin{lemm}\label{l:estimates1d}
Let $m>0$, $\alpha\in \mathbb{C}$, $\Re\alpha>0$. Then
 $\left({-}\tau^2{+}\lambda{+}m^2-i0 \right)^{-\cv}_{\lambda\in \mathbb{R}_{\geqslant 0}}$ is a family of tempered distributions 
which satisfies 
\beq\label{eq:kll}
\big\Vert  \left({-}\tau^2{+}\lambda{+}m^2-i0 \right)^{-\cv}\big\Vert_{H^{\ell-\Re\alpha,-s_1}(\mathbb{R})}=\cO(\lambda^{-\frac{\Re \alpha}{2}})
\eeq
for all $\ell\in \clopen{0,\frac{1}{2}}$, $s_1>\frac{1}{2}$,
where $H^{\ell-\Re\alpha,-s_1}(\mathbb{R})$ denotes the weighted Sobolev space
$\left\langle t\right\rangle^{s_1}H^{\ell-\Re\alpha}(\mathbb{R})$. 
\end{lemm}
\begin{proof}
We cut the domain in three regions using a smooth partition of unitiy $1=\chi_1+\chi_2+\chi_3 $, where
where $\supp\chi_3 \subset \{\tau\geqslant \delta>0\}$, $\supp\chi_1\subset \{\tau\leqslant -\delta<0\}$ and $\supp \chi_2\subset \{ \vert\tau\vert\leqslant \frac{m}{2} \}$.

 Observe that $\chi_2\left( {-}\tau^2{+}\lambda{+}m^2-i0 \right)^{-\cv} $ is smooth, compactly supported and uniformly bounded in Schwartz functions when $\lambda\in \mathbb{R}_{\geqslant 0}$, so this term satisfies \eqref{eq:kll}.

Let us  examine the term  $\chi_{1}\left({-}\tau^2{+}\lambda{+}m^2-i0 \right)^{-\cv} $. On the support of $\chi_3$,  $\tau+\sqrt{\lambda+m^2}\geqslant \delta>0$  uniformly in $\lambda$, therefore in the factorization 
$$
\chi_{1}\left( {-}\tau^2{+}\lambda{+}m^2-i0 \right)^{-\cv} =\chi_{1}(\tau)\big( \tau+\sqrt{\lambda+m^2} \big)^{-\cv}\big( {-}\tau{+}\sqrt{\lambda+m^2}-i0 \big)^{-\cv},
 $$
 the term  $\chi_{1}(\tau)\left( \tau+\sqrt{\lambda+m^2} \right)^{-\cv}$  is $\cO_{C^\infty(\mathbb{R})}(\lambda^{-\frac{\Re\cv}{2}})$. By inverse Fourier transform in the variable $\tau$, we get
$$  \cF\Big( {-}\tau{+}\sqrt{\lambda+m^2}-i0 \Big)^{-\cv}=Ce^{it\sqrt{\lambda+m^2}}1_{\mathbb{R}_{\geqslant 0}}(t)t^{\alpha-1} $$
where $C$ is some constant. Thus,
$$
\bea
\Big\vert\Big\langle \chi_{1}\big( {-}\tau{+}\sqrt{\lambda+m^2}-i0 \big)^{-\cv},\psi  \Big\rangle\Big\vert&=\module{C} \bigg\vert\int_{\mathbb{R}}  e^{it\sqrt{\lambda+m^2}}1_{\mathbb{R}_{\geqslant 0}}(t)t^{\alpha-1} \widehat{\psi}(t)dt \bigg\vert \leqslant
 \module{C} \int_{0}^\infty \big\vert t^{\alpha-1} \widehat{\psi}(t) \big\vert dt\\
&\leqslant \module{C} \Big( \big\Vert \left\langle t\right\rangle^{\Re\alpha-\ell}\widehat{\psi} \big\Vert_{L^2(\rr)}+\Vert \widehat{\psi}\Vert_{H^{s_1}(\mathbb{R})}  \Big) \\ &\lesssim \Vert\widehat{\psi}\Vert_{H^{s_1,\Re\alpha-\ell}(\mathbb{R})}=\Vert\psi \Vert_{H^{\Re\alpha-\ell,s_1}(\mathbb{R})}
\eea
$$
if $\ell\in \clopen{0,\frac{1}{2}}$, $\ell\leqslant \Re \cv$ and $s_1>\frac{1}{2}$. By duality, $\chi_{1}\left( {-}\tau{+}\sqrt{\lambda+m^2}-i0 \right)^{-\cv} $ is bounded in the dual weighted Sobolev space $ H^{\ell-\Re\alpha,-s_1}(\mathbb{R})$. 

The term with $\chi_{3}$ is treated in the same way.
\end{proof}

Our objective is to study the regularity of the distribution
$ \left\langle u_2, \lim_{\varepsilon\rightarrow 0^+}  (P-i\varepsilon)^{-\cv}u_1\right\rangle $ where $P$ is the Klein--Gordon operator.
By compactness of $Y$, one has a discrete spectral resolution of the Laplacian $-\Delta_h:C^\infty(M)\to L^2(M)$, there is an orthonormal basis $(e_\lambda)_{\lambda\in \sp(-\Delta_h)}$ of $L^2(M)$, where $-\Delta_h e_\lambda=\lambda e_\lambda$.
By functional calculus, for any $f\in L^\infty(\mathbb{R}_{\geqslant 0})$, we 
have 
$$f(-\Delta_h)u= \sum_{\lambda\in \sp(-\Delta_h)} f(\lambda) \left\langle u, e_\lambda\right\rangle e_\lambda $$ where $f(-\Delta_h):L^2(Y) \to L^2(Y)$ acts as bounded operator.

 For all test functions $u\in C^\infty_\c(M)$, we define 
$\widehat{u}(\tau,\lambda)=\int_{\mathbb{R}}e^{-it\tau}\left\langle(u(t,.),e_\lambda\right\rangle  dt\in {\cS}(\mathbb{R})$.
By Fourier transform we are reduced to study the pairing
\begin{eqnarray*}
\Big\langle u_2, \lim_{\varepsilon\rightarrow 0^+}  (P-i\varepsilon)^{-\cv}u_1\Big\rangle=\sum_{\lambda\in \sp(-\Delta_h)} \int_{\mathbb{R}}  \frac{\overline{\widehat{u_2}}(\tau,\lambda)\widehat{u_1}(\tau,\lambda)}{({-}\tau^2{+}\lambda{+}m^2-i0)^\cv}d\tau 
\end{eqnarray*}  
for all test functions  $u_1$ and $u_2$.
The functions $\widehat{u_1}\widehat{u_2}$ are Schwartz in $\tau$, more precisely
\begin{eqnarray*}
\sum_{\lambda\in \sp(-\Delta_h)}\int_{\mathbb{R}} (1+\lambda+\tau^2)^N \Vert \left\langle \partial_\tau\right\rangle^{N} \widehat{u_i}(\tau,\lambda)\Vert^2_{L^2(Y)}d\tau <\infty
\end{eqnarray*}
for all $N$, hence the product $\widehat{u_2}(\tau,\lambda)\widehat{u_1}(\tau,\lambda)$ is Schwartz in $\tau$ with fast decay in $\lambda$ which implies the distributional pairings
$\int_{\mathbb{R}}  \frac{\overline{\widehat{u_2}}(\tau,\lambda)\widehat{u_1}(\tau,\lambda)}{({-}\tau^2{+}\lambda{+}m^2-i0)^\cv}d\tau $ is well-defined for all $\lambda\in \sigma(-\Delta_h)$.

Now using the crucial Lemma~\ref{l:estimates1d}, we get that
$$
\bea
\big\vert\big\langle u_2, \lim_{\varepsilon\rightarrow 0^+}  (P-i\varepsilon)^{-\cv}u_1\big\rangle\big\vert&\leqslant  \sum_{\lambda\in \sp(-\Delta_h)} \bigg\vert\bigg(\int_{\mathbb{R}}  \frac{\overline{\widehat{u_2}}(\tau,\lambda)\widehat{u_1}(\tau,\lambda)}{({-}\tau^2{+}\lambda{+}m^2-i0)^\cv} d\tau\bigg) \bigg\vert \\
&\leqslant C \sum_{\lambda\in \sp(-\Delta_h)} \left\langle\lambda \right\rangle^{-\frac{\Re \alpha}{2}} \Vert \overline{\widehat{u_2}}(.,\lambda)\widehat{u_1}(.,\lambda)   \Vert_{H^{\Re\cv,s_1}(\mathbb{R})} 
\eea
$$
for some $s_1>\frac{1}{2}$. Now if $s_2>\frac{1}{2}$, $s_2\geqslant \Re\alpha$,  the Moser estimates yield
$$ \Vert uv \Vert_{H^{s_2}( \mathbb{R} )}\lesssim \Vert v \Vert_{H^{s_2}( \mathbb{R} )}\Vert v \Vert_{H^{s_2}( \mathbb{R} )} ,$$ or in the weighted 
version,
$$ \Vert uv \Vert_{H^{s_2,s_1}( \mathbb{R} )}\lesssim \Vert v \Vert_{H^{s_2,\ell_1}( \mathbb{R} )}\Vert v \Vert_{H^{s_2,\ell_2}( \mathbb{R} )}   $$
where $\ell_1+\ell_2=s_1>\frac{1}{2}$.
This implies that for each $\lambda\in \sp(-\Delta_h)$,
$$
\Vert \widehat{u_2}(.,\lambda)\widehat{u_1}(.,\lambda) \Vert_{H^{\Re\alpha,s_1}(\mathbb{R})}
\leqslant 
C \Vert \widehat{u_1}(.,\lambda) \Vert_{H^{s_2,\ell_1}_\tau(\mathbb{R})}  \Vert \widehat{u_2}(.,\lambda)\Vert_{H^{s_2,\ell_2}_\tau(\mathbb{R})}.  
$$
Therefore, using Cauchy--Schwarz inequality we obtain
$$
\bea
\big\vert\big\langle u_2, \lim_{\varepsilon\rightarrow 0^+}  (P-i\varepsilon)^{-\cv}u_1\big\rangle\big\vert
&\leqslant \sum_{\lambda\in \sp(-\Delta_h)} C
\left\langle\lambda \right\rangle^{-\frac{\Re\alpha}{2}}
\Vert \widehat{u_1}(.,\lambda) \Vert_{H^{s_2,\ell_1}_\tau(\mathbb{R})}  \Vert \widehat{u_2}(.,\lambda)\Vert_{H^{s_2,\ell_2}_\tau(\mathbb{R})}\\
&\leqslant 
C
\bigg(\sum_{\lambda\in \sp(-\Delta_h)}  \left\langle\lambda \right\rangle^{-2p_1} 
\Vert \widehat{u_1}(.,\lambda) \Vert_{H^{s_2,\ell_1}_\tau(\mathbb{R})}^2 
\bigg)^{\frac{1}{2}} \fantom \times
\bigg(\sum_{\lambda\in \sp(-\Delta_h)}  \left\langle\lambda \right\rangle^{-2p_2} 
\Vert \widehat{u_2}(.,\lambda) \Vert_{H^{s_2,\ell_2}_\tau(\mathbb{R})}^2   \bigg)^{\frac{1}{2}}
\eea
$$
where $p_1+p_2=\frac{\Re\alpha}{2}$.

To estimate the r.h.s.~we need the following simple result.

\begin{lemm}
For all $u\in C^\infty_\c(M)$,
\begin{equation}
\sum_{\lambda\in \sp(-\Delta_h)}  \left\langle \lambda\right\rangle^{p} \Vert u_\lambda\Vert_{H^{s,\ell}_t(\mathbb{R})}^2= \Vert u\Vert^2_{H^{(s,\ell),p}(M)}
\end{equation}
where $\Vert .\Vert_{H^{(s,\ell),p}(M)}$ is the product-type weighted norm from Definition~\ref{d:sobolevaniso}.
\end{lemm}
\begin{proof}
By definition of  $\Vert .\Vert_{H^{(s,\ell),p}(M)}$ and Fubini's theorem, 
$$
\bea
\sum_{\lambda\in \sp(-\Delta_h)}  \left\langle \lambda\right\rangle^{p} \Vert u_\lambda\Vert_{H^{s,\ell}_t(\mathbb{R})}^2  &=\sum_{\lambda\in \sp(-\Delta_h)}  \left\langle \lambda\right\rangle^{p}
\int_{\mathbb{R}} \vert \left\langle D_t\right\rangle^s  \left\langle t\right\rangle^\ell \left\langle u(t),e_\lambda\right\rangle \vert^2 dt  \\
&=\int_{\mathbb{R}} \sum_{\lambda\in \sp(-\Delta_h)}  \left\langle \lambda\right\rangle^{p}\big\vert\big\langle\big(\left\langle  D_t\right\rangle^s  \langle t\rangle^\ell u\big)(t),e_\lambda\big\rangle \big\vert^2   dt\\
 &= \int_{\mathbb{R}} \Vert(\one-\Delta_h)^{\frac{p}{2}}
 \left\langle  D_t\right\rangle^s  \left\langle t\right\rangle^\ell u\Vert^2_{L^2(Y)}dt=\Vert u\Vert^2_{H^{(s,\ell),p}(M)}
\eea
$$
where we used functional calculus and the fact that the spectral projection commutes with operators depending only on the $t$ variable.
\end{proof}

Therefore, we find that
$$
\Big|\Big\langle u_2, \lim_{\varepsilon\rightarrow 0^+}  (P-i\varepsilon)^{-\cv}u_1\Big\rangle\Big|
\leqslant 
C\Vert  u_1 \Vert_{H^{(\ell_1,s_2),p_1}(M)}
\Vert  u_2 \Vert_{H^{(\ell_2,s_2),p_2}(M)}
$$
for all $s_2>\frac{1}{2}$, $s_2\geqslant \Re \cv$, $\ell_1+\ell_2>\frac{1}{2}$ and $p_1+p_2=\frac{\Re \cv}{2}$ which concludes the proof of the limiting absorption principle stated below. 

\begin{thm}
Let $M=  \mathbb{R}\times Y$ be an ultrastatic Lorentzian manifold such that $Y$ is compact, and let  $P$ be the Klein--Gordon operator with $m>0$.
The complex Feynman powers acts as a continuous map between weighted Sobolev spaces, namely, the weak operator limit
\begin{eqnarray*}
 (P-i0)^{-\cv}: H^{(\ell_1,s_2),p}(M)\to H^{(-\ell_2,-s_2),p-\frac{\Re\cv}{2}}(M)
\end{eqnarray*} 
is well-defined and continuous for all $p\in \mathbb{R}$, $s_2>\frac{1}{2}$, $s_2\geqslant \Re\cv$ and $\ell_1+\ell_2>\frac{1}{2}$.
\end{thm}
}

\section{Various auxiliary proofs}

%{
\subsection{A Wick rotation lemma} We state below a lemma used several times   in the main part  of the text. As in \sec{sec:elementary}, $Q$ is the quadratic form $Q(\xi)= -\xi_0^2+\sum_{i=1}^{n-1}\xi_i^2$ on $\rr^n$.

\begin{lemm}\label{l:wick} 
Let $\cv\in\cc$. When $\theta\rightarrow -\frac{\pi}{2}$, the distribution $ \left(e^{i2\theta}\xi_1^2+\xi_2^2+\dots+\xi_n^2\right)^{-\cv}\rightarrow (Q(\xi)-i0)^{-\cv} $ in $\cD^\prime_\Gamma(\mathbb{R}^n\setminus\{0\})$, $\Gamma=\{(\xi;\tau dQ(\xi)) \,|\, Q(\xi)=0, \tau<0\}$.
\end{lemm}
\begin{proof}
The proof follows closely the proof of~\cite[Thm 3.1.15]{H} for the convergence in $\pazocal{D}^\prime$. For the control of the wavefront set, the proof follows closely~\cite[Thm 8.4.8]{H} which gives a wavefront bound in the sense of quasi--analytic classes.
\end{proof}

\subsection{Wavefront set of the pull-back \texorpdfstring{$G^*\Fse{z}$}{GF(z)}}
\label{ss:pullback}
We compute the 
wavefront set
of the germ of distribution $G^*\Fse{z}$ as stated in Lemma \ref{Wavefrontpullback}. Let us recall that $\Fse{z}$ is the elementary family of distributions on $\rr^n$ introduced in \sec{ss:elementary}, and that the pull-back  by the submersion $G$ defined in \sec{sss:pullback} gives a distribution defined on a neighborhood $\cU$ of the diagonal $\Delta$ in  $M\times M$.

\begin{refproof}{Lemma \ref{Wavefrontpullback}}   \step{1} An application of the 
pull-back  theorem \cite[Thm.~8.2.4]{H} in 
our situation
gives
\begin{equation}\label{pulledbackWF}
\WF(G^*F_\cv)\subset \{(x_1,x_2; k\circ d_{x_1}G,k\circ d_{x_2}G) \st (G(x_1,x_2),k)\in \WF(F_\cv)\} 
\end{equation}
We denote by $(x_1,x_2;\eta_1,\eta_2)$ an element
of $T^* \pazocal{V}\subset T^* M^2$ and
$(h^\mu;k_\mu) $ the coordinates in $T^* \mathbb{R}^{n}$.
The pull-back with 
indices reads:
$$(x_1,x_2; k\circ d_{x_1}G,k\circ d_{x_2}G)=(x_1,x_2; k_\mu d_{x_1}G^\mu,k_\mu d_{x_2}G^\mu).$$

\step{2} We first 
compute
$\WF(G^* F_\cv)$
outside
the set
$\diag=\{x_1=x_2\}$.
The condition 
$(G(x_1,x_2),k)\in \WF(F_\cv)$ in (\ref{pulledbackWF}), 
reads 
$(G^\mu(x_1,x_2);k_\mu)=(G^\mu(x_1,x_2);\lambda\eta_{\mu\nu}G^\nu(x_1,x_2))$ where $\lambda <0$.
We obtain $$(x_1,x_2; \lambda k\circ d_{x_1}G,\lambda k\circ d_{x_2}G)=(x_1,x_2;\lambda G^\mu\eta_{\mu\nu_2}d_{x_1}G^{\nu_2},\lambda G^\mu\eta_{\mu\nu_2} d_{x_2}G^{\nu_2})$$
and also 
$G^\mu(x_1,x_2)\eta_{\mu\nu}G^\nu(x_1,x_2)=0$.
Now set
$\Gamma(x_1,x_2)=G^\mu(x_1,x_2)\eta_{\mu\nu}G^\nu(x_1,x_2)$.
The key
observation is that
$d_{x_1}\Gamma=2G^\mu\eta_{\mu\nu}d_{x_1}G^\nu$ and $d_{x_2}\Gamma=
2G^\mu\eta_{\mu\nu}d_{x_2}G^\nu$, 
hence:
$$\WF(G^{* }F_\cv)\subset \{(x_1,x_2; \lambda d_{x_1}\Gamma,\lambda d_{x_2}\Gamma)\,|\, \Gamma(x_1,x_2)=0, \, \lambda\in\mathbb{R}_{<0}  \}.$$

We first interpret the term 
$$\{(x_1,x_2; \lambda d_{x_1}\Gamma,\lambda d_{x_2}\Gamma) \,| \,\Gamma(x_1,x_2)=0, \, \lambda <0   \}$$ 
appearing in the last formula
as the subset
of all elements
in $T^* \pazocal{U}$
of the conormal bundle
of the conoid $\{\Gamma=0\}$ 
such 
that 
elements of positive energy are propagated in the future and elements of negative energy are propagated in the past:
this is exactly the 
\emphasize{Feynman condition}.
In fact, if we use the metric 
to lift the indices,
$d_{x_1}\Gamma\left(e_\mu(x_1)\right)\eta^{\mu\nu}e_\nu(x_1)$
and $d_{x_2}\Gamma\left(e_\mu(x_2)\right)\eta^{\mu\nu}e_\nu(x_2)$
are the Euler vector fields $\nabla_1\Gamma,\nabla_2\Gamma$ defined by Hadamard. 
The vectors 
$\nabla_1\Gamma,-\nabla_2\Gamma$ 
are parallel along the null geodesic
connecting $x_1$ and $x_2$ (which is easily checked using normal coordinates centered at $x_1$, proving
$(d_{x_1}\Gamma,-d_{x_2}\Gamma)$ are in fact \emphasize{coparallel}
along this null geodesic. Denoting $q_1=(x_1: \lambda d_{x_1}\Gamma)$ and $q_2=(x_2; \lambda d_{x_2}\Gamma)$, the relation $\exp_{x_1}(\nabla_1\Gamma)=x_2$ implies that $\nabla_1\Gamma$ points to the future (resp.~past) if and only if $q_1 \succ q_2$ (resp.~$q_2\succ q_1$), which implies the Feynman condition.

\step{3 (``diagonal'')} For any function
$G$ on $M^2$, we
uniquely decompose the total differential
in two
parts
as follows
$$dG=d_{x_1}G+d_{x_2}G,\text{ where }d_{x_1}G|_{\{0\}\times T_{x_2}M}=0,d_{x_2}G|_{T_{x_1}M\times\{0\}}=0.$$
Let $i$ be the diagonal inclusion
map $i \defeq M \ni x\mapsto (x,x)\in \diag\subset M$.
The  
$\forall x\in M$, $G\circ i(x)=0$ implies 
$d_x G\circ i=0$, which is equivalent to  $d_{x_1}G\circ di+d_{x_2}G\circ di=0$.
Since
$$d_{x_2}G^\mu(x,x)=d_{x_2}\varalpha^\mu_{x_1}\left(\exp_{x_1}^{-1}(x_2)\right)|_{x_1=x_2=x}=
\varalpha^\mu_{x_1}\left(d_{x_2}\exp_{x_1}^{-1}(x_2)\right)|_{x_1=x_2=x}=\varalpha^{\mu}(x),$$ because $d_{x_2}\exp^{-1}_{x_1}(x_2)|_{x_1=x_2=x}=\id_{T_xM\to T_xM}=e_{\mu}(x)\varalpha^{\mu}(x)$.
Thus
$d_{x_1}G^\mu(x,x)=-\varalpha^\mu(x)$ and
$$ \{(x_1,x_2; k\circ d_{x_1}G,k\circ d_{x_2}G) \st x_1=x_2, \}$$
$$=\{(x,x; -k_\mu \varalpha^\mu(x),k_\mu \varalpha^\mu(x)) \st x\in M\}.$$ 

This concludes the proof of Lemma \ref{Wavefrontpullback}.\end{refproof}

\subsection{H\"older, scaling and Fourier decay} \label{sss:Holderfourierdecay} We  turn now our attention to the proof of regularity estimates for $G^*\Fse{z}$ which are uniform in $z$.

We first recall a position space definition of H\"older functions $\pazocal{C}^\varalpha(\rr^n)$
which coincides with the Fourier definition for non integer $\varalpha>0$. 
The equivalence is proved in~\cite[Prop.~8.1]{taylorpartial3},  \cite[Prop.~8.6.1]{Hormander-97}. 
Let us recall a version adapted to our discussion.
\begin{lemm}\label{l:holdersobolev}
Let $\varalpha\in \mathbb{R}$. Then
$u\in \pazocal{C}^\varalpha_{\loc}(\mathbb{R}^n)$ iff for every test function $\chi\in C^\infty_\c(\mathbb{R}^n)$
$$ \vert \widehat{u\chi}(\xi) \vert\leqslant C(1+\vert \xi\vert)^{-\varalpha-n} .$$
{As a consequence, we have the continuous injection $\pazocal{C}^{\varalpha}_{\loc}(\mathbb{R}^n) \hookrightarrow H_{\loc}^{\varalpha+\frac{n}{2}-\varepsilon} (\mathbb{R}^n)$ for all $\varepsilon>0$.}
\end{lemm}

\begin{proof}
If $u\in \pazocal{C}^\varalpha(\mathbb{R}^n)$ with $\varalpha>0, k< \varalpha <k+1$ then it means for any $x$, there exists a polynomial $P$ of degree $k$, which is nothing but the Taylor polynomial of $u$ at $x$, s.t.~{for all test function $\varphi\in C^\infty_{\rm c}(\mathbb{R}^n)$ (one could also take $\varphi$ in the Schwartz class):} 
$$\bigg| \int_{\mathbb{R}^n}(u-P)(\lambda(y-x)+x)\varphi(y)d^ny\bigg|\leqslant C\lambda^\varalpha\Vert \varphi\Vert_{L^\infty} .$$
Now let $u\in \pazocal{C}^\varalpha_{\loc}(\mathbb{R}^n)$, hence we may multiply
$u$ with some cut-off $\chi\in C^\infty_\c(\mathbb{R}^n)$ so that
$u\chi\in \pazocal{C}^\varalpha$. 
In particular choosing the function on the r.h.s.~as $ e^{ix.\xi}$ yields 
$$ \sup_{0<\lambda\leqslant 1} \lambda^{-\varalpha} \sup_{1\leqslant\vert \xi\vert \leqslant 2} \module{\left\langle   (u\chi-P)(\lambda.), e^{i\xi.x}\right\rangle}\leqslant  \Vert u\chi\Vert_{\pazocal{C}^\varalpha} \sup_{1\leqslant\vert \xi\vert \leqslant 2} \Vert e^{i\left\langle\xi,.\right\rangle } \Vert_{L^\infty}= \Vert u\chi\Vert_{\pazocal{C}^\varalpha}. $$
Therefore, {using the fact} that $\left\langle   (u\chi-P)(\lambda.),e^{i\xi.x}\right\rangle=\left\langle   u\chi(\lambda.), e^{i\xi.x}\right\rangle  $ since the Fourier transform restricted to $\vert\xi\vert\geqslant 1$ does not see the polynomial, and 
$\left\langle   u\chi(\lambda.), e^{i\xi.x}\right\rangle=\lambda^{-n}\widehat{u\chi}(\frac{\xi}{\lambda}) $,
we get
$$ \sup_{0<\lambda\leqslant 1} \lambda^{-\varalpha-n} \sup_{1\leqslant\vert \xi\vert \leqslant 2} \vert\widehat{u\chi}(\xi/\lambda)\vert\leqslant \Vert u\chi\Vert_{\pazocal{C}^\varalpha} .$$
Hence for $\vert \xi\vert\geqslant 1$, we get
$$\vert\widehat{u\chi}(\xi)\vert= \vert\widehat{u\chi}(\xi\vert\xi\vert/\vert \xi\vert)\vert\leqslant \Vert u\chi\Vert_{\pazocal{C}^\varalpha}\vert \xi\vert^{-\varalpha-n} $$
and finally this means 
that: 
$$ \vert \widehat{u\chi}(\xi) \vert\leqslant C(1+\vert \xi\vert)^{-\varalpha-n} .$$

Conversely, if we have the Fourier decay 
$ \vert\widehat{u\chi}(\xi)\vert\leqslant C(1+\vert\xi \vert)^{-r} $ for $r\in \mathbb{R}_{\geqslant 0}$, then
the Littlewood--Paley blocks are bounded by: 
$$ \bea 
\Vert \psi(2^{-j}\sqrt{-\Delta}) (u\chi) \Vert_{L^\infty} &=\Vert \pazocal{F}^{-1}\left( \psi(2^{-j}\vert \xi\vert) \widehat{u\chi}(\xi)\right) \Vert_{L^\infty}
\leqslant \int_{\mathbb{R}^n} \vert  \psi(2^{-j}\vert \xi\vert) \widehat{u\chi}(\xi) \vert d^n\xi
\\
&\leqslant 
2^{jn}\int_{\mathbb{R}^n} \vert  \psi(\vert \xi\vert) \widehat{u\chi}(2^{-j}\xi) \vert d^n\xi\leqslant 
C2^{jn} \int_{\mathbb{R}^n} \psi(\vert \xi\vert)(1+2^{j}\vert \xi\vert)^{-r}d^n\xi\\
&\leqslant C2^{j(n-r)}  \int_{\mathbb{R}^n} \psi(\vert \xi\vert)(2^{-j}+\vert \xi\vert)^{-r}d^n\xi\lesssim 2^{j(n-r)}.
\eea $$ 
This means that $u\in \pazocal{C}^{ r-n}_{\loc}(\rr^n)$.
\end{proof}

Set $\cv\in \mathbb{C}$ with $\Re\cv\geqslant 0$. We  consider the H\"older regularity under pull-back of $G^*F_\cv\in  \pazocal{D}^\prime(\pazocal{U})$ where $\pazocal{U}\subset M\times M $ is the neighborhood of the diagonal 
and $G:\pazocal{U} \ni(x,y) \mapsto G(x,y)\in  \mathbb{R}^n$ is the $C^\infty$ submersive map defined by \eqref{Wavefrontpullback}.

\begin{lemm}\label{l:holderpullback}
Let $\varm=\plancher{\Re\cv}+1$ and $\Fse{z}\in \pazocal{D}^\prime(\mathbb{R}^n)$ as defined in equation~(\ref{e:defF}).
Let $G$ be the $C^\infty$ submersive map defined in \eqref{Wavefrontpullback}. Then the pull-back
$\mathbf{F}_{\cv}(z,.)=G^*F_\cv(z,.)$ is in 
$\pazocal{C}_{\loc}^{(2-2a)(\Re\cv+1)-\varm-n}(\pazocal{U})$ with  decay in $z$ of order $\pazocal{O}(\module{\Im z}^{-a(\Re\cv+1)})$ for $a\in [0,1]$. 
\end{lemm}
\begin{proof}
Let $U$ be some small geodesically convex open subset in $M$.
We choose some test function $\chi\in C^\infty_\c(U)$ in such a way that, in the support of $\chi\otimes \chi$, we have
a local diffeomorphism $E:U\times \mathbb{R}^{n} \ni (x,h)\mapsto (x,\exp_x(h))\in U\times U$.
Then by definition of $G$ and of the exponential map, we have the  
identity 
$$ E^*\left(\chi\otimes \chi G^*\Fse{z}\right)(x;h)=F_\cv(z,\vert h\vert_g) \chi(x)\chi(\exp_x(h))\in \pazocal{D}^\prime(U\times \mathbb{R}^n).$$
Now observe that $F_\cv(z,\vert h\vert_g)=\pazocal{O}_{\pazocal{C}^\varalpha}((1+\module{\Im z})^{-a(\Re\cv+1)})$ for $a\in [0,1]$, 
$\varalpha\leqslant (2-2a)(\Re\cv+1)-\varm-n $ and that 
$\chi(x)\chi(\exp_x(h))\in C^\infty_\c(U\times \mathbb{R}^n)$ hence the result follows.
\end{proof}

\bibliographystyle{abbrv}
\bibliography{complexpowers}

\begin{thebibliography}{100}

\bibitem{Baereinstein}
B.~Ammann and C.~B{\"{a}}r.
\newblock {The Einstein-Hilbert action as a spectral action}.
\newblock In {\em Noncommutative Geom. Stand. Model Elem. Part. Phys.}, volume
  596, pages 75--108. 2002.

\bibitem{ALNV}
B.~Ammann, R.~Lauter, V.~Nistor, and A.~Vasy.
\newblock {Complex powers and non-compact manifolds}.
\newblock {\em Commun. Partial Differ. Equations}, 29(5-6):671--705, 2004.

\bibitem{Antoniano1985}
J.~L. Antoniano and G.~A. Uhlmann.
\newblock {A functional calculus for a class of pseudodifferential operators
  with singular symbols}.
\newblock In {\em Pseudodifferential Oper. Fourier Integr. Oper. with Appl. to
  Partial Differ. Equations, Notre Dame, Indiana, April 2-5, 1984, Proc. Symp.
  Pure Math.}, pages 5--16. 1985.

\bibitem{Atiyah1975}
M.~Atiyah, R.~Bott, and V.~K. Patodi.
\newblock {On the heat equation and the index theorem}.
\newblock {\em Invent. Math.}, 28(3):277--280, 1975.

\bibitem{Bar2019}
C.~B{\"{a}}r and A.~Strohmaier.
\newblock {An index theorem for Lorentzian manifolds with compact spacelike
  Cauchy boundary}.
\newblock {\em Am. J. Math.}, 141(5):1421--1455, 2019.

\bibitem{Baer2020}
C.~B{\"{a}}r and A.~Strohmaier.
\newblock {Local index theory for Lorentzian manifolds}.
\newblock {\em arXiv:2012.01364}, 2020.

\bibitem{BVW}
D.~Baskin, A.~Vasy, and J.~Wunsch.
\newblock {Asymptotics of radiation fields in asymptotically Minkowski space}.
\newblock {\em Am. J. Math.}, 137(5):1293--1364, 2015.

\bibitem{Battisti2011}
U.~Battisti and S.~Coriasco.
\newblock {Wodzicki residue for operators on manifolds with cylindrical ends}.
\newblock {\em Ann. Glob. Anal. Geom.}, 40(2):223--249, 2011.

\bibitem{BGV}
N.~Berline, E.~Getzler, and M.~Vergne.
\newblock {\em {Heat Kernels and Dirac Operators}}.
\newblock 2004.

\bibitem{bernal1}
A.~N. Bernal and M.~S{\'{a}}nchez.
\newblock {On smooth Cauchy hypersurfaces and Geroch's splitting theorem}.
\newblock {\em Commun. Math. Phys.}, 243(3):461--470, 2003.

\bibitem{bernal2}
A.~N. Bernal and M.~S{\'{a}}nchez.
\newblock {Smoothness of time functions and the metric splitting of globally
  hyperbolic spacetimes}.
\newblock {\em Commun. Math. Phys.}, 257(1):43--50, 2005.

\bibitem{bernstein1971}
I.~N. Bernshtein.
\newblock {Modules over a ring of differential operators. Study of the
  fundamental solutions of equations with constant coefficients}.
\newblock {\em Funct. Anal. Its Appl.}, 5(2):89--101, 1971.

\bibitem{bernstein1972}
J.~Bernstein.
\newblock {The analytic continuation of generalized functions with respect to a
  parameter}.
\newblock {\em Funct. Anal. its Appl.}, 6(4):273--285, 1972.

\bibitem{Besnard2018}
F.~Besnard and N.~Bizi.
\newblock {On the definition of spacetimes in noncommutative geometry: Part I}.
\newblock {\em J. Geom. Phys.}, 123:292--309, 2018.

\bibitem{Bizi2018}
N.~Bizi, C.~Brouder, and F.~Besnard.
\newblock {Space and time dimensions of algebras with application to Lorentzian
  noncommutative geometry and quantum electrodynamics}.
\newblock {\em J. Math. Phys.}, 59(6):062303, 2018.

\bibitem{Bouclet2018}
J.-M. Bouclet and N.~Burq.
\newblock {Sharp resolvent and time decay estimates for dispersive equations on
  asymptotically Euclidean backgrounds}.
\newblock {\em {to {a}ppear in} Duke Math. J., arXiv1810.01711}, 2018.

\bibitem{Bourgain2015}
J.~Bourgain, P.~Shao, C.~D. Sogge, and X.~Yao.
\newblock {On $L^p$-resolvent estimates and the density of eigenvalues for
  compact Riemannian manifolds}.
\newblock {\em Commun. Math. Phys.}, 333(3):1483--1527, 2015.

\bibitem{Braverman2020}
M.~Braverman.
\newblock {An index of strongly Callias operators on Lorentzian manifolds with
  non-compact boundary}.
\newblock {\em Math. Zeitschrift}, 294(1-2):229--250, 2020.

\bibitem{Viet-wf2}
C.~Brouder, N.~V. Dang, and F.~H{\'{e}}lein.
\newblock {Continuity of the fundamental operations on distributions having a
  specified wavefront set (with a counterexample by {S}emyon {A}lesker)}.
\newblock {\em Stud. Math.}, 232:201--226, 2016.

\bibitem{Brunetti2000}
R.~Brunetti and K.~Fredenhagen.
\newblock {Microlocal analysis and interacting Quantum Field Theories:
  renormalization on physical backgrounds}.
\newblock {\em Commun. Math. Phys.}, 208(3):623--661, 2000.

\bibitem{Bytsenko2003}
A.~A. Bytsenko, G.~Cognola, V.~Moretti, S.~Zerbini, and E.~Elizalde.
\newblock {\em {Analytic Aspects of Quantum Fields}}.
\newblock World Scientific Publishing, Singapore, 2003.

\bibitem{Chamseddine1997}
A.~H. Chamseddine and A.~Connes.
\newblock {The spectral action principle}.
\newblock {\em Commun. Math. Phys.}, 1997.

\bibitem{Chamseddine2007}
A.~H. Chamseddine, A.~Connes, and M.~Marcolli.
\newblock {Gravity and the standard model with neutrino mixing}.
\newblock {\em Adv. Theor. Math. Phys.}, 11(6):991--1089, 2007.

\bibitem{Chernoff1973}
P.~R. Chernoff.
\newblock {Essential self-adjointness of powers of generators of hyperbolic
  equations}.
\newblock {\em J. Funct. Anal.}, 12(4):401--414, 1973.

\bibitem{cdv}
Y.~{Colin de Verdi{\`{e}}re} and C.~{Le Bihan}.
\newblock {On essential-selfadjointness of differential operators on closed
  manifolds}.
\newblock {\em arXiv:2004.06937}, 2020.

\bibitem{Connes1996}
A.~Connes.
\newblock {Gravity coupled with matter and the foundation of non-commutative
  geometry}.
\newblock {\em Commun. Math. Phys.}, 182(1):155--176, 1996.

\bibitem{Connes2008a}
A.~Connes and M.~Marcolli.
\newblock {\em {Noncommutative Geometry, Quantum Fields and Motives}}.
\newblock American Mathematical Society, Providence, RI, 2008.

\bibitem{Connes2008}
A.~Connes and M.~Marcolli.
\newblock {\em {Noncommutative Geometry, Quantum Fields and Motives}}.
\newblock 2008.

\bibitem{Connes1995}
A.~Connes and H.~Moscovici.
\newblock {The local index formula in noncommutative geometry}.
\newblock {\em Geom. Funct. Anal.}, 5(2):174--243, 1995.

\bibitem{Connes2014}
A.~Connes and H.~Moscovici.
\newblock {Modular curvature for noncommutative two-tori}.
\newblock {\em J. Am. Math. Soc.}, 27(3):639--684, 2014.

\bibitem{Coriasco2003}
S.~Coriasco, E.~Schrohe, and J.~Seiler.
\newblock {Bounded imaginary powers of differential operators on manifolds with
  conical singularities}.
\newblock {\em Math. Zeitschrift}, 244(2):235--269, 2003.

\bibitem{DAndrea2016}
F.~D'Andrea, M.~A. Kurkov, and F.~Lizzi.
\newblock {Wick rotation and fermion doubling in noncommutative geometry}.
\newblock {\em Phys. Rev. D}, 94(2):025030, 2016.

\bibitem{Dang2021}
N.~V. Dang and M.~Wrochna.
\newblock {Dynamical residues of Lorentzian spectral zeta functions}.
\newblock {\em arXiv:2108.07529}, 2021.

\bibitem{Derezinski2020}
J.~Derezi{\'{n}}ski, A.~Latosi{\'{n}}ski, and D.~Siemssen.
\newblock {Pseudodifferential Weyl calculus on (pseudo-)Riemannian manifolds}.
\newblock {\em Ann. Henri Poincar{\'{e}}}, 21(5):1595--1635, 2020.

\bibitem{derezinski}
J.~Derezi{\'{n}}ski and D.~Siemssen.
\newblock {Feynman propagators on static spacetimes}.
\newblock {\em Rev. Math. Phys.}, 2018.

\bibitem{Derezinski2019}
J.~Derezi{\'{n}}ski and D.~Siemssen.
\newblock {An evolution equation approach to linear Quantum Field Theory}.
\newblock {\em arXiv:1912.10692}, 2019.

\bibitem{Devastato2018}
A.~Devastato, S.~Farnsworth, F.~Lizzi, and P.~Martinetti.
\newblock {Lorentz signature and twisted spectral triples}.
\newblock {\em J. High Energy Phys.}, 2018(3):89, 2018.

\bibitem{DeWitt1975}
B.~S. DeWitt.
\newblock {Quantum field theory in curved spacetime}.
\newblock {\em Phys. Rep.}, 19(6):295--357, 1975.

\bibitem{Ferreira2014}
D.~{Dos Santos Ferreira}, C.~E. Kenig, and M.~Salo.
\newblock {On $L^p$ resolvent estimates for Laplace–Beltrami operators on
  compact manifolds}.
\newblock {\em Forum Math.}, 26(3), 2014.

\bibitem{DH}
J.~J. Duistermaat and L.~H{\"{o}}rmander.
\newblock {Fourier integral operators. II}.
\newblock {\em Acta Math.}, 128:183--269, 1972.

\bibitem{Dyatlov2016}
S.~Dyatlov and M.~Zworski.
\newblock {Dynamical zeta functions for Anosov flows via microlocal analysis}.
\newblock {\em Ann. Sci. l'{\'{E}}cole Norm. Sup{\'{e}}rieure}, 49(3):543--577,
  2016.

\bibitem{DZ}
S.~Dyatlov and M.~Zworski.
\newblock {\em {Mathematical theory of scattering resonances}}.
\newblock American Mathematical Society, Providence, RI, 2019.

\bibitem{Dyatlov2019}
S.~Dyatlov and M.~Zworski.
\newblock {Microlocal analysis of forced waves}.
\newblock {\em Pure Appl. Anal.}, 1(3):359--384, 2019.

\bibitem{Eckstein2018}
M.~Eckstein and B.~Iochum.
\newblock {\em {Spectral Action in Noncommutative Geometry}}, volume~27 of {\em
  SpringerBriefs in Mathematical Physics}.
\newblock Springer International Publishing, 2018.

\bibitem{Enciso2017}
A.~Enciso, M.~d.~M. Gonz{\'{a}}lez, and B.~Vergara.
\newblock {Fractional powers of the wave operator via Dirichlet-to-Neumann maps
  in anti-de Sitter spaces}.
\newblock {\em J. Funct. Anal.}, 2016.

\bibitem{Franco2013}
N.~Franco and M.~Eckstein.
\newblock {An algebraic formulation of causality for noncommutative geometry}.
\newblock {\em Class. Quantum Gravity}, 30(13):135007, 2013.

\bibitem{Fulling1989}
S.~A. Fulling.
\newblock {\em {Aspects of Quantum Field Theory in Curved Space-Time}}.
\newblock Cambridge University Press, Cambridge, 1989.

\bibitem{gannotwrochna}
O.~Gannot and M.~Wrochna.
\newblock {Propagation of singularities on AdS spacetimes for general boundary
  conditions and the holographic Hadamard condition}.
\newblock {\em J. Inst. Math. Jussieu}, 2020.

\bibitem{Gelfand-ShilovI}
I.~M. Gel'fand and G.~E. Shilov.
\newblock {\em {Generalized Functions: Volume I -- Properties and Operations}}.
\newblock Academic Press, New York, 1964.

\bibitem{GHV}
J.~Gell-Redman, N.~Haber, and A.~Vasy.
\newblock {The Feynman propagator on perturbations of Minkowski space}.
\newblock {\em Commun. Math. Phys.}, 342(1):333--384, 2016.

\bibitem{hassell}
J.~Gell-Redman, A.~Hassell, J.~Shapiro, and J.~Zhang.
\newblock {Existence and asymptotics of nonlinear Helmholtz eigenfunctions}.
\newblock {\em arxiv.org/abs/1908.04890}, 2019.

\bibitem{GOW}
C.~G{\'{e}}rard, O.~Oulghazi, and M.~Wrochna.
\newblock {Hadamard states for the Klein–Gordon equation on Lorentzian
  manifolds of bounded geometry}.
\newblock {\em Commun. Math. Phys.}, 2017.

\bibitem{Gerard2020}
C.~G{\'{e}}rard and M.~Wrochna.
\newblock {The Feynman problem for the Klein--Gordon equation}.
\newblock {\em S{\'{e}}minaire Laurent Schwartz --- EDP Appl.},
  (2019-2020):Expos{\'{e}} no IV.

\bibitem{GWfeynman}
C.~G{\'{e}}rard and M.~Wrochna.
\newblock {The massive Feynman propagator on asymptotically Minkowski
  spacetimes}.
\newblock {\em Am. J. Math.}, 141(6):1501--1546, 2019.

\bibitem{Gerard2019b}
C.~G{\'{e}}rard and M.~Wrochna.
\newblock {The massive Feynman propagator on asymptotically Minkowski
  spacetimes II}.
\newblock {\em Int. Math. Res. Not.}, 2019.

\bibitem{geroch}
R.~Geroch.
\newblock {Domain of dependence}.
\newblock {\em J. Math. Phys.}, 11(2):437--449, 1970.

\bibitem{gilkey}
P.~Gilkey.
\newblock {Invariance theory: the heat equation and the Atiyah-Singer index
  theorem}.
\newblock CRC Press, Boca Raton, FL, 1995.

\bibitem{Greenleaf1990}
A.~Greenleaf and G.~Uhlmann.
\newblock {Estimates for singular Radon transforms and pseudodifferential
  operators with singular symbols}.
\newblock {\em J. Funct. Anal.}, 89(1):202--232, 1990.

\bibitem{Grubb1986}
G.~Grubb.
\newblock {\em {Functional Calculus of Pseudo-Differential Boundary Problems}}.
\newblock Birkh{\"{a}}user Boston, Boston, MA, 1986.

\bibitem{Guillemin1985}
V.~Guillemin.
\newblock {A new proof of Weyl's formula on the asymptotic distribution of
  eigenvalues}.
\newblock {\em Adv. Math. (N. Y).}, 55(2):131--160, 1985.

\bibitem{sectorial}
M.~Haase.
\newblock {\em {The Functional Calculus for Sectorial Operators}}.
\newblock Birkh{\"{a}}user Basel, 2006.

\bibitem{Hafner2020}
D.~H{\"{a}}fner, P.~Hintz, and A.~Vasy.
\newblock {Linear stability of slowly rotating Kerr black holes}.
\newblock {\em Invent. Math.}, 2020.

\bibitem{Hawking1977}
S.~W. Hawking.
\newblock {Zeta function regularization of path integrals in curved spacetime}.
\newblock {\em Commun. Math. Phys.}, 1977.

\bibitem{Hintz2017}
P.~Hintz.
\newblock {Resonance expansions for tensor-valued waves on asymptotically
  Kerr–de Sitter spaces}.
\newblock {\em J. Spectr. Theory}, 7(2):519--557, 2017.

\bibitem{hintz20}
P.~Hintz.
\newblock {Resolvents and complex powers of semiclassical cone operators}.
\newblock {\em arXiv:2010.01593}, 2020.

\bibitem{Hintz2015}
P.~Hintz and A.~Vasy.
\newblock {Semilinear wave equations on asymptotically de Sitter, Kerr–de
  Sitter and Minkowski spacetimes}.
\newblock {\em Anal. PDE}, 8(8):1807--1890, 2015.

\bibitem{Hintz2016}
P.~Hintz and A.~Vasy.
\newblock {Global analysis of quasilinear wave equations on asymptotically
  Kerr-de Sitter spaces}.
\newblock {\em Int. Math. Res. Not.}, (17):5355--5426, 2016.

\bibitem{Hintz2018a}
P.~Hintz and A.~Vasy.
\newblock {The global non-linear stability of the Kerr–de Sitter family of
  black holes}.
\newblock {\em Acta Math.}, 220(1):1--206, 2018.

\bibitem{Hollands2001}
S.~Hollands and R.~M. Wald.
\newblock {Local Wick polynomials and time ordered products of quantum fields
  in curved spacetime}.
\newblock {\em Commun. Math. Phys.}, 223(2):289--326, 2001.

\bibitem{Hormander1971}
L.~H{\"{o}}rmander.
\newblock {\em {On the Existence and the Regularity of Solutions of Linear
  Pseudo-differential Equations}}.
\newblock L'Enseignement Math{\'{e}}matique, Gen{\`{e}}ve, 1971.

\bibitem{H}
L.~H{\"{o}}rmander.
\newblock {\em {The Analysis of Linear Partial Differential Operators I.
  Distribution Theory and Fourier Analysis}}.
\newblock Springer Verlag, Berlin, second edition, 1990.

\bibitem{Hormander-97}
L.~H{\"{o}}rmander.
\newblock {\em {Lectures on Nonlinear Hyperbolic Differential Equations}}.
\newblock Springer Verlag, Berlin, 1997.

\bibitem{HormanderIII}
L.~H{\"{o}}rmander.
\newblock {\em {The Analysis of Linear Partial Differential Operators III.
  Pseudo-Differential Operators}}.
\newblock Classics in Mathematics. Springer Berlin Heidelberg, Berlin,
  Heidelberg, 2007.

\bibitem{joshi}
M.~S. Joshi.
\newblock {Complex powers of the wave operator}.
\newblock {\em Port. Math.}, 54(3):345--362, 1997.

\bibitem{kalauwalze}
W.~Kalau and M.~Walze.
\newblock {Gravity, non-commutative geometry and the Wodzicki residue}.
\newblock {\em J. Geom. Phys.}, 16(4):327--344, 1995.

\bibitem{kaminski}
W.~Kami{\'{n}}ski.
\newblock {Non self-adjointness of the Klein-Gordon operator on globally
  hyperbolic and geodesically complete manifold. An example}.
\newblock {\em arXiv:1904.03691}, 2019.

\bibitem{kastler}
D.~Kastler.
\newblock {The Dirac operator and gravitation}.
\newblock {\em Commun. Math. Phys.}, 166(3):633--643, 1995.

\bibitem{Kay1991}
B.~S. Kay and R.~M. Wald.
\newblock {Theorems on the uniqueness and thermal properties of stationary,
  nonsingular, quasifree states on spacetimes with a bifurcate Killing
  horizon}.
\newblock {\em Phys. Rep.}, 207(2):49--136, 1991.

\bibitem{Lesch}
M.~Lesch.
\newblock {On the noncommutative residue for pseudodifferential operators with
  log-polyhomogeneous symbols}.
\newblock {\em Ann. Glob. Anal. Geom.}, 17:151--187, 1999.

\bibitem{Lesch2000}
M.~Lesch and M.~J. Pflaum.
\newblock {Traces on algebras of parameter dependent pseudodifferential
  operators and the eta–invariant}.
\newblock {\em Trans. Am. Math. Soc.}, 352(11):4911--4936, 2000.

\bibitem{Lewandowski2020}
M.~Lewandowski.
\newblock {Hadamard states for bosonic quantum field theory on globally
  hyperbolic spacetimes}.
\newblock {\em arXiv:2008.13156}, 2020.

\bibitem{Loya2001}
P.~Loya.
\newblock {Tempered operators and the heat kernel and complex powers of
  elliptic pseudodifferential operators}.
\newblock {\em Commun. Partial Differ. Equations}, 26(7-8):1253--1321, 2001.

\bibitem{Loya2003}
P.~Loya.
\newblock {Complex powers of differential operators on manifolds with conical
  singularities}.
\newblock {\em J. d'Analyse Math{\'{e}}matique}, 89(1):31--56, 2003.

\bibitem{Maeda2005}
Y.~Maeda, D.~Manchon, and S.~Paycha.
\newblock {Stokes' formulae on classical symbol valued forms and applications}.
\newblock {\em arXiv:math/0510454}, 2005.

\bibitem{martinetti}
P.~Martinetti and D.~Singh.
\newblock {Lorentzian fermionic action by twisting euclidean spectral triples}.
\newblock {\em arXiv:1907.02485}, 2019.

\bibitem{Melrose1993}
R.~Melrose.
\newblock {\em {The Atiyah--Patodi--Singer Index Theorem}}.
\newblock CRC Press, Boca Raton, FL, 1993.

\bibitem{melrosered}
R.~Melrose.
\newblock {Spectral and scattering theory for the Laplacian on asymptotically
  Euclidian spaces}.
\newblock In {\em Spectr. Scatt. Theory}. Dekker, New York, 1994.

\bibitem{Melrose1996}
R.~Melrose and M.~Zworski.
\newblock {Scattering metrics and geodesic flow at infinity}.
\newblock {\em Invent. Math.}, 124(1-3):389--436, 1996.

\bibitem{meyer1981regularite}
Y.~Meyer.
\newblock {R{\'{e}}gularit{\'{e}} des solutions des {\'{e}}quations aux
  d{\'{e}}riv{\'{e}}es partielles non lin{\'{e}}aires}.
\newblock In {\em S{\'{e}}minaire Bourbaki vol. 1979/80 Expo. 543--560}, pages
  293--302. Springer, 1981.

\bibitem{Meyer}
Y.~Meyer.
\newblock {\em {Wavelets, Vibrations and Scalings}}.
\newblock American Mathematical Society, Providence, RI, 1998.

\bibitem{Minakshisundaram1949}
S.~Minakshisundaram and {\AA}.~Pleijel.
\newblock {Some properties of the eigenfunctions of the Laplace-operator on
  Riemannian manifolds}.
\newblock {\em Can. J. Math.}, 1(3):242--256, 1949.

\bibitem{Moretti1999}
V.~Moretti.
\newblock {Local $\zeta$-function techniques vs. point-splitting procedure: A
  few rigorous results}.
\newblock {\em Commun. Math. Phys.}, 201(2):327--363, 1999.

\bibitem{Moretti2003}
V.~Moretti.
\newblock {Aspects of noncommutative Lorentzian geometry for globally
  hyperbolic spacetimes}.
\newblock {\em Rev. Math. Phys.}, 15(10):1171--1217, 2003.

\bibitem{nakamurataira}
S.~Nakamura and K.~Taira.
\newblock {Essential self-adjointness of real principal type operators}.
\newblock {\em Ann. Henri Lebesgue}, 4:1035--1059, 2021.

\bibitem{Parenti1972}
C.~Parenti.
\newblock {Operatori pseudo-differenziali in $\mathbb{R}^n$ e applicazioni}.
\newblock {\em Ann. di Mat. Pura ed Appl. Ser. 4}, 1972.

\bibitem{Paschke2004}
M.~Paschke and R.~Verch.
\newblock {Local covariant quantum field theory over spectral geometries}.
\newblock {\em Class. Quantum Gravity}, 21(23):5299--5316, 2004.

\bibitem{paycha2}
S.~Paycha.
\newblock {The noncommutative residue and canonical trace in the light of
  Stokes' and continuity properties}.
\newblock {\em arXiv:0706.2552}, 2007.

\bibitem{paycha}
S.~Paycha.
\newblock {\em {Regularised integrals, sums and traces; analytic aspects}}.
\newblock American Mathematical Society, 2012.

\bibitem{Ponge2006}
R.~Ponge.
\newblock {Spectral Asymmetry, Zeta Functions and the Noncommutative Residue}.
\newblock {\em Int. J. Math.}, 17(09):1065--1090, 2006.

\bibitem{Ponge2007}
R.~Ponge.
\newblock {Noncommutative residue for Heisenberg manifolds. Applications in CR
  and contact geometry}.
\newblock {\em J. Funct. Anal.}, 252(2):399--463, nov 2007.

\bibitem{Ponge2008}
R.~Ponge.
\newblock {Noncommutative Geometry and Lower Dimensional Volumes in Riemannian
  Geometry}.
\newblock {\em Lett. Math. Phys.}, 83(1):19--32, 2008.

\bibitem{Quillen1985}
D.~Quillen.
\newblock {Determinants of Cauchy-Riemann operators over a Riemann surface}.
\newblock {\em Funct. Anal. Its Appl.}, 19(1):31--34, 1985.

\bibitem{Radzikowski1996}
M.~J. Radzikowski.
\newblock {Micro-local approach to the Hadamard condition in quantum field
  theory on curved space-time}.
\newblock {\em Commun. Math. Phys.}, 179(3):529--553, 1996.

\bibitem{Ray1971}
D.~Ray and I.~Singer.
\newblock {R-Torsion and the Laplacian on Riemannian manifolds}.
\newblock {\em Adv. Math.}, 7(2):145--210, 1971.

\bibitem{Rempel1983}
S.~Rempel and B.~W. Schulze.
\newblock {Complex powers for pseudo-differential boundary problems, I}.
\newblock {\em Math. Nachrichten}, 111(1):41--109, 1983.

\bibitem{sabbah}
C.~Sabbah.
\newblock {Aspects alg{\'{e}}briques de la division des distributions.
  \texttt{www.math.polytechnique.fr/xups/xups03-03.pdf}}.

\bibitem{Schrohe1986}
E.~Schrohe.
\newblock {Complex powers of elliptic pseudodifferential operators}.
\newblock {\em Integr. Equations Oper. Theory}, 9(3):337--354, 1986.

\bibitem{Schrohe1988}
E.~Schrohe.
\newblock {Complex powers on noncompact manifolds and manifolds with
  singularities}.
\newblock {\em Math. Ann.}, 281(3):393--409, 1988.

\bibitem{seeley}
R.~T. Seeley.
\newblock {Complex powers of an elliptic operator}.
\newblock {\em Proc. Symp. Pure Math.}, 10:288--307, 1967.

\bibitem{Shubin1978}
M.~A. Shubin.
\newblock {Pseudodifferential operators in $\mathbb{R}^n$}.
\newblock {\em Dokl. Akad. Nauk SSSR}, 196:316--319, 1978.

\bibitem{shubin}
M.~A. Shubin.
\newblock {\em {Pseudodifferential Operators and Spectral Theory}}.
\newblock 2001.

\bibitem{Sogge1988}
C.~D. Sogge.
\newblock {Concerning the $L^p$ norm of spectral clusters for second-order
  elliptic operators on compact manifolds}.
\newblock {\em J. Funct. Anal.}, 77(1):123--138, 1988.

\bibitem{soggeHangzhou}
C.~D. Sogge.
\newblock {\em {Hangzhou Lectures on Eigenfunctions of the Laplacian}}.
\newblock Princeton University Press, Princeton, 2014.

\bibitem{sogge2017fourier}
C.~D. Sogge.
\newblock {\em {Fourier integrals in classical analysis}}, volume 210.
\newblock Cambridge University Press, 2017.

\bibitem{Strichartz1983}
R.~S. Strichartz.
\newblock {Analysis of the Laplacian on the complete Riemannian manifold}.
\newblock {\em J. Funct. Anal.}, 52(1):48--79, 1983.

\bibitem{Strohmaier2006}
A.~Strohmaier.
\newblock {On noncommutative and pseudo-Riemannian geometry}.
\newblock {\em J. Geom. Phys.}, 56(2):175--195, 2006.

\bibitem{Strohmaier2020b}
A.~Strohmaier and S.~Zelditch.
\newblock {A Gutzwiller trace formula for stationary space-times}.
\newblock {\em Adv. Math.}, 376, 2020.

\bibitem{Strohmaier2020}
A.~Strohmaier and S.~Zelditch.
\newblock {Semi-classical mass asymptotics on stationary spacetimes}.
\newblock {\em Indag. Math.}, 2020.

\bibitem{Strohmaier2020a}
A.~Strohmaier and S.~Zelditch.
\newblock {Spectral asymptotics on stationary space-times}.
\newblock {\em Rev. Math. Phys.}, page X206000, 2020.

\bibitem{taira}
K.~Taira.
\newblock {Equivalence of classical and quantum completeness for real principal
  type operators on the circle}.
\newblock {\em arXiv:2004.07547}, 2020.

\bibitem{Taira2020}
K.~Taira.
\newblock {Strichartz estimates for non-degenerate Schr{\"{o}}dinger
  equations}.
\newblock {\em Math. Nachrichten}, 293(4):774--793, 2020.

\bibitem{Taira2020a}
K.~Taira.
\newblock {Limiting absorption principle and equivalence of Feynman propagators
  on asymptotically Minkowski spacetimes}.
\newblock {\em Commun. Math. Phys.}, 388(1):625--655, 2021.

\bibitem{taylorpartial3}
M.~Taylor.
\newblock {\em {Partial differential equations III: Nonlinear equations}},
  volume 117.
\newblock Springer Science \& Business Media, 2013.

\bibitem{Uhlmann2016}
G.~Uhlmann and A.~Vasy.
\newblock {The inverse problem for the local geodesic ray transform}.
\newblock {\em Invent. Math.}, 205(1):83--120, 2016.

\bibitem{VandenDungen2016}
K.~van~den Dungen.
\newblock {Krein spectral triples and the fermionic action}.
\newblock {\em Math. Physics, Anal. Geom.}, 19(1):4, 2016.

\bibitem{VanDenDungen2018}
K.~van~den Dungen.
\newblock {Families of spectral triples and foliations of space(time)}.
\newblock {\em J. Math. Phys.}, 59(6):063507, 2018.

\bibitem{suijlekom}
W.~D. van Suijlekom.
\newblock {\em {Noncommutative Geometry and Particle Physics}}.
\newblock Mathematical Physics Studies. Springer Netherlands, Dordrecht, 2015.

\bibitem{Varilly2009}
J.~V{\'{a}}rilly.
\newblock {\em {An Introduction to Noncommutative Geometry}}.
\newblock European Mathematical Society, Z{\"{u}}rich, 2006.

\bibitem{Vasy2013}
A.~Vasy.
\newblock {Microlocal analysis of asymptotically hyperbolic and Kerr-de Sitter
  spaces (with an appendix by Semyon Dyatlov)}.
\newblock {\em Invent. Math.}, 194(2):381--513, 2013.

\bibitem{vasygrenoble}
A.~Vasy.
\newblock {A Minicourse on Microlocal Analysis for Wave Propagation}.
\newblock In {\em Asymptot. Anal. Gen. Relativ.}, pages 219--374. Cambridge
  University Press, 2017.

\bibitem{Vasy2017b}
A.~Vasy.
\newblock {On the positivity of propagator differences}.
\newblock {\em Ann. Henri Poincar{\'{e}}}, 18(3):983--1007, 2017.

\bibitem{vasyessential}
A.~Vasy.
\newblock {Essential self-adjointness of the wave operator and the limiting
  absorption principle on Lorentzian scattering spaces}.
\newblock {\em J. Spectr. Theory}, 10(2):439--461, 2020.

\bibitem{Vasy2019}
A.~Vasy.
\newblock {Limiting absorption principle on Riemannian scattering
  (asymptotically conic) spaces, a Lagrangian approach}.
\newblock {\em Commun. Partial Differ. Equations}, pages 1--76, 2020.

\bibitem{vasywrochna}
A.~Vasy and M.~Wrochna.
\newblock {Quantum fields from global propagators on asymptotically Minkowski
  and extended de Sitter spacetimes}.
\newblock {\em Ann. Henri Poincar{\'{e}}}, 2018.

\bibitem{wodzicki}
M.~Wodzicki.
\newblock {Local invariants of spectral asymmetry}.
\newblock {\em Invent. Math.}, 75(1):143--177, 1984.

\bibitem{holographic}
M.~Wrochna.
\newblock {The holographic Hadamard condition on asymptotically anti-de Sitter
  spacetimes}.
\newblock {\em Lett. Math. Phys.}, 107(12):2291--2331, 2017.

\bibitem{StevenZelditch2012}
S.~Zelditch.
\newblock {Pluri-potential theory on Grauert tubes of real analytic Riemannian
  manifolds, I}.
\newblock In {\em Spectr. Geom.}, number~3, pages 299--339. American
  Mathematical Society, 2012.

\bibitem{Zelditch2017}
S.~Zelditch.
\newblock {Eigenfunctions of the Laplacian of Riemannian manifolds.
  \texttt{www.math.northwestern.edu/}
  \texttt{$\sim$zelditch/Eigenfunction.pdf}.}
\newblock 2017.

\end{thebibliography}

\end{document}